\documentclass[12pt]{article}

\usepackage{amsmath} 
\usepackage{amsthm} 
\usepackage{amsfonts} 
\usepackage{amssymb} 
\usepackage{stmaryrd} 
\usepackage{mathtools} 

\usepackage{bm} 
\usepackage{bbm} 

\usepackage{array} 
\usepackage[inline]{enumitem} 

\usepackage{tikz} 
\usetikzlibrary{cd} 
\usetikzlibrary{calc} 

\usepackage[in]{fullpage} 

\usepackage{subcaption} 

\usepackage{pdftexcmds} 
\makeatletter
\newcommand\compareStringsDo[5]{%
  \lowercase{\ifcase\pdf@strcmp{#1}{#2}}%
    #4\or
    #5\else
    #3\fi
}
\makeatother

\usepackage[%
  pdftex,
  plainpages=false,
  pdfpagelabels,
  colorlinks,
  citecolor=black,
  linkcolor=black,
  urlcolor=black,
  filecolor=black,
  bookmarksopen=false
]{hyperref} 
\usepackage[all]{hypcap} 

\newtheoremstyle{thmstyle}
  {\medskipamount}
  {\smallskipamount}
  {\slshape}
  {0pt}
  {\bfseries}
  {.}
  { }
  {\thmname{#1}\thmnumber{ #2}{\normalfont\thmnote{ (#3)}}}

\newtheoremstyle{plainstyle}
  {\medskipamount}
  {\smallskipamount}
  {\rmfamily}
  {0pt}
  {\bfseries}
  {.}
  { }
  {\thmname{#1}\thmnumber{ #2}{\normalfont\thmnote{ (#3)}}}

\theoremstyle{thmstyle}
\newtheorem{theorem}{Theorem}[section]
\newtheorem{lemma}[theorem]{Lemma}
\newtheorem{corollary}[theorem]{Corollary}
\newtheorem{proposition}[theorem]{Proposition}

\theoremstyle{plainstyle}
\newtheorem{definition}[theorem]{Definition}
\newtheorem{convention}[theorem]{Convention}
\newtheorem{remark}[theorem]{Remark}
\newtheorem{discussion}[theorem]{Discussion}

\newtheorem{example}[theorem]{Example}

\newenvironment{proofof}[1]{\begin{proof}[Proof of #1.]}{\end{proof}}

\setlist[enumerate]{label={\roman*.}, ref={(\roman*)}}

\newcommand{\rn}{\bm}
\newcommand{\df}{\stackrel{\text{def}}{=}}
\newcommand{\place}{\mathord{-}}
\def\symdiff{\mathbin{\triangle}}
\newcommand{\comp}{\mathbin{\circ}}
\newcommand{\rest}{\mathord{\vert}}

\newcommand{\cat}{\textsc}
\newcommand{\disjcup}{\mathbin{\stackrel\cdot\cup}}

\newcommand{\tot}{\leftrightarrow}

\DeclareMathOperator{\Hom}{Hom}
\DeclareMathOperator{\Aut}{Aut}
\DeclareMathOperator{\im}{im}
\DeclareMathOperator{\id}{id}

\DeclareMathOperator{\VC}{VC}

\DeclareMathOperator{\Th}{Th}
\DeclareMathOperator{\QR}{QR}
\newcommand{\HomT}[1]{\Hom^+(\cA[#1],\RR)}
\def\Tind{T_{\operatorname{ind}}}
\def\tind{t_{\operatorname{ind}}}
\def\Dopen{D_{\operatorname{open}}}
\DeclareMathOperator{\Forb}{Forb}

\newcommand{\NN}{\mathbb{N}}

\newcommand{\RR}{\mathbb{R}}

\newcommand{\One}{\mathbbm{1}}

\newcommand{\cA}{\mathcal{A}}
\newcommand{\cB}{\mathcal{B}}
\newcommand{\cC}{\mathcal{C}}
\newcommand{\cD}{\mathcal{D}}
\newcommand{\cE}{\mathcal{E}}
\newcommand{\cF}{\mathcal{F}}
\newcommand{\cH}{\mathcal{H}}
\newcommand{\cK}{\mathcal{K}}
\newcommand{\cL}{\mathcal{L}}
\newcommand{\cM}{\mathcal{M}}
\newcommand{\cN}{\mathcal{N}}
\newcommand{\cP}{\mathcal{P}}
\newcommand{\cS}{\mathcal{S}}

\def\prop{\texttt}
\def\UCouple{\prop{UCouple}}
\def\UInduce{\prop{UInduce}}
\def\AEHP{\prop{AEHP}}
\def\WR{\prop{WR}}
\def\EHP{\prop{EHP}}

\newcommand{\TGraph}{T_{\operatorname{Graph}}}
\newcommand{\TkHypergraph}[1][k]{T_{#1\operatorname{-Hypergraph}}}
\newcommand{\TLinOrder}{T_{\operatorname{LinOrder}}}
\newcommand{\TPerm}{T_{\operatorname{Perm}}}
\newcommand{\TPermGraph}{T_{\operatorname{PermGraph}}}
\newcommand{\TBipartite}{T_{\operatorname{Bipartite}}}
\newcommand{\TPerfect}{T_{\operatorname{Perfect}}}
\newcommand{\TTri}{T_{\operatorname{Tri}}}

\def\Erdos{Erd\H{o}s}
\def\Caratheodory{Carath\'{e}odory}

\title{Weak randomness in graphons and theons}
\author{Leonardo N.~Coregliano\thanks{Institute for Advanced Study, {\tt lenacore@ias.edu}. This
    material is based upon work supported by a grant from the Institute for Advanced Study School of
    Mathematics.} \and%
  Maryanthe Malliaris\thanks{University of Chicago, {\tt mem@math.uchicago.edu}. Research partially
    supported by NSF-BSF 2051825.}%
}
\date{\today}

\begin{document}
\maketitle

\begin{abstract}
  Call a hereditary family $\cF$ of graphs strongly persistent if there exists a graphon $W$ such
  that in all subgraphons $W'$ of $W$, $\cF$ is precisely the class of finite graphs that have
  positive density in $W'$. Our first result is a complete characterization of the hereditary
  families of graphs that are strongly persistent as precisely those that are closed under
  substitutions.

  We call graphons with the self-similarity property above weakly random. A hereditary family $\cF$
  is said to have the weakly random \Erdos--Hajnal property ($\WR$) if every graphon that is a limit
  of graphs in $\cF$ has a weakly random subgraphon. Among families of graphs that are closed under
  substitutions, we completely characterize the families that belong to $\WR$ as those with ``few''
  prime graphs.

  We also extend some of the results above to structures in finite relational languages by using the
  theory of theons.

\end{abstract}

\section{Introduction}
\label{sec:intro}

The theory of graph quasirandomness implies that quasirandom graphons are the only graphons $W$ with
the self-similarity property that densities of finite graphs are invariant across subgraphons of $W$
(see~\cite{Tho85,CGW89} for graph quasirandomness and~\cite{Lov12} for graphons). An interesting weakening
of this property, which we will motivate further below, is to require only that the family $\cF$ of
finite graphs that have positive density is invariant across subgraphons of $W$. We call graphons
with this property \emph{weakly random}. It is natural to ask which families $\cF$ can be realized
in this way in some weakly random $W$. Since all constant graphons are quasirandom, thus also weakly
random, three such families are the cliques, the anti-cliques and the family of all finite
graphs. However, there are other families that can be realized in this way such as the family
$\cF_{C_4}$ of all cographs, that is, graphs such that every induced subgraph of size at least $2$
can be partitioned into two non-trivial parts that are either complete to each other, or empty to each
other. Alternatively, $\cF_{C_4}$ is precisely the family of finite graphs that are induced
subgraphs of some recursive blow-up of the $4$-cycle. Strong persistence of $\cF_{C_4}$ is seen
since the limit $W_{C_4}$ of the balanced recursive blow-ups of the $4$-cycle is weakly random and
realizes the family $\cF_{C_4}$.

The work of this paper is to show that this notion of weak randomness supports a rich structure
theory and provides an illuminating way of studying hereditary classes of graphs based on properties
of their limit objects. Before stating our main results, let us further motivate why the study of
weak randomness is both natural and tractable, which begins by asking what is special about large
cliques and anti-cliques.

Recall that the \Erdos--Hajnal Conjecture~\cite{EH89} says that for any proper hereditary class of
graphs, there exists a constant $c > 0$ such that any graph of size $n$ in this class either has a
clique or an anti-clique of size $n^c$; we will refer to this property of a hereditary class as the
\Erdos--Hajnal property and abbreviate it as $\EHP$ (see also~\cite{Chu14} for a
survey). In~\cite{CM22}, we studied a natural variant of this question in the presence of
convergence, called the approximate \Erdos--Hajnal property ($\AEHP$), in which we allow for a
negligible amount of non-edges in the almost clique or a negligible amount of edges in the almost
anti-clique, but require it to be linear-sized. The framework of $\AEHP$ naturally lends itself to
analysis via limit theory, i.e., graphons~\cite{LS06} in the graph case, or more generally, flag
algebras~\cite{Raz07} and theons~\cite{CR20a} in the case of universal theories in finite relational
languages.

The aforementioned family $\cF_{C_4}$ of cographs plays a key role in some of the classical results
on the usual \Erdos--Hajnal Conjecture: namely, a consequence of~\cite[Theorem~1.1]{APS01} is that
any hereditary class that does not contain $\cF_{C_4}$ has $\EHP$. However, this is not a
characterization of $\EHP$ as several classes that contain $\cF_{C_4}$ still have $\EHP$; easy
examples include perfect graphs and $P_3$-free graphs (i.e., disjoint unions of cliques) and hard
examples include bull-free graphs~\cite{Chu08} and $C_5$-free graphs~\cite{CSSS23}.

On the other hand, surprisingly, hereditary classes of graphs with $\AEHP$ can be characterized as
precisely those that avoid containing $\cF_{C_4}$, see~\cite[Theorem~8.10]{CM22}. In what follows,
it will be more convenient to think about hereditary classes of graphs as the models of a particular
universal first-order theory $T$ of graphs, so a graphon of $T$ is simply a limit of finite models
of $T$. This shift in language supports the model theoretic perspective of studying the theory $T$
(i.e., a hereditary class of graphs) by studying the variation in the class of its infinite models
(i.e., its graphons). In this language, a universal theory $T$ of graphs has $\AEHP$ if every
graphon of $T$ has a (large) trivial subgraphon, i.e., an almost clique or an almost anti-clique,
see~\cite[\S7]{CM22} and Definition~\ref{def:AEHPgraphs}.

In the proof of the negative side of the characterization of $\AEHP$ for graphs, if all cographs are
models of $T$, then the limit $W_{C_4}$ is a graphon of $T$. Looking through the lens of weak
randomness, it is clear that $W_{C_4}$ does not contain trivial subgraphons since both the edge and
the non-edge must persistently have positive density in all subgraphons of $W_{C_4}$. Part of the
characterization of $\AEHP$ involved showing that persistence of the edge and non-edge implies
persistence of every graph in $\cF_{C_4}$. Thus, we are led to ask which families arise as
\emph{persistent classes} of graphons, i.e., families $\cF$ of graphs that are precisely those that
have positive density in all subgraphons of a given graphon $W$. A related notion is that of a
\emph{strongly persistent class}, in which the graphon is further required to be weakly random. A
priori these notions are different since a non-weakly random graphon can have finite graphs with
positive density in only some of its subgraphons.

The first theorem of the present paper is to show the equivalence of strong persistence and
persistence and to characterize such families as precisely those that are closed under substructures
and substitutions (see Definition~\ref{def:subst}). This requires both understanding properties of
substitutions and the construction of appropriate weakly random limits. We prove this result first
for graphs (Theorem~\ref{thm:graphpersistence}) and then a suitable generalization of it for
structures in arbitrary finite relational languages (Theorem~\ref{thm:persistence}) after developing
suitable extensions of the relevant concepts. The appearance of substitution in this
characterization, and of the related notion of primality in what follows, is not completely
unexpected as both the \Erdos--Hajnal property and its approximate version behave very well under
substitution (see~\cite{APS01,Chu14} for $\EHP$ and~\cite{CM22} for $\AEHP$).

Since cliques and anti-cliques are weakly random, we can extend the picture of $\AEHP$ by defining
the class $\WR$ as follows: a universal theory of graphs is in $\AEHP$ if all its graphons have
trivial subgraphons and a universal theory of graphs is in $\WR$ if all its graphons have weakly
random subgraphons. It is immediate that $\AEHP\subseteq\WR$, it is less immediate but shown in the
present paper that this containment is proper and that not every universal theory is in
$\WR$. Because of the nature and simplicity of the characterization of $\AEHP$ for graphs cited
above, it becomes plausible that a characterization of the richer $\WR$ class may exist.

In Theorem~\ref{thm:graphs:WR}, we characterize theories of graphs in $\WR$ under the additional
natural assumption of closure under substitution as those that have ``few'' prime graphs in the
sense that there are no infinite antichains of prime graphs in the induced subgraph partial order, a
condition we call \emph{primally almost finite}. In one direction, we build on the analysis of
persistence of Theorem~\ref{thm:graphpersistence} and in the other direction, the technology of
recursive blow-ups plays a key role. Note that without the assumption of closure under
substitutions, it is obvious that $\WR$ is no longer characterized by the primally almost finite
condition as, e.g., the theory of bipartite graphs is in $\WR$ (even in $\AEHP$) but has infinite
antichains of prime graphs.

Many further questions are discussed in the concluding Section~\ref{sec:concl}.

\medskip

Let us point out that although~\cite{CM22} provides a good motivation for the current work, it is
not a pre-requisite for the current paper and we do not rely on any of the results of~\cite{CM22}
for our study of weak randomness and the class $\WR$, except for a straightforward characterization
of subgraphons and sub-objects~\cite[Lemmas~3.3 and~5.8]{CM22} (see also Section~\ref{sec:prelim}
below). To read the current paper, it will be useful to have some familiarity with the theories of
graphons and theons, but we repeat the relevant definitions and results in Section~\ref{sec:prelim}
to set the notation.

Now we describe the structure of the paper. In Section~\ref{sec:prelim}, we review the necessary
preliminaries and set notation. Section~\ref{sec:subst} starts to develop the properties of
substitution, primality and almost finiteness, which we will need for the rest of the paper.
Section~\ref{sec:pers:graph} is devoted to proving the persistence result for graphs,
Theorem~\ref{thm:graphpersistence}. Section~\ref{sec:WR:graph} defines the class $\WR$ for graphs
and proves the characterization under the assumption of closure under substitutions,
Theorem~\ref{thm:graphs:WR}. In Section~\ref{sec:VC:graph}, we study how the notions of weak
randomness interact with VC~dimension, show that weakly random graphons of proper theories of graphs
must be a.e.\ $\{0,1\}$-valued (Theorem~\ref{thm:weaklyrandom01}) and show that primally almost
finite families of graphs must have bounded VC~dimension
(Theorem~\ref{thm:primallyalmostfiniteNIP}). In Section~\ref{sec:pers:univ}, we prove the general
characterization of strongly persistent classes of structures in finite relational languages
(Theorem~\ref{thm:persistence}). In the brief Section~\ref{sec:WR:univ}, we point out which results
concerning $\WR$ generalize easily to finite relational languages. In the final
Section~\ref{sec:concl}, we summarize and discuss some open problems.

\subsection*{Acknowledgments}

We thank two anonymous reviewers for helpful comments.

\section{Preliminaries}
\label{sec:prelim}

In this section, we establish the notation and background results that will be used throughout the
paper. The core results of the paper are in probabilistic combinatorics, and most of the results and
proofs are stated in that language. Still, there are quite a few points where we believe the
introduction of (simple) model theoretic language is more natural both to explain our approach and
to organize the results, as we shall explain.

\medskip

We denote the set of non-negative integers by $\NN$ and the set of positive integers by
$\NN_+\df\NN\setminus\{0\}$ and given $n,k\in\NN$, we let $[n]\df\{1,\ldots,n\}$ and let $(n)_k\df
n(n-1)\cdots(n-k+1)$ denote the falling factorial. Given a set $V$ and $k\in\NN$, we let $(V)_k$ be
the set of \emph{injective} functions $[k]\to V$, we let $\binom{V}{k}\df\{A\subseteq V \mid\lvert
A\rvert = k\}$ be the set of subsets of $V$ of size $k$, let $\binom{V}{\leq
  k}\df\bigcup_{\ell=0}^k\binom{V}{\ell}$ and we let $r(V)\df\bigcup_{k\in\NN_+}\binom{V}{k}$ be the
set of non-empty finite subsets of $V$. We will often abuse notation and write $n$ in place of $[n]$
when $V=[n]$ in some of the notation.

\subsection{Terminology from model theory: structures and theories}

In this paper a main object of study is hereditary classes of graphs. These can be seen as a
special case of what are called in model theoretic language ``classes of structures in finite
relational languages'' or even ``universal theories'', and often the greater level of generality is
useful. We now explain all these (quite natural) definitions.

Recall that a family of graphs (up to isomorphism) is called \emph{hereditary} if it is closed under
induced subgraphs. As an example, consider the triangle-free graphs, and observe the
following. Using $R$ as the binary edge symbol, we can write a set of axioms $\TTri$ of first order
logic which capture this class of graphs. First, the theory of graphs $\TGraph$ will say: the edge
relation is symmetric [~$\forall x \forall y (R(x,y) \iff R(y,x))$~] and irreflexive [~$\forall x
  \forall y (R(x,y) \implies x \neq y)$~ ]. To obtain $\TTri$, we add the axiom that there are no
triangles [~$\forall x \forall y \forall z ~ \neg (x \neq y \land y \neq z \land z \neq x \land
  R(x,y) \land R(y,z) \land R(x,z))$~]. This $\TTri$ is called a universal theory because it uses
only universal quantifiers, and to a model theorist, this explains the fact that the axioms still
hold on any induced substructure, or in other words, the graphs satisfying the axioms $\TTri$ form a
hereditary class. Model theorists consider a set of axioms and the class of structures satisfying
those axioms to be two sides of the same coin, so in logical parlance we could say we are studying
the hereditary family of triangle-free graphs, or equivalently, we are studying the universal theory
$\TTri$.

To motivate the phrase ``finite relational languages'', observe that there are other natural
hereditary classes we might want to study, such as: linear orders, tournaments, $3$-uniform
hypergraphs; or perhaps the class of hypergraphs on which we have both a graph edge $E$ and a
$3$-uniform hyperedge $R$, and $E$ has no triangles and $R$ has no tetrahedra (i.e., there are no
four vertices such that every three form an $R$-hyperedge).\footnote{A series of formal examples
  will be worked out later in this section.} The following three definitions give us the right level
of generality. First, we choose our alphabet.

\begin{definition}\label{def:frl}
  A \emph{finite relational language} $\cL$ is a set of finitely many symbols $P_1,\ldots,P_n$, each
  given with an \emph{arity} $k(P_i)\in\NN_+$.
\end{definition}

Second, when we define a graph $G$, we present it as a set $V$ of vertices along with a set
$R\subseteq V\times V$ of edges, and $\cL$-structures just extend this in the obvious way:

\begin{definition}\label{def:structure}
  Given a finite relational language $\cL$, an $\cL$-\emph{structure} $M$ is given by:
  \begin{enumerate}[label={(\alph*)}]
  \item the data of a set $V(M)$, called the vertices of $M$ or the domain of $M$, and
  \item for each $P_i\in\cL$, a subset of $V(M)^{k(P_i)}$, that is, the set of $k(P_i)$-tuples on
    which $P_i$ holds. This set is denoted $P_i^M$ and called the interpretation of $P_i$ in $M$.
  \end{enumerate}
\end{definition}

Finally, we make the bridge to theories:\footnote{Observe that each of the hereditary classes listed
  before Definition~\ref{def:frl} can be expressed as a class of models for some appropriate
  universal axioms using an appropriate $\cL$: for instance $\{{<}\}$ with $k({<}) = 2$, $\{R\}$
  with $k(R) = 2$, $\{R\}$ with $k(R) = 3$, and $\{E,R\}$ with $k(E) = 2$ and $k(R) = 3$,
  respectively.}

\begin{definition}
  A \emph{universal theory} $T$ in the language $\cL$ is a set of axioms (i.e., a set of well formed
  formulas of first order logic, using basic logical symbols along with the symbols from $\cL$) in
  which the only quantifiers are universal. An $\cL$-structure $M$ is said to be a \emph{model} for
  $T$, in symbols $M\models T$, if all the axioms $T$ hold in $M$.\footnote{Formally defining
    ``$M\models T$'' requires an induction on formula complexity, as in~\cite[Chapter~2]{CK90}.}
\end{definition}

Throughout this text, unless explicitly mentioned otherwise, all languages are assumed to be finite
relational languages. We allow\footnote{This may seem curious to model theorists but simplifies our
  calculations.} structures to have empty vertex sets and the unique structure with empty vertex
set, called \emph{trivial structure}, is denoted $K_0$. Given an $\cL$-structure $M$, $V\subseteq
V(M)$ and $v\in V(M)$, we denote the \emph{substructure} of $M$ induced by $V$ by $M\rest_V$ (i.e.,
we have $V(M\rest_V)\df V$ and $P^{M\rest_V}\df P^M\cap V^{k(P)}$ for every $P\in\cL$) and we let $M
- v\df M\rest_{V(M)\setminus\{v\}}$.

We put the following in a convention environment to emphasize its importance:
\begin{convention}\label{conv:23}
  Our substructures and subgraphs will always be induced, but keeping with the tradition of the
  fields, we will use the short term ``substructure'' for the former but the full term ``induced
  subgraph'' for the latter.
\end{convention}

In this paper, and elsewhere, the word ``structure'' and ``model'' are used interchangeably; often
the first emphasizes the abstract aspect, and the second emphasizes the relation to a theory. Basic
model theory verifies that these work as intended: when $T$ is a universal theory, the class of
models of $T$ is closed under substructures, i.e., is a hereditary class.

\begin{convention}
  In this paper the phrase ``universal theory'' will denote both a fixed universal set of axioms $T$
  in some fixed finite relational language $\cL$, and the \emph{class of $\cL$-structures} which are
  models for $T$, which, as noted, is a hereditary class.

  The reader is free to substitute the phrase ``hereditary class'' for ``universal theory''
  throughout, keeping in mind the language being used and the appropriate notion of substructure,
  and Convention~\ref{conv:23}.
\end{convention}

\begin{discussion}
  If the reader is essentially free to read ``universal theory'' as ``hereditary class,'' why do we
  introduce this terminology? This indicates a certain change in perspective which appears to be
  useful for theorems and proofs. Part of this choice reflects a history of work in the area, as in
  the theons of~\cite{CR20b}. Centrally for the present work, a characteristically model-theoretic
  move of ``studying all models of a theory'' can be seen in the definition of the class $\WR$ and
  in various aspects of the proofs.
\end{discussion}

\subsection{Counting embeddings of graphs and structures}

In order to develop the theory of graph limits (see~\ref{subsec:theons} below), one starts by
defining \emph{labeled (induced) density of a finite graph $H$ in some other finite graph $G$}. That
is, let $\Tind(H,G)$ be the set of injective maps $f\colon V(H)\to V(G)$ which preserve edges and
non-edges, and let
\begin{align*}
  \tind(H,G) & \df
  \begin{dcases*}
    \frac{\lvert\Tind(H,G)\rvert}{(\lvert G\rvert)_{\lvert H\rvert}},
    & if $\lvert H\rvert\leq\lvert G\rvert$,
    \\
    0, & otherwise.
  \end{dcases*}
\end{align*}
One also defines the \emph{induced density of $H$ in $G$}:
\begin{align*}
  p(H,G)
  & \df
    \frac{
      \lvert\{U\subseteq V(H) \mid G\rest_U\cong H\}\rvert
    }{
      \binom{\lvert H\rvert}{\lvert G\rvert}
    }
  =
  \frac{\lvert H\rvert!}{\lvert\Aut(H)\rvert}\cdot\tind(H,G),
\end{align*}
when $\lvert H\rvert\leq\lvert G\rvert$ (and defined as $0$ otherwise), which gives the normalized
number of (unlabeled) copies. The discussion in the previous subsection suggests an obvious
extension of this definition to the setting of finite relational languages. That is, given finite
structures $M$ and $N$ in a language $\cL$, we let $\Tind(M,N)$ be the set of embeddings of $M$ in
$N$ (i.e., the set of injective maps $f\colon V(M)\to V(N)$ that preserve all relations and their
negations) and let
\begin{align*}
  \tind(M,N) & \df
  \begin{dcases*}
    \frac{\lvert\Tind(M,N)\rvert}{(\lvert N\rvert)_{\lvert M\rvert}},
    & if $\lvert M\rvert\leq\lvert N\rvert$,
    \\
    0, & otherwise
  \end{dcases*}
\end{align*}
be the \emph{labeled (induced) density} of $M$ in $N$. We also define the \emph{(induced) density}
of $M$ in $N$ as the normalized number of substructures of $N$ that are isomorphic to $M$ given
by
\begin{align*}
  p(M,N)
  & \df
  \frac{
    \lvert\{U\subseteq V(N) \mid N\rest_U\cong M\}\rvert
  }{
    \binom{\lvert N\rvert}{\lvert M\rvert}
  }
  =
  \frac{\lvert M\rvert!}{\lvert\Aut(M)\rvert}\cdot\tind(M,N),
\end{align*}
when $\lvert M\rvert\leq\lvert N\rvert$ (and defined as $0$ otherwise), where $\Aut(M)$ is the group
of automorphisms of $M$.

Given a universal theory $T$ in $\cL$ and a set $V$, we let $\cK_V[T]$ be the set of all models $M$
of $T$ whose vertex set $V(M)$ is $V$. Given $n\in\NN$, we let $\cM_n[T]$ be the set of models of
$T$ of size $n$ up to isomorphism; we typically think of $\cM_n[T]$ as a subset of $\cK_{[n]}[T]$ by
putting one representative of each isomorphism class in $\cM_n[T]$. We also let
$\cM[T]\df\bigcup_{n\in\NN}\cM_n[T]$.

\subsection{Open interpretations}

When comparing hereditary classes it is useful to know when one contains the information of the
other in perhaps a different presentation. As a trivial example, consider the hereditary class of
triangle-free graphs in which the edge relation is called $E$ ($k(E) = 2$) and the hereditary class
of triangle-free graphs in which the edge relation is called $R$ ($k(R) = 2$). As a slightly less
trivial example, compare these to the hereditary class of graphs in which each vertex is either
colored red or green, and the green vertices form a triangle-free graph.

This subsection introduces language for addressing such situations by identifying obviously
equivalent pieces of hereditary classes. Model theorists will recognize ``open'' as meaning
``quantifier-free''.

Recall that for universal theories $T_1$ and $T_2$ in finite relational languages $\cL_1$ and
$\cL_2$, respectively, an \emph{open interpretation} (or \emph{definition}) from $T_1$ to $T_2$ is a
function $I$ (denoted $I\colon T_1\leadsto T_2$) that maps each predicate symbol $P\in\cL_1$ to an
open (i.e., quantifier-free) formula $I(P)(x_1,\ldots,x_{k(P)})$ in $\cL_2$ and such that for each
axiom $\forall\vec{x}, F(\vec{x})$ of $T_1$, we have $T_2\vdash\forall\vec{x}, I(F)(\vec{x})$ when
we declare $I$ to commute with logical connectives. Open interpretations of the form $I\colon
T_1\leadsto T_2$ contra-variantly define maps $\cK_V[T_2]\to\cK_V[T_1]$ for each set $V$ given by
$(I(M)\vDash P(\vec{x})) \iff (M\vDash I(P)(\vec{x}))$ for each $P\in\cL_1$. In turn, for an open
interpretation $I\colon T_1\leadsto T_2$, we let $I(T_2)$ be the universal theory in the language of
$T_1$ whose finite models are precisely those of the form $I(M)$ for some $M\in\cM[T_2]$, that is,
the axioms of $I(T_2)$ are
\begin{align*}
  \forall x_1,\ldots,x_n, \bigvee_{M\in\cK_n[T_2]}\Dopen(I(M))(x_1,\ldots,x_n) \qquad (n\in\NN),
\end{align*}
where $\Dopen(N)$ is the \emph{open diagram} of $N$, that is, the open formula
\begin{align*}
  \bigwedge_{1\leq i < j\leq n} x_i\neq x_j\land
  \bigwedge_{P\in\cL_2}\left(\bigwedge_{\alpha\in P^N} P(x_{\alpha_1},\ldots,x_{\alpha_{k(P)}})\land
  \bigwedge_{\alpha\in V(N)^{k(P)}\setminus P^N} \neg P(x_{\alpha_1},\ldots,x_{\alpha_{k(P)}})\right)
\end{align*}
that completely encodes the quantifier-free type (over $\varnothing$) of the tuple $(1,\ldots,n)$ in
$N$. To make sense out of $\Dopen(K_0)$ (which must be a quantifier-free formula on zero variables),
we allow our formulas to use the \emph{tautological truth symbol} $\top$ so that $\Dopen(K_0)$ is
defined as $\top$.

An \emph{$\cat{Int}$-isomorphism} (or \emph{interdefinition}) is an open interpretation $I\colon
T_1\leadsto T_2$ such that there exists an open interpretation $J\colon T_2\leadsto T_1$ such that
for every set $V$, the compositions $J\comp I\colon\cK_V[T_2]\to\cK_V[T_2]$ and $I\comp
J\colon\cK_V[T_1]\to\cK_V[T_1]$ are the identity maps. Since $p(M,N) = p(I(M),I(N))$ whenever
$I\colon T_1\leadsto T_2$ is an $\cat{Int}$-isomorphism and $M$ and $N$ are finite models of $T_2$,
we typically do not distinguish between $\cat{Int}$-isomorphic universal theories.

\subsection{Canonical theories: avoiding the diagonal}

When defining graphs we require the edge relation to be irreflexive. In general, a universal theory
$T$ in $\cL$ is \emph{canonical} if $T$ entails
\begin{align}\label{eq:canonical}
  \forall x_1,\ldots,x_{k(P)}, \left(\bigvee_{i\neq j} x_i = x_j \to \neg P(\vec{x})\right)
\end{align}
for every predicate symbol $P\in\cL$. By~\cite[Theorem~2.3]{CR20a} (see also~\cite[\S2.2]{AC14}),
every universal theory is $\cat{Int}$-isomorphic to some canonical theory and as such, from this
point forward, all theories are assumed to be canonical theories, unless explicitly mentioned
otherwise.

We say that a canonical theory $T$ is \emph{non-degenerate} if it contains some infinite model
(equivalently, if $\cM_n[T]$ is non-empty for every $n\in\NN$).

\subsection{Examples of theories used in the text}

To make more concrete our use of ``universal theories'' rather than simply hereditary classes of
graphs, in this subsection we lay out some of the main examples used in the text, along with a
useful construction of a canonical theory. We have already seen the \emph{theory of graphs}
$\TGraph$, and we will say that $T$ ``is a universal theory of graphs'' when $T$ is a universal
theory that is obtained from $\TGraph$ by adding axioms, that is, its finite models are some
hereditary class of graphs (an obvious example is $\TGraph$ itself, a less obvious one is the theory
$\TTri$ of triangle-free graphs).

A second kind of example is the \emph{theory of $k$-hypergraphs} $\TkHypergraph$, that is, the
canonical theory with a single \emph{symmetric} \emph{irreflexive}\footnote{In the sense that the
  predicate is \emph{not} true in any non-injective tuple.}  $k$-ary predicate $E$. Obviously the
theory of graphs is simply $\TGraph=\TkHypergraph[2]$. In these theories, we denote by
$K_n^{(k)}\in\cM_n[\TkHypergraph]$ the \emph{complete $k$-hypergraph on $n$ vertices} (i.e., we have
$V(K_n^{(k)})\df[n]$ and $E^{K_n^{(k)}}\df ([n])_k$) and we let $K_n\df K_n^{(2)}$ be the
\emph{complete graph on $n$ vertices}. Given a $k$-hypergraph $G$, we let $\overline{G}$ denote the
\emph{complement} hypergraph of $G$ (given by $V(\overline{G})\df V(G)$ and $E^{\overline{G}}\df
([n])_k\setminus E^G$). In particular, $\overline{K}_n^{(k)}$ is the \emph{empty $k$-hypergraph on
  $n$ vertices}.

Another example of a canonical theory is the \emph{theory of (strict) linear orders} $\TLinOrder$,
i.e., the theory with a binary predicate symbol $\prec$ and axioms
\begin{gather*}
  \forall x, \neg(x\prec x),\\
  \forall\vec{x}, (x_1\neq x_2 \to (x_1\prec x_2\lor x_2\prec x_1)),\\
  \forall\vec{x}, (x_1\prec x_2\land x_2\prec x_3 \to x_1\prec x_3).
\end{gather*}

Another useful theory is the \emph{theory of permutations} $\TPerm\df\TLinOrder\cup\TLinOrder$,
which is the theory of two (strict) linear orders on the same ground set. Its finite models (up to
isomorphism) are in one-to-one correspondence to usual permutations via $S_n\ni\sigma\mapsto
M_\sigma\in\cM_n[\TPerm]$, where the first order ${\prec_1^{M_\sigma}}$ of $M_\sigma$ is simply the
natural order on $[n]$ and the second order is given by
\begin{align*}
  (M_\sigma\vDash i\prec_2 j) & \iff \sigma^{-1}(i) < \sigma^{-1}(j).
\end{align*}

Some other examples of theories that can be obtained from $\TGraph$ by adding axioms that will be
used are the \emph{theory of graphs of agreements of permutations} $\TPermGraph\df I(\TPerm)$, where
$I\colon\TGraph\leadsto\TPerm$ is given by
\begin{align*}
  I(E)(x,y) & \df (x\neq y\land (x\prec_1 y\tot x\prec_2 y)),
\end{align*}
the \emph{theory of bipartite graphs} $\TBipartite$, which is obtained from $\TGraph$ by adding the
axioms $\forall\vec{x},\neg\Dopen(C_{2n+1})(\vec{x})$ for every $n\in\NN_+$, where $C_\ell$ is the
$\ell$-cycle graph and the \emph{theory of perfect graphs}, which by the Strong Perfect Graph
Theorem~\cite{CRST06}, is obtained from $\TGraph$ by adding the axioms
$\forall\vec{x},\neg(\Dopen(C_{2n+1})(\vec{x})\lor\Dopen(\overline{C}_{2n+1})(\vec{x}))$ for every
$n\geq 2$.

For every finite relational language $\cL$, we let $T_\cL$ be the \emph{pure canonical theory} in
$\cL$, that is, the theory whose axioms are precisely~\eqref{eq:canonical} for each
$P\in\cL$. Unless explicitly mentioned otherwise, all $\cL$-structures are assumed to be
\emph{canonical structures}, that is, models of $T_\cL$.

Given a family $\cF$ of models of a canonical theory $T$, we let $\Forb_T(\cF)$ be the theory of
models of $T$ that do not have any copies of models in $\cF$, that is, $\Forb_T(\cF)$ is obtained
from $T$ by adding the axioms $\forall\vec{x},\neg\Dopen(F)(\vec{x})$ for every $F\in\cF$ (note that
if $F = K_0$, then this formula takes the form $\neg\top$, which is tautologically false, so
$\Forb_T(\cF)$ has no models when $K_0\in\cF$).

Another simple but useful construction is the following: if $\cF$ is a family of finite
$\cL$-structures that is closed under substructures, then we let
$\Th(\cF)\df\Forb_{T_\cL}(\cM[T_\cL]\setminus\cF)$ be the unique universal theory (up to
reaxiomatization) such that $\cM[\Th(\cF)]=\cF$.

\subsection{Basics of graphons and theons}
\label{subsec:theons}

We start here with the language of models before specializing to the case of graphons, which will be
central for an initial segment of the text.

A sequence $(N_n)_{n\in\NN}$ of finite models of a canonical theory $T$ is called \emph{convergent}
if it is \emph{increasing} in the sense that $\lvert N_n\rvert < \lvert N_{n+1}\rvert$ for every
$n\in\NN$ and if for every $M\in\cM[T]$, the limit $\lim_{n\to\infty} p(M,N_n)$ exists. Another way
of seeing this convergence is that each finite model $N$ of $T$ corresponds to the point
$p(\place,N)\in [0,1]^{\cM[T]}$ and convergence of an increasing sequence $(N_n)_{n\in\NN}$ amounts
to convergence in the (compact metrizable) product topology of $[0,1]^{\cM[T]}$ of the corresponding
sequence $(p(\place,N_n))_{n\in\NN}$.

There are essentially two ways to encode limits of convergent sequences. The first is
algebraically/syntactically: we say that a function $\phi\colon\cM[T]\to[0,1]$ is the limit of a
convergent sequence $(N_n)_{n\in\NN}$ if $\phi(M)=\lim_{n\to\infty} p(M,N_n)$ for every
$M\in\cM[T]$. The theory of flag algebras~\cite{Raz07} describes the set $\HomT{T}$ of functions
that are limits of convergent sequences as precisely as the ones that induce positive homomorphisms
from a particular commutative $\RR$-algebra $\cA[T]$ to $\RR$, but for this work, the unfamiliarized
reader can safely think of $\HomT{T}$ as simply a fancy notation for the subset of $[0,1]^{\cM[T]}$
of all $\phi$ that are limits of some convergent sequence. Note that compactness of $[0,1]^{\cM[T]}$
implies that $\HomT{T}$ is non-empty if and only if $T$ is non-degenerate.

For $\phi\in\HomT{T}$, the \emph{theory of positive models} of $\phi$ is the universal theory
$\Th(\phi)$ whose finite models are precisely those models $M$ of $T$ such that $\phi(M) > 0$, that
is, the axioms of $\Th(\phi)$ are
\begin{align*}
  \forall\vec{x}, \bigvee_{\substack{M\in\cK_n[T]\\\phi(M) > 0}}\Dopen(M)(x_1,\ldots,x_n) \qquad (n\in\NN).
\end{align*}

\medskip

The second way of encoding limits is geometrically/semantically. In the case of graphs, we can
encode limits using a \emph{graphon} $W$ over an atomless standard probability space
$\Omega=(X,\cA,\mu)$, that is, $W$ is a symmetric function $X\times X\to[0,1]$ measurable in the
completion of the product $\sigma$-algebra (typically, the space $\Omega$ is taken to be $[0,1]$
equipped with the Lebesgue measure $\lambda$, in which case, a graphon is simply a symmetric
Lebesgue measurable function $[0,1]^2\to[0,1]$). Given one such graphon $W$ over
$\Omega=(X,\cA,\mu)$ and a finite graph $G$, the \emph{labeled (induced) density} and the
\emph{(induced) density} of $G$ in $W$ are defined respectively as
\begin{align*}
  \tind(G,W)
  & \df
  \int_{X^{V(G)}} \prod_{\{v,w\}\in E(G)} W(x_v,x_w) \prod_{\{v,w\}\in E(\overline{G})} (1 -
  W(x_v,x_w))
  \ d\mu(x),
  \\
  \phi_W(G) \df p(G,W)
  & \df
  \frac{\lvert G\rvert!}{\lvert\Aut(G)\rvert}\cdot\tind(G,W),
\end{align*}
where $E(G)\df\{\{v,w\} \mid G\vDash E(v,w)\}$ is the edge set of $G$ and $\overline{G}$ is the
complement of $G$. We say that a convergent sequence $(H_n)_{n\in\NN}$ of graphs converges to $W$ if
$\lim_{n\to\infty} p(G,H_n) = \phi_W(G)$ for every $G\in\cM[\TGraph]$. Another way of interpreting
$\tind(G,W)$ above is to define the set $\Tind(G,W)$ of \emph{labeled (induced) copies} of $G$ in
$W$ as
\begin{multline}
  \Tind(G,W)
  \df
  \biggl\{(x,y)\in X^{V(G)}\times [0,1)^{\binom{V(G)}{2}} \;\bigg\vert\;
  \forall\{v,w\}\in\binom{V(G)}{2},
  \\
  (\{v,w\}\in E(G)\tot y_{\{v,w\}} < W(x_v,x_w))
  \biggr\}
\end{multline}
and note that $\tind(G,W) = (\mu\otimes\lambda)(\Tind(G,W))$. We also use the shorthand
$\Th(W)\df\Th(\phi_W)$ for the theory of positive graphs of $W$. Note that when $W$ is
$\{0,1\}$-valued, we can interpret it as the adjacency matrix of a graph with vertex set $X$ and for
$(x,y)\in X^{V(G)}\times [0,1)^{\binom{V(G)}{2}}$ such that all coordinates of $x$ are distinct, we
have $(x,y)\in\Tind(G,W)$ if and only if $x$ is an embedding of $G$ in the graph encoded by the
$\{0,1\}$-valued $W$.

The main theorem of the theory of graphons~\cite{LS06} says that graphons precisely encode limits of
convergent graph sequences. Along with the flag algebra description, this can be easily summarized
as $\HomT{\TGraph} = \{\phi_W \mid W\text{ is a graphon}\}$. However, let us note that different
graphons can represent the same limit; for example, for any graphon $W$ over $[0,1]$, the graphon
$W'$ defined by $W'(x,y)\df W(2x\bmod 1, 2y\bmod 1)$ represents the same limit as $W$ (i.e., we have
$\phi_W = \phi_{W'}$).

Another very useful theorem is the Graphon Removal Lemma~\cite[Theorem~1]{Pet13}, which says that
for any graphon $W$ over $\Omega$, there exists a graphon $W'$ that differs from $W$ only by a zero
measure set (hence $\phi_W=\phi_{W'}$) and such that for every $G\in\cM[\TGraph]$, if
$\tind(G,W')=0$, then $\Tind(G,W')\subseteq\cD_{V(G)}$, where
\begin{align}\label{eq:graphondiagonal}
  \cD_V
  & \df
  \{(x,y)\in X^V\times [0,1)^{\binom{V}{2}} \mid \exists v,w\in V, (v\neq w\land x_v = x_w)\}
\end{align}
is the \emph{diagonal set}, that is, the Graphon Removal Lemma says that we only need to change $W$
in a zero measure set to remove all off-diagonal copies of finite graphs that have zero density in
$W$.

\medskip

For the general case, we will use the theory of theons~\cite{CR20a} (see also~\cite{Aus08}
and~\cite{AC14} for alternative semantic limits). Given an atomless standard probability space
$\Omega=(X,\cA,\mu)$ and a set $V$, we let $\cE_V(\Omega)\df X^{r(V)}$ (recall that
$r(V)\df\bigcup_{k\in\NN_+}\binom{V}{k}$), equipping it with the completion of the product measure,
which is denoted $\mu$ as well, by abuse.

For a predicate symbol $P$ in a finite relational language $\cL$, a \emph{$P$-on} over $\Omega$ is a
measurable subset of $\cE_{k(P)}(\Omega)$. An \emph{Euclidean structure} in $\cL$ over $\Omega$ is a
function $\cN$ that maps each predicate symbol $P\in\cL$ to a $P$-on
$\cN_P\subseteq\cE_{k(P)}(\Omega)$. If we are further given a finite (canonical) $\cL$-structure
$K$, we define the set of \emph{labeled (induced) copies} of $K$ in $\cN$ as
\begin{align*}
  \Tind(K,\cN)
  & \df
  \bigcap_{P\in\cL}\left(
  \bigcap_{\alpha\in P^K} (\alpha^*)^{-1}(\cN_P)\cap
  \bigcap_{\alpha\in (V(K))_{k(P)}\setminus P^K} (\alpha^*)^{-1}(\cE_{k(P)}(\Omega)\setminus\cN_P)
  \right),
\end{align*}
where for each injection $\alpha\colon [k]\to V$, $\alpha^*\colon\cE_V(\Omega)\to\cE_k(\Omega)$ is
the contra-variantly defined ``projection'' given by
\begin{align}\label{eq:alpha*}
  \alpha^*(x)_A & \df x_{\alpha(A)} \qquad (x\in\cE_V(\Omega), A\in r(k)).
\end{align}
Similarly to the graphon case, we let
\begin{align*}
  \tind(K,\cN) & \df \mu(\Tind(K,\cN)), &
  \phi_\cN(K) & \df \frac{\lvert K\rvert!}{\lvert\Aut(K)\rvert}\cdot\tind(K,\cN),
\end{align*}
and we say that a convergent sequence of finite structures $(N_n)_{n\in\NN}$ converges to $\cN$ if
$\lim_{n\to\infty} p(K,N_n) = \phi_\cN(K)$ for every finite structure $K$. Similarly, we use the
shorthand $\Th(\cN)\df\Th(\phi_\cN)$ for the theory of positive models of $\cN$.

For a canonical theory $T$ in $\cL$, a \emph{(weak) $T$-on} over $\Omega$ is an Euclidean structure
$\cN$ in $\cL$ over $\Omega$ such that $\phi_\cN(K)=0$ whenever $K$ is a finite $\cL$-structure that
is \emph{not} a model of $T$. A \emph{strong $T$-on} over $\Omega$ is a $T$-on $\cN$ such that for
every finite $\cL$-structure $K$ that is \emph{not} a model of $T$, we have
$\Tind(K,\cN)\subseteq\cD_{V(K)}(\Omega)$, where
\begin{align}\label{eq:diagonal}
  \cD_V(\Omega)
  & \df
  \{x\in\cE_V(\Omega) \mid \exists v,w\in V, (v\neq w\land x_{\{v\}} = x_{\{w\}})\}
\end{align}
denotes the \emph{diagonal set}.

The main theorem of the theory of theons says that theons precisely encode limits of convergent
sequences of models, that is, we have $\HomT{T} = \{\phi_\cN \mid \cN\text{ is a $T$-on}\}$. In
fact, the easy inclusion of this equality is worth spelling out: given a $T$-on $\cN$ over $\Omega$,
for each $n\in\NN$, we sample $\rn{\theta}$ in $\cE_n(\Omega)$ according to $\mu$ and let $\rn{N_n}$
be the random element of $\cK_n[T_\cL]$ given by
\begin{gather*}
  V(\rn{N_n}) \df [n],\\
  (\rn{N_n}\vDash P(\alpha)) \iff \alpha^*(\rn{\theta})\in\cN_P
  \qquad (P\in\cL, \alpha\in ([n])_{k(P)}),
\end{gather*}
where $\alpha^*\colon [k(P)]\to[n]$ is given by~\eqref{eq:alpha*}. It is a straightforward exercise
on distribution concentration to check that with probability $1$, the sequence
$(\rn{N_n})_{n\in\NN}$ converges to $\phi_\cN$. In particular, this means that any limit
$\phi\in\HomT{T}$ is also a limit of a sequence of models $(N_n)_{n\in\NN}$ that does not omit sizes
in the sense that $\lvert N_n\rvert = n$ for every $n\in\NN$. However, similarly to graphons,
different theons can represent the same limit.

Similarly to graphons, another very useful theorem of the theory of theons is the Induced Euclidean
Removal Lemma~\cite[Theorem~3.3]{CR20a}, which says that any weak $T$-on can be turned into a strong
$T$-on by changing its peons only on a zero measure set (which in particular means the two theons
represent the same limit). A fortiori, by viewing a $T$-on $\cN$ as a $\Th(\cN)$-on, the Induced
Euclidean Removal Lemma implies that there exists a $T$-on $\cN'$ whose peons differ from those of
$\cN$ only by a zero measure set and such that $\Tind(K,\cN')\subseteq\cD_{V(K)}(\Omega)$ whenever
$\tind(K,\cN') = 0$.

\medskip

Given an open formula $F(x_1,\ldots,x_n)$ in $\cL$ and an Euclidean structure $\cN$ in $\cL$ over
$\Omega$, the \emph{truth set} $T(F,\cN)\subseteq\cE_n(\Omega)$ of $F$ in $\cN$ is defined
inductively as follows.
\begin{enumerate}
\item $T(x_i = x_i,\cN)\df\cE_n(\Omega)$.
\item $T(x_i = x_j,\cN)\df\varnothing$, if $i\neq j$.%
  \label{it:truthnoninjequal}
\item $T(P(x_{\alpha_1},\ldots,x_{\alpha_{k(P)}}),\cN)\df\varnothing$, if $\alpha\colon[k(P)]\to[n]$
  is not injective.
  \label{it:truthnotinj}
\item $T(P(x_{\alpha_1},\ldots,x_{\alpha_{k(P)}}),\cN)\df(\alpha^*)^{-1}(\cN_P)$ if
  $\alpha\colon[k(P)]\to[n]$ is injective, where $\alpha^*$ is as in~\eqref{eq:alpha*} for $V=[n]$.
\item $T(\place,\cN)$ commutes with logical connectives (so e.g., $T(\neg
  F,\cN)\df\cE_n(\Omega)\setminus T(F,\cN)$ and $T(F_1\lor F_2,\cN)\df T(F_1,\cN)\cup T(F_2,\cN)$).
\end{enumerate}
One might argue that items~\ref{it:truthnoninjequal} and~\ref{it:truthnotinj} above should be
defined as particular subsets of the diagonal $\cD_n(\Omega)$ (see~\eqref{eq:diagonal}), but since
all information is lost in $\cD_n(\Omega)$ (both in weak and strong theons), the definition above is
just as good but simpler. It is straightforward to check that for $K\in\cK_n[T_\cL]$, we have
$\Tind(K,\cN) = T(\Dopen(K),\cN)$.

Open interpretations behave very well with the notion of convergence. Furthermore, there are natural
operations in the theories of flag algebras~\cite[Definition~4 and Theorem~2.6]{Raz07} and
theons~\cite[Remark~6]{CR20a} that capture this action in the limit. Namely, if $I\colon T_1\leadsto
T_2$ is an open interpretation and $(N_n)_{n\in\NN}$ is a convergent sequence of finite models of
$T_2$ converging to $\phi\in\HomT{T_2}$ and to the $T_2$-on $\cN$, then $(I(N_n))_{n\in\NN}$ is a
convergent sequence of models of $T_1$ converging to $\phi^I\in\HomT{T_1}$ and to the $T_1$-on
$I(\cN)$ given by
\begin{align}
  \phi^I(M) & \df \sum_{\substack{M'\in\cM[T_2]\\I(M')\cong M}} \phi(M') \qquad (M\in\cM[T_1]),
  \label{eq:phiI}
  \\
  I(\cN)_P & \df T(I(P),\cN) \qquad (P\in\cL).
  \notag
\end{align}

\medskip

Given $\phi\in\HomT{T}$, a \emph{sub-object} of $\phi$ of measure $c > 0$ is a $\psi\in\HomT{T}$
such that there exist a sequence $(N_n)_{n\in\NN}$ of models converging to $\phi$ and sets
$A_n\subseteq V(N_n)$ such that $\lim_{n\to\infty}\lvert A_n\rvert/\lvert N_n\rvert = c$ and
$(N_n\rest_{A_n})_{n\in\NN}$ converges to $\psi$. By a small abuse, we may use theons $\cN$ and
$\cH$ in place of $\phi$ and/or $\psi$, respectively when $\phi_\cN=\phi$ and $\phi_\cH=\psi$. When
the underlying theory is $\TGraph$, we will use the more natural name \emph{subgraphon} for the
concept of sub-object and with a similar abuse, we will use graphons $W$ and $W'$ in place of $\phi$
and/or $\psi$, respectively when $\phi_W=\phi$ and $\phi_{W'}=\psi$.

For simplicity and for later quotation we spell this out:

\begin{definition}[Subgraphons]
  Given a graphon $W$ over an atomless standard probability space $\Omega$, a (positive measure)
  subgraphon $W'$ of $W$ is a graphon over a space $\Omega'$ such that there exist a sequence
  $(H_n)_{n\in\NN}$ of graphs converging to $W$ and sets $U_n\subseteq V(H_n)$ such that
  $\lim_{n\to\infty}\lvert U_n\rvert/\lvert H_n\rvert > 0$ and $(H_n\rest_{U_n})_{n\in\NN}$
  converges to $W'$.
\end{definition}

Some useful equivalences are the following.

By~\cite[Lemma~3.3]{CM22}, $\psi\in\HomT{\TGraph}$ is a subgraphon of $W$ of measure $c > 0$ if and
only if there exists a measurable function $f\colon X\to [0,1]$ with $\int_X f\ d\mu = c$ such that
$\psi = \phi_{W\rest_f}$, where $W\rest_f$ is the graphon over the space $\Omega_f\df(X,\cA,\mu_f)$
defined by
\begin{align}
  \mu_f(B) & \df \frac{1}{c}\int_B f(x)\ d\mu(x),\label{eq:muf}
  \\
  W\rest_f(x,y) & \df W(x,y).\notag
\end{align}

More generally, by~\cite[Lemma~5.8]{CM22}, $\psi\in\HomT{T}$ is a sub-object of a theon $\cN$ over
$\Omega=(X,\cA,\mu)$ of measure $c > 0$ if and only if there exist a measurable function $f\colon
X\to [0,1]$ with $\int_X f\ d\mu = c$ and a measure-isomorphism $F$ modulo $0$ from the space
$\Omega_f=(X,\cA,\mu_f)$ given by~\eqref{eq:muf} to $\Omega$ such that $\psi=\phi_{\cN\rest_f^F}$
for the theon $\cN\rest_f^F$ over $\Omega_f$ defined by
\begin{align*}
  (\cN\rest_f^F)_P & \df \{x\in\cE_{k(P)}(\Omega_f) \mid x^F\in\cN_P\}\qquad (P\in\cL),
\end{align*}
where $x^F\in\cE_{k(P)}(\Omega)$ is given by
\begin{align*}
  x^F_B & \df
  \begin{dcases*}
    x_B, & if $\lvert B\rvert=1$,\\
    F(x_B), & if $\lvert B\rvert\geq 2$.
  \end{dcases*}
\end{align*}

When the function $f$ in the above is the indicator function $\One_A$ of some positive measure set
$A\subseteq X$, we use the shorthands $\mu_A\df\mu_{\One_A}$, $\Omega_A\df\Omega_{\One_A}$, $W\rest_A\df
W\rest_{\One_A}$ and $\cN\rest_A^F\df\cN\rest_{\One_A}^F$ for the concepts above. However, we point out
that not every sub-object of $\cN$ is necessarily of the form $\cN\rest_A^F$ for some positive
measure set $A$ (see~\cite[Example~45]{CR20b}).

\subsection{The approximate \Erdos--Hajnal property}

We conclude this section by recalling the definition of the approximate \Erdos--Hajnal property ($\AEHP$)
from~\cite[Definition~7.1]{CM22}.

\begin{definition}\label{def:AEHP}
  A universal theory $T$ in a finite relational language $\cL$ has the approximate \Erdos--Hajnal
  property ($\AEHP$) if every limit $\phi$ of $T$ has a \emph{trivial} sub-object, i.e., a
  sub-object $\psi$ of the form $\psi=\phi_\cN$ for some $T$-on $\cN$ whose peons all have measure
  in $\{0,1\}$.
\end{definition}

In particular, the definition above specializes to universal theories of graphs as follows.

\begin{definition}\label{def:AEHPgraphs}
  If $T$ is a universal theory of graphs (in other words, if $T$ is a hereditary class of graphs),
  then $T\in\AEHP$ if every graphon that is a limit of $T$ has a \emph{trivial} subgraphon, i.e., a
  subgraphon that is either a.e.\ equal to $0$ or a.e.\ equal to $1$.
\end{definition}

An equivalent formulation of $\AEHP$ (see~\cite[Theorem~7.11]{CM22}) is that for every convergent
sequence $(N_n)_{n\in\NN}$ of models of $T$, there exist sets $U_n\subseteq V(N_n)$ with
$\lim_{n\to\infty}\lvert U_n\rvert/\lvert N_n\rvert > 0$ and $(N_n\rest_{U_n})_{n\in\NN}$ converges
to a trivial limit. In other words, $\AEHP$ for graphs requires linear-sized almost cliques or
almost anti-cliques in the presence of convergence.

The property $\AEHP$ was introduced and its graph version was characterized both combinatorially and
model-theoretically in~\cite{CM22}.

\section{Substitution and primality}
\label{sec:subst}

A simple operation for graphs that is useful in studying the \Erdos--Hajnal property and its
approximate version ($\AEHP$, see Definitions~\ref{def:AEHP} and~\ref{def:AEHPgraphs}) is
substitution. While for graphs this operation has received considerable
attention~\cite{Gia97,APS01,Zve03,CKOS16}, we will need a slight generalization of it for structures
in finite relational languages along with some associated notions (e.g., primality).

In the graph case, some of the results of this section have appeared in some shape or another in the
literature. As such, in Section~\ref{subsec:substgraphs}, we state the definitions and the results
for graphs that we will use without proofs but with pointers to their corresponding generalizations
for relational structures that appear in Section~\ref{subsec:subststruc} with proofs. The reader
that feels sufficiently confident in their knowledge of these and is only interested in the graph
case may read Section~\ref{subsec:substgraphs} then freely skip the remainder of the section.

\subsection{Substitution and primality for graphs}
\label{subsec:substgraphs}

\begin{definition}[Graph version of Definition~\ref{def:subst}]
  Given two graphs $F_1$ and $F_2$ and $v\in V(F_1)$, the \emph{substitution} of $v$ in $F_1$ by
  $F_2$ is the graph $F_1^{v\to F_2}$ obtained from the disjoint union of $F_1 - v$ with $F_2$ by
  adding all edges of the form $\{u,w\}$, where $u\in V(F_1)$ is a vertex that is adjacent to $v$
  in $F_1$ and $w$ is a vertex of $F_2$ (see Figure~\ref{fig:graphsubst} for an example).

  \begin{figure}[htbp]
    \begingroup
\def\scale{1}
\def\firstbaseangle{0}
\def\secondbaseangle{45}
\def\baseradius{1cm}
\def\labelradius{1.5cm}

\let\bigbaseangle\firstbaseangle
\def\bigradius{1.75cm}
\def\biglabelradius{2.25cm}
\let\smallbaseangle\secondbaseangle
\def\smallradius{0.5cm}
\def\smalllabelradius{1.2cm}

\def\ptsize{3pt}
\def\smallptsize{2pt}
\def\captionheight{-3cm}
\def\secondcaptionheight{-3cm}

\def\drawgraph#1#2#3#4#5{
  \begin{tikzpicture}[scale=\scale]
    \pgfmathtruncatemacro{\nmo}{#1-1}
    \foreach \i in {0,...,\nmo}{
      \pgfmathsetmacro{\angle}{#5 + \i * 360 / #1}
      \coordinate (P\i) at (\angle:\baseradius);
      \coordinate (L\i) at (\angle:\labelradius);

      \filldraw (P\i) circle (\ptsize);
      \node at (L\i) {#3};
    }

    \foreach \i/\j in #2{
      \draw (P\i) -- (P\j);
    }

    \node at (0,\captionheight) {#4};
  \end{tikzpicture}
}

\def\drawsubstitutedgraph#1#2#3#4#5#6#7#8{
  \begin{tikzpicture}[scale=\scale]
    \pgfmathtruncatemacro{\nmo}{#1-1}
    \foreach \i in {0,...,\nmo}{
      \pgfmathsetmacro{\angle}{\bigbaseangle + \i * 360 / #1}
      \coordinate (GP\i) at (\angle:\bigradius);
      \coordinate (GL\i) at (\angle:\biglabelradius);
    }

    \foreach \i in {1,...,\nmo}{
      \filldraw (GP\i) circle (\ptsize);
      \node at (GL\i) {#3};
    }

    \pgfmathtruncatemacro{\mmo}{#5-1}
    \foreach \i in {0,...,\mmo}{
      \pgfmathsetmacro{\angle}{\smallbaseangle + \i * 360 / #5}
      \coordinate (FP\i) at ($(GP#4) + (\angle:\smallradius)$);
      \coordinate (FL\i) at ($(GP#4) + (\angle:\smalllabelradius)$);

      \filldraw (FP\i) circle (\smallptsize);
      \node at (FL\i) {#7};
    }

    \foreach \i/\j in #6{
      \draw (FP\i) -- (FP\j);
    }

    \foreach \i/\j in #2{
      \ifnum\i=#4
        \foreach \k in {0,...,\mmo}{
          \draw (FP\k) -- (GP\j);
        }
      \else
        \ifnum\j=#4
          \foreach \k in {0,...,\mmo}{
            \draw (GP\i) -- (FP\k);
          }
        \else
          \draw (GP\i) -- (GP\j);
        \fi
      \fi
    }

    \node at (0,\secondcaptionheight) {#8};
  \end{tikzpicture}
}

\def\Gn{5}
\def\Gedges{0/2,0/3,1/2,1/3,2/4}
\def\Glabel{$v_{\i}$}

\def\Fn{4}
\def\Fedges{0/2,1/2,1/3,2/3}
\def\Flabel{$w_{\i}$}

\def\substvert{0}

\begin{center}
\drawgraph{\Gn}{\Gedges}{{\small\Glabel}}{$F_1$}{\firstbaseangle}
\qquad
\drawgraph{\Fn}{\Fedges}{{\small\Flabel}}{$F_2$}{\secondbaseangle}
\qquad
\drawsubstitutedgraph{\Gn}{\Gedges}{{\small\Glabel}}{\substvert}{\Fn}{\Fedges}%
                     {{\footnotesize\Flabel}}{$F_1^{v_{\substvert}\to F_2}$}
\end{center}
\endgroup

    \caption{Example of a graph substitution.}
    \label{fig:graphsubst}
  \end{figure}
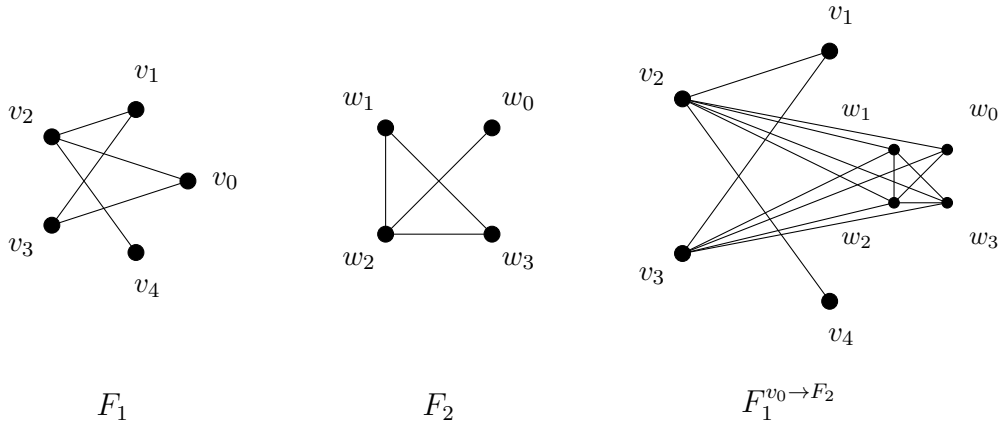

  We say a family $\cF$ of graphs (up to isomorphism) is \emph{closed under substitution} if for
  every $F_1,F_2\in\cF$ and every $v\in V(F_1)$, we have $F_1^{v\to F_2}\in\cF$. The \emph{closure
    under substitutions} of $\cF$ is the smallest family $S(\cF)$ containing $\cF$ that is closed
  under substitution.

  A graph $F$ is called \emph{prime} if it is \emph{not} a substitution of $v$ in $F_1$ by $F_2$ for
  any graphs $F_1,F_2$ and $v\in V(F_1)$ with $\lvert F_1\rvert,\lvert F_2\rvert < \lvert F\rvert$.
\end{definition}

\begin{lemma}[Graph version of Lemma~\ref{lem:substitutionK0}]\label{lem:graph:substitutionK0}
  Let $\cF$ be a non-empty family of graphs (up to isomorphism) that is closed under
  substitutions. Then $\cF$ is closed under induced subgraphs if and only if $\cF$ contains the
  trivial graph $K_0$ of size $0$.
\end{lemma}

\begin{lemma}[Graph version of Lemma~\ref{lem:substitutionsubstructure}]\label{lem:graph:substitutionsubstructure}
  Let $F_1,F_2$ be finite graphs and let $v\in V(F_1)$.

  If $F$ is an induced subgraph of $F_1^{v\to F_2}$, then there exist induced subgraphs $G_1$ and
  $G_2$ of $F_1$ and $F_2$, respectively, with $v\in V(G_1)$ such that $F\cong G_1^{v\to G_2}$.

  Conversely, if $G_1$ and $G_2$ are induced subgraphs of $F_1$ and $F_2$, respectively, with $v\in
  V(G_1)$, then $G_1^{v\to G_2}$ is an induced subgraph of $F_1^{v\to F_2}$.
\end{lemma}

\begin{lemma}[Graph version of Lemma~\ref{lem:primesubstructure}]\label{lem:graph:primesubstructure}
  If $\cF$ is a (not necessarily finite) family of graphs (up to isomorphism), $F\in S(\cF)$ and $P$
  is a prime induced subgraph of $F$, then $P$ is an induced subgraph of some $F'\in\cF$.
\end{lemma}

\begin{lemma}[Graph version of Lemma~\ref{lem:strongclosuresubst}]\label{lem:graph:strongclosuresubst}
  Let $\cF$ be a family of finite graphs (up to isomorphism) that is closed under substitutions and
  closed under induced subgraphs and let $\cP$ be the set of graphs in $\cF$ that are prime. Then
  $\cF = S(\cP)$.

  Conversely, if $\cP'$ is a family of prime finite graphs that is closed under prime induced
  subgraphs and $\cF = S(\cP')$, then $\cP'=\cP$.
\end{lemma}

\begin{definition}[Graph version of Definition~\ref{def:almostfinite}]
  We say that a family of graphs $\cF$ (up to isomorphism) is \emph{almost finite} if $\cF$ does not
  contain any infinite antichain in the induced subgraph partial order. Equivalently, $\cF$ is
  almost finite if for every infinite $\cF'\subseteq\cF$, there exist $F_1,F_2\in\cF'$ such that
  $F_1$ is a proper induced subgraph of $F_2$.

  By letting further $\cP$ be the set of all graphs of $\cF$ that are prime, we say that $\cF$ is
  \emph{primally finite} if $\cP$ is finite and we say that $\cF$ is \emph{primally almost finite}
  if $\cP$ is almost finite.
\end{definition}

\begin{lemma}[Graph version of Lemma~\ref{lem:almostfinite}]\label{lem:graph:almostfinite}
  The following are equivalent for a family $\cF$ of finite graphs (up to isomorphism).
  \begin{enumerate}
  \item The family $\cF$ is almost finite.
  \item For every sequence $(F_n)_{n\in\NN}$ in $\cF$, there exist $n,m\in\NN$ such that $n < m$ and
    $F_n$ is an induced subgraph of $F_m$.
  \end{enumerate}
\end{lemma}

\subsection{Substitution and primality for relational structures}
\label{subsec:subststruc}

\begin{definition}\label{def:subst}
  Given two structures $F_1$ and $F_2$ in a finite relational language $\cL$ and $v\in V(F_1)$, a
  \emph{substitution} of $v$ in $F_1$ by $F_2$ is an $\cL$-structure $F$ such that there exist
  functions $f_1\colon V(F_1-v)\to V(F)$ and $f_2\colon V(F_2)\to V(F)$ such that
  \begin{enumerate}
  \item $V(F)=\im(f_1)\cup\im(f_2)$,
  \item $f_2$ is an embedding of $F_2$ in $F$,
  \item For every $u\in V(F_2)$, the extension of $f_1$ to a function $V(F_1)\to V(F)$ that maps $v$
    to $f_2(u)$ is an embedding of $F_1$ in $F$.
  \end{enumerate}
  We call the substitution $F$ \emph{standard} if $V(F_1)\cap V(F_2)=\varnothing$ and the functions
  $f_1$ and $f_2$ act identically on their domains (thus $V(F)= (V(F_1)\setminus\{v\})\cup
  V(F_2)$). (See Figure~\ref{fig:graphsubst} for a graph example.)

  The unique substitution $F$ (up to isomorphism) of $v$ in $F_1$ by $F_2$ that has the smallest
  possible relation sets $P^F$ ($P\in\cL$) is called the \emph{conservative substitution of $v$ in
    $F_1$ by $F_2$} and is denoted $F_1^{v\to F_2}$. If $V(F_1)\cap V(F_2)=\varnothing$, then we can
  formally define $F_1^{v\to F_2}$ by
  \begin{align*}
    V(F_1^{v\to F_2}) & \df (V(F_1)\setminus\{v\})\cup V(F_2),\\
    P^{F_1^{v\to F_2}} & \df P^{F_2}\cup\{f_u\comp\alpha \mid \alpha\in P^{F_1}\land u\in V(F_2)\},
  \end{align*}
  for every $P\in\cL$, where $f_u\colon V(F_1)\to V(F_2)$ is the function that acts identically on
  $V(F_1)\setminus\{v\}$ and has $f_u(v)=u$.

  We say that a family $\cF$ of $\cL$-structures (up to isomorphism) is \emph{strongly closed under
    substitutions} if for every $F_1,F_2\in\cF$, every $v\in V(F_1)$ and every substitution $F$ of
  $v$ in $F_1$ by $F_2$, we have $F\in\cF$. We say that $\cF$ is \emph{weakly closed under
    substitutions} if for every $F_1,F_2\in\cF$, every $v\in V(F_1)$, there exists some substitution
  $F$ of $v$ in $F_1$ by $F_2$ such that $F\in\cF$. The \emph{strong closure under substitutions} of
  $\cF$ is the smallest family $S(\cF)$ containing $\cF$ that is strongly closed under
  substitutions.

  We say that a finite $\cL$-structure $F$ is \emph{prime}\footnote{This should not to be confused
    with the notion of prime model/structure of model theory.} if it is not a substitution of $v$ in
  $F_1$ by $F_2$ for any $F_1,F_2$ and $v\in V(F_1)$ with $\lvert F_1\rvert,\lvert F_2\rvert <
  \lvert F\rvert$. We say that an $\cL$-structure $F$ is \emph{monochromatic} if for every
  \emph{unary} predicate symbol $P\in\cL$, we have $F\vDash\forall x\forall y, P(x)\tot P(y)$, that
  is, each unary predicate is either true everywhere or true nowhere in $F$.
\end{definition}

\begin{example}
  To illustrate Definition~\ref{def:subst}, which includes subtleties that do not appear the case of
  graphs, suppose $\cH$ is a family of $3$-uniform hypergraphs, with hyperedge $R$. Suppose our
  family contains $F_1$, the empty hypergraph on vertices $\{v_1,v_2,v_3,v_4\}$, and $F_2$ is the
  hypergraph on vertices $\{w_1,w_2,w_3\}$ having a single hyperedge $(w_1,w_2,w_3)$. Suppose we
  substitute $F_2$ into $F_1$ by replacing the vertex $v_1$ by a copy of $F_2$. In the hypergraph
  resulting from the substitution, the vertex set is $\{w_1,w_2,w_3,v_2,v_3,v_4\}$ and the
  requirements of the substitution are that there is a hyperedge $(w_1,w_2,w_3)$, and that there is
  no hyperedge involving exactly one of the $w_i$'s and exactly two of the $v_j$'s. However, this
  does not completely determine a hypergraph, since we haven't expressed an opinion about hyperedges
  involving two $w_i$'s and one $v_j$. Roughly speaking the \emph{conservative substitution} is the
  result of saying no to all additional hyperedges, whereas saying that $\cH$ is \emph{strongly
    closed under substitution} says any choice on these non-determined hyperedges is represented in
  the class.
\end{example}

\begin{remark}\label{rmk:substitutionunary}
  Note that if $F$ is a substitution of $v$ in $F_1$ by $F_2$, then for every unary predicate symbol
  $P\in\cL$, we must have
  \begin{align*}
    (F_1\vDash P(v)) \implies (F_2\vDash\forall x, P(x)),\\
    (F_1\vDash \neg P(v)) \implies (F_2\vDash\forall x, \neg P(x)).
  \end{align*}

  In particular, this means that any $\cF$ that is weakly closed under substitution can have at most
  one structure $M_1$ of size $1$ (up to isomorphism), all structures $F$ of $\cF$ are monochromatic
  and of the same ``color'' in the sense that for every unary predicate symbol $P\in\cL$ and every
  $F\in\cF$ with $\lvert F\rvert\geq 1$, we have $M_1\vDash\forall x, P(x)$ if and only if
  $F\vDash\forall x, P(x)$.
\end{remark}

\begin{remark}\label{rmk:increasingsubstitution}
  Note that if $F$ is a substitution of $v$ in $F_1$ by $F_2$, then $\lvert F\rvert\leq\max\{\lvert
  F_1\rvert,\lvert F_2\rvert\}$ if and only if $\min\{\lvert F_1\rvert,\lvert F_2\rvert\}\leq
  1$. When $F_1$ has a single vertex, then $F\cong F_2$; when $F_2$ has a single vertex, then
  $F\cong F_1$; and when $F_2$ has no vertices (i.e., $F_2=K_0$), then $F\cong F_1-v$.

  In particular, this means that every structure of size at most $2$ is prime.
\end{remark}

\begin{remark}\label{rmk:substarity2}
  If all predicates in $\cL$ have arity at most $2$ (which in particular covers the case of the
  theory of graphs), then all substitutions are conservative and the notions of weakly closed under
  substitutions and strongly closed under substitutions coincide. As such, in
  Sections~\ref{sec:pers:graph}, \ref{sec:WR:graph} and~\ref{sec:VC:graph} concerning $\TGraph$, we
  will drop the superfluous qualifiers ``weakly'' and ``strongly'' from the terminology.
\end{remark}

\begin{remark}\label{rmk:prime3}
  If all predicates have arity at least $3$, then the notion of prime structure completely
  degenerates: the only prime structures are the unique structures $K_0$, $M_1$ and $M_2$ of sizes
  $0$, $1$ and $2$, respectively. The reason why every structure $K$ of size at least $3$ is not
  prime is that for any $u\in V(K)$ and $v\in V(M_2)$, $K$ is a substitution of $v$ in $M_2$ by
  $K-u$ since all relations involving $u$ must involve at least two other vertices.
\end{remark}

\begin{remark}\label{rmk:nondegenerate}
  If $T$ is a universal theory such that $\cM[T]$ is weakly closed under substitution, then $T$ is
  non-degenerate if and only if $\cM_2[T]\neq\varnothing$.
\end{remark}

Let us now prove some basic facts about substitutions and primality.

\begin{lemma}\label{lem:substitutionK0}
  Let $\cF$ be a non-empty family of $\cL$-structures (up to isomorphism) that is weakly closed
  under substitutions. Then $\cF$ is closed under substructures if and only if $\cF$ contains the
  trivial structure $K_0$ of size $0$.
\end{lemma}

\begin{proof}
  Follows since a substitution of $v$ in $F$ by $K_0$ is isomorphic to $F-v$.
\end{proof}

\begin{lemma}\label{lem:substitutionsubstructure}
  Let $F_1,F_2$ be finite $\cL$-structures, let $v\in V(F_1)$.

  If $F$ is a substitution of $v$ in $F_1$ by $F_2$ and $U\subseteq V(F)$, then there exist sets
  $U_1\subseteq V(F_1)$ and $U_2\subseteq V(F_2)$ with $v\in U_1$ such that $F\rest_U$ is a
  substitution of $v$ in $F_1\rest_{U_1}$ by $F_2\rest_{U_2}$.

  Conversely, if $U_1\subseteq V(F_1)$ and $U_2\subseteq V(F_2)$ are such that $v\in U_1$ and $F'$
  is a substitution of $v$ in $F_1\rest_{U_1}$ by $F_2\rest_{U_2}$, then there exist a substitution
  $F$ of $v$ in $F_1$ by $F_2$ and a set $U\subseteq V(F)$ such that $F\rest_U\cong F'$.
\end{lemma}

\begin{proof}
  Without loss of generality, by possibly renaming vertices, we can consider only the case when $F$
  is a standard substitution. Then it is straightforward to check that for $U_1\df (U\cap
  V(F_1))\cup\{v\}$ and $U_2\df U\cap V(F_2)$, we have that $F\rest_U$ is a substitution of $v$ in
  $F_1\rest_{U_1}$ by $F_2\rest_{U_2}$.

  \medskip

  For the second assertion, by possibly renaming vertices, we may suppose without loss of generality
  that $F'$ is also a standard substitution. Then it is straightforward to see that setting $U\df
  (U_1\setminus\{v\})\cup U_2$, there exists a standard substitution $F$ of $v$ in $F_1$ by $F_2$
  such that $F\rest_U = F'$.
\end{proof}

\begin{lemma}\label{lem:primesubstructure}
  If $\cF$ is a (not necessarily finite) family of finite $\cL$-structures, $F\in S(\cF)$ and $P$ is
  a prime substructure of $F$, then $P$ is a substructure of some $F'\in\cF$.
\end{lemma}

\begin{proof}
  The proof is by induction in the minimum length $\ell$ of a sequence of substitutions needed to
  obtain $F$ from elements of $\cF$.

  If $\ell = 0$, then $F\cong F'$ for some $F'\in\cF$, so $P$ is a substructure of $F'$.

  If $\ell > 0$, then $F$ is a substitution of $v$ in $M_1$ by $M_2$ for some $M_1,M_2\in S(\cF)$
  and some $v\in V(M_1)$ such that if the minimum lengths of sequences of substitutions needed to
  obtain $M_1$ and $M_2$ from elements in $\cF$ are $\ell_1$ and $\ell_2$, respectively, then
  $\ell_1 + \ell_2 + 1 = \ell$, which in particular implies $\ell_1,\ell_2 < \ell$.

  Without loss of generality, suppose $F$ is a standard substitution of $v$ in $M_1$ by $M_2$. By
  Lemma~\ref{lem:substitutionsubstructure}, we know that there exist $U_1\subseteq V(M_1)$ and
  $U_2\subseteq V(M_2)$ with $v\in U_1$ such that $P$ is isomorphic to some substitution of $v$ in
  $M_1\rest_{U_1}$ by $M_2\rest_{U_2}$. Since $P$ is prime, we must have $\lvert
  P\rvert\leq\max\{\lvert U_1\rvert,\lvert U_2\rvert\}$, so by
  Remark~\ref{rmk:increasingsubstitution}, either $P\cong M_1\rest_{U_1}$, $P\cong M_2\rest_{U_2}$
  or $P\cong M_1\rest_{U_1\setminus\{v\}}$, so $P$ is a substructure of either $M_1$ or $M_2$ and by
  inductive hypothesis, it follows that $P$ is a substructure of some $F'\in\cF$.
\end{proof}

As one might expect, prime structures play a major role in characterizing classes that are strongly
closed under substitutions. This is made precise by the next two lemmas.

\begin{lemma}\label{lem:strongclosuresubst}
  Let $\cF$ be a family of finite $\cL$-structures (up to isomorphism) that is strongly closed under
  substitutions and closed under substructures and let $\cP$ be the set of structures in $\cF$ that
  are prime. Then $\cF = S(\cP)$.

  Conversely, if $\cP'$ is a family of prime finite $\cL$-structures that is closed under prime
  substructures and $\cF = S(\cP')$, then $\cP'=\cP$.
\end{lemma}

\begin{proof}
  Let $\cF' = S(\cP)$. It is obvious that $\cF'\subseteq\cF$. Suppose toward a contradiction that
  $\cF\setminus\cF'\neq\varnothing$ and let $F$ be an $\cL$-structure in $\cF\setminus\cF'$ of
  minimum size.

  We claim that $F$ is prime. Indeed, if not, then $F$ is a substitution of some $v$ in some $F_1$
  by some $F_2$ with $\lvert F_1\rvert,\lvert F_2\rvert<\lvert F\rvert$. Since both $F_1$ and $F_2$
  are proper substructures of $F$ and both $\cF$ and $\cF'$ are strongly closed under substitutions
  and closed under substructures, this contradicts the minimality of $F$. Thus $F$ is prime, so
  $F\in\cP$, which contradicts $F\notin\cF'$.

  \medskip

  For the second assertion, if $\cF$ is empty, then both $\cP$ and $\cP'$ must also be empty. If
  $\cF$ is not empty, then each $P'\in\cP'$ is in $\cF=S(\cP)$, so Lemma~\ref{lem:primesubstructure}
  implies that $P'$ must be a substructure of some element in $\cP$, hence must also be in $\cP$ as
  it is closed under prime substructures (as $\cF$ is closed under substructures). Similarly, every
  element of $\cP$ must be an element of $\cP'$ as the latter is also closed under prime
  substructures.
\end{proof}

\begin{lemma}\label{lem:stronglyclosedprimecF}
  Let $T$ be a canonical theory in a finite relational language $\cL$ and let $\cF$ be the set of
  minimal $\cL$-structures that are \emph{not} models of $T$, that is, the set of all
  $M\in\cM[T_\cL]\setminus\cM[T]$ such that every proper substructure of $M$ is a model of $T$.

  Then $\cM[T]$ is strongly closed under substitution if and only if $\cF$ contains only prime
  structures.
\end{lemma}

\begin{proof}
  For the forward implication, note that if $M\in\cM[T_\cL]\setminus\cM[T]$ is \emph{not} prime, then it is
  a substitution of $v$ in $M_1$ by $M_2$ for some $M_1,M_2\in\cM[T_\cL]$ with $\lvert
  M_1\rvert,\lvert M_2\rvert < \lvert M\rvert$ and $v\in V(M_1)$. Since $\cM[T]$ is strongly closed
  under substitutions, we must have either $M_1\notin\cM[T]$ or $M_2\notin\cM[T]$, hence
  $M\notin\cF$.

  \medskip

  For the backward implication, first note that $T$ is a reaxiomatization of
  $\Forb_{T_\cL}(\cF)$. Let us show that if $M$ is a substitution of $v$ in $M_1\in\cM[T]$ by
  $M_2\in\cM[T]$, then $M\in\cM[T]$. Without loss of generality, let us assume the substitution to
  be standard so there is the natural identification of $V(M)$ with $(V(M_1)\setminus\{v\})\cup
  V(M_2)$.

  Suppose toward a contradiction that $M\notin\cM[T]$, that is, there exists $F\in\cF$ and
  $U\subseteq V(M)$ such that $M\rest_U\cong F$. Since $M_2\in\cM[T]$, we must have $U\cap
  V(M_1)\neq\varnothing$ (otherwise, $F$ would be a substructure of $M_2$) and since $M_1\in\cM[T]$,
  we must have $\lvert U\cap V(M_2)\rvert\geq 2$ (otherwise, $F$ would be a substructure of
  $M_1$). But then $F$ is a substitution of $v$ in $M_1\rest_{(U\cap V(M_1))\cup\{v\}}$ by
  $M_2\rest_{U\cap V(M_2)}$ and since
  \begin{align*}
    \lvert (U\cap V(M_1))\cup\{v\}\rvert
    & \leq
    \lvert U\rvert - \lvert U\cap V(M_2)\rvert + 1
    <
    \lvert U\rvert
    =
    \lvert F\rvert,
    \\
    \lvert U\cap V(M_2)\rvert
    & \leq
    \lvert U\rvert - \lvert U\cap V(M_1)\rvert
    <
    \lvert U\rvert
    =
    \lvert F\rvert,
  \end{align*}
  this contradicts the fact that $F$ is prime (as it is an element of $\cF$). Thus $\cM[T]$ is
  strongly closed under substitutions.
\end{proof}

Next we cover the notion of an almost finite family of structures and some associated notions.

\begin{definition}\label{def:almostfinite}
  We say that a family $\cF$ of $\cL$-structures (up to isomorphism) is \emph{almost finite} if
  $\cF$ does not contain any infinite antichain in the substructure partial order. Equivalently,
  $\cF$ is almost finite if for every infinite $\cF'\subseteq\cF$, there exist $F_1,F_2\in\cF'$ such
  that $F_1$ is a proper substructure of $F_2$.

  Given a family $\cF$ of $\cL$-structures (up to isomorphism), let $\cP\subseteq\cF$ be the set of
  prime structures in $\cF$ and $\cP'\subseteq\cP$ be the set of monochromatic prime structures in
  $\cF$. We say that $\cF$ is
  \begin{enumerate}
  \item \emph{primally finite}, if $\cP$ is finite,
  \item \emph{primally almost finite}, if $\cP$ is almost finite,
  \item \emph{monochromatically primally finite}, if $\cP'$ is finite,
  \item \emph{monochromatically primally almost finite}, if $\cP'$ is almost finite.
  \end{enumerate}
\end{definition}

\begin{example}\label{ex:substitutiongeneration}
  An example of a family of graphs that is primally almost finite, strongly (equivalently,
  weakly) closed under substitutions and closed under substructures but is not primally finite is
  $S(\{P_n \mid n\in\NN\})$, where $P_n$ is the path on $n$ vertices.

  An example of a proper family of graphs that is not primally almost finite, but is strongly closed
  under substitutions and closed under substructures is $S(\{K_0\}\cup\{C_n \mid n\geq 5\})$, where
  $C_n$ is the cycle on $n$ vertices (it is straightforward to check that when $n\geq 5$, $C_n$ is
  prime).

  Another very important such example is the family $S(\{G_n \mid n\geq 6\}\cup\{K_0\})$, where
  $G_n$ is the graph obtained from $P_n$ by adding four vertices $a,b,c,d$ and adding the edges
  $\{a,b\},\{c,d\},\{a,2\},\{b,3\},\{c,n-2\},\{d,n-1\}$, assuming that $V(P_n)=[n]$ in the natural
  order of the path (see Figure~\ref{fig:pathwithfourcycles}). It is straightforward to check that
  each $G_n$ is prime and that they are pairwise incomparable in the induced subgraph partial
  order. Note also that the paths $P_n$ are elements of $S(\{G_n \mid n\geq 6\}\cup\{K_0\})$ as
  $P_n$ is a substructure of $G_n$.
\end{example}

\begin{figure}[htb]
  \begingroup
\def\basedist{1}
\def\ptsize{2pt}
\def\subfigwidth{\linewidth}
\newcommand{\draw}[1]{%
  \begin{subfigure}{\subfigwidth}
    \begin{center}
      \begin{tikzpicture}
        \foreach \i in {1,...,#1}{
          \pgfmathsetmacro{\x}{\i * \basedist}
          \coordinate (P\i) at (\x,0);
          \fill (P\i) circle (\ptsize);
          \node[above] at (P\i) {$\i$};
        }
        
        \foreach \i[%
          remember=\i as \pi (initially 1)%
        ] in {2,...,#1}{
          \draw (P\pi) -- (P\i);
        }

        \coordinate (A) at ($(P2) + (0,-\basedist)$);
        \coordinate (B) at ($(P3) + (0,-\basedist)$);
        \fill (A) circle (\ptsize);
        \fill (B) circle (\ptsize);
        \node[left] at (A) {$a$};
        \node[right] at (B) {$b$};

        \draw (P2) -- (A) -- (B) -- (P3);
        
        \pgfmathtruncatemacro{\nmone}{#1-1}
        \pgfmathtruncatemacro{\nmtwo}{#1-2}
        \coordinate (C) at ($(P\nmtwo) + (0,-\basedist)$);
        \coordinate (D) at ($(P\nmone) + (0,-\basedist)$);
        \fill (C) circle (\ptsize);
        \fill (D) circle (\ptsize);
        \node[left] at (C) {$c$};
        \node[right] at (D) {$d$};

        \draw (P\nmtwo) -- (C) -- (D) -- (P\nmone);
      \end{tikzpicture}
      \caption*{$G_{#1}$}
    \end{center}
  \end{subfigure}
}
\begin{center}
  \draw{6}
  \vskip 0.5cm
  \draw{7}
  \vskip 0.5cm
  \draw{8}
  \caption{Prime graphs $G_n$ of Example~\ref{ex:substitutiongeneration} that form a family that is
    not almost finite.}
  \label{fig:pathwithfourcycles}
\end{center}

\endgroup

\end{figure}
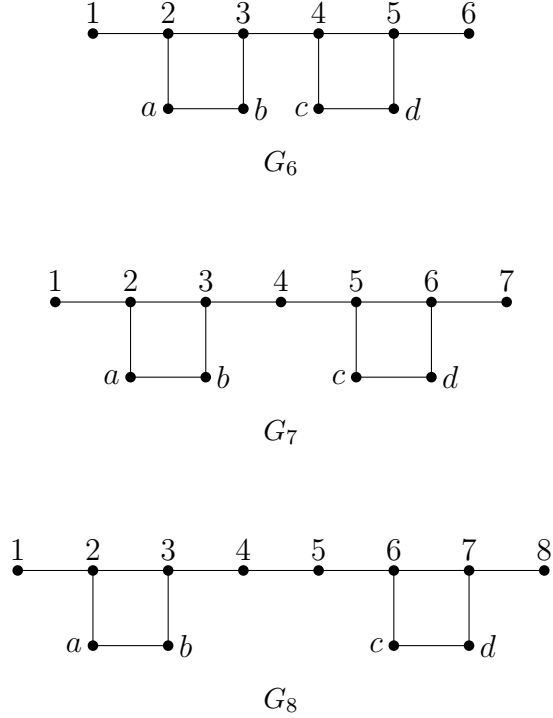

\begin{remark}\label{rmk:primallyalmostfinite}
  A family $\cF$ of the form $\cF = S(\cP')$ for some almost finite set of prime structures $\cP'$
  is not necessarily primally almost finite; this is because the set $\cP$ of structures in $\cF$
  that are prime is equal to the set of prime substructures of elements of $\cP'$, which may be a
  proper superset of $\cP'$ (see Lemma~\ref{lem:primesubstructure}).

  As an example, consider the graphs $G_n$ of Example~\ref{ex:substitutiongeneration} and for each
  $n\geq 6$, define the (prime) graph $H_n$ as the graph obtained from the disjoint union of
  $G_6,\ldots,G_n$ by connecting all first vertices of the $P_k$ inside of $G_k$ in a clique (see
  Figure~\ref{fig:gluedpathswithfourcycles}). Obviously, each $H_n$ is an induced subgraph of
  $H_{n+1}$, so $\cP'\df\{K_0\}\cup\{H_n\mid n\geq 6\}$ is almost finite, but since $\{G_n\mid n\geq
  6\}\subseteq S(\cP')$, it follows that $S(\cP')$ is not primally almost finite.
\end{remark}

\begin{figure}[htb]
  \input{gluedpathswithfourcycles}
\end{figure}

The next lemma uses the fact that the substructure partial order on finite structures is
well-founded to provide a useful equivalent formulation of almost finiteness.

\begin{lemma}\label{lem:almostfinite}
  The following are equivalent for a family $\cF$ of finite $\cL$-structures (up to isomorphism).
  \begin{enumerate}
  \item The family $\cF$ is almost finite.%
    \label{lem:almostfinite:almostfinite}
  \item For every sequence $(F_n)_{n\in\NN}$ in $\cF$, there exist $n,m\in\NN$ such that
    $n < m$ and $F_n$ is a substructure of $F_m$.%
    \label{lem:almostfinite:noincreases}
  \end{enumerate}
\end{lemma}

\begin{proof}
  We start with the
  implication~\ref{lem:almostfinite:almostfinite}$\implies$\ref{lem:almostfinite:noincreases}.

  If there exist $n,m\in\NN$ with $n < m$ and $F_n\cong F_m$, then $F_n$ is a substructure of
  $F_m$. Suppose then that the $F_n$ are pairwise non-isomorphic. Let $I$ be the set of all
  $n\in\NN$ such that for every $m\in\NN$, $F_m$ is not a proper substructure of $F_n$.

  We claim that $I$ is finite. Indeed, otherwise, $\cF'\df\{F_n\mid n\in I\}$ would be an infinite
  subfamily of $\cF$ such that for all $F,F'\in\cF'$, $F$ is not a proper substructure of $F'$.

  We now construct inductively a sequence $m_0,\ldots,m_t$ as follows. Let $m_0 > \max(I)$. Given
  $m_i$, if $m_i\notin I$, then there exists $m_{i+1}$ such that $F_{m_{i+1}}$ is a proper
  substructure of $F_{m_i}$; otherwise, we set $t\df i$ and stop the construction.

  Since $\lvert F_{i+1}\rvert < \lvert F_i\rvert$ and all structures are finite, the construction
  above must stop and by a simple induction, we have that $F_{m_t}$ is a proper substructure of
  $F_{m_0}$. Finally, since $m_t\in I$ and $m_0 > \max(I)$, we get $m_t < m_0$.

  \medskip

  Let us now show the
  implication~\ref{lem:almostfinite:noincreases}$\implies$\ref{lem:almostfinite:almostfinite}. Let
  $\cF'$ be an infinite subfamily of $\cF$ and enumerate it as $(F_n)_{n\in\NN}$ without
  repetitions. Then there exists $n,m\in\NN$ such that $n < m$ and $F_n$ is a substructure of
  $F_m$. Since $F_n\neq F_m$, it follows that $F_n$ is a proper substructure of $F_m$.
\end{proof}

We end this section with the following proposition that relates the notions of primally finite and
primally almost finite for a family $\cF$ strongly closed under substitutions and closed under
substructures with the number of subclasses of $\cF$ that are strongly closed under substitutions
and closed under substructures.

\begin{proposition}\label{prop:primcount}
  Let $\cF$ be a family of finite $\cL$-structures that is strongly closed under substitutions and
  closed under substructures, let $\cS$ be the set of subfamilies $\cF'$ of $\cF$ that are strongly
  closed under substitutions and closed under substructures. Then the following hold.
  \begin{enumerate}
  \item $\cF$ is primally finite if and only if $\cS$ is finite.%
    \label{prop:primcount:finite}
  \item $\cF$ is primally almost finite if and only if $\cS$ is countable.%
    \label{prop:primcount:almostfinite}
  \end{enumerate}
\end{proposition}

\begin{proof}
  Let $\cP$ be the set of prime structures of $\cF$ and let $\cS'$ be the set of
  subfamilies $\cP'\subseteq\cP$ that are closed under prime substructures. Then
  Lemma~\ref{lem:strongclosuresubst} gives a natural bijection between $\cS$ and $\cS'$.

  Since $\cP$ is finite if and only if $\cS'$ is finite, item~\ref{prop:primcount:finite}
  follows.

  The same bijection between $\cS$ and $\cS'$ implies that for
  item~\ref{prop:primcount:almostfinite}, it is sufficient to show that $\cP$ almost finite is
  equivalent to $\cS'$ countable.

  Suppose first that $\cP$ is not almost finite, that is, $\cP$ contains an infinite (countable)
  antichain $\cA\subseteq\cP$. Then for each $\cA'\subseteq\cA$, let $\cC_{\cA'}\subseteq\cP$ be the
  set of elements of $\cP$ that are substructures of some element of $\cA'$. Since $\cA$ is an
  antichain, it follows that for $\cA',\cA''\subseteq\cA$ distinct, we have
  $\cC_{\cA'}\neq\cC_{\cA''}$, hence $\cS'$ has cardinality of the continuum.

  \medskip

  Suppose now that $\cS'$ is uncountable. Let us define by induction in $n\in\NN$ two sequences
  $(P_n)_{n\in\NN}$ and $(\cS'_n)_{n\in\NN}$ with the following properties.
  \begin{enumerate}[label={\arabic*.}, ref={(\arabic*)}]
  \item $\cS'_n\subseteq\cS'$ is uncountable.\label{prop:primcount:cS'n}
  \item There exists $\cC\in\cS'_n$ such that $P_n\in\cC$.\label{prop:primcount:Pnin}
  \item For every $\cC'\in\cS'_{n+1}$, we have $P_n\notin\cC'$.\label{prop:primcount:Pnout}
  \end{enumerate}

  We start by setting $\cS'_0\df\cS'$. Given $\cS'_n$, let $\cP_n\df\bigcup_{\cC\in\cS'_n}\cC$ and
  for each $P\in\cP_n$, we let $\cS'_n(P)\df\{\cC\in\cS'_n \mid P\notin\cC\}$. Since $\cS'_n$ is
  uncountable, $\cP_n$ is countable and $\cS'_n\subseteq\{\cP_n\}\cup\bigcup_{P\in\cP_n}\cS'_n(P)$,
  by the pigeonhole principle, there exists $P_n\in\cP_n$ such that $\cS'_n(P_n)$ is
  uncountable. Set $\cS'_{n+1}\df\cS'_n(P_n)$ so that by definition, we have $P_n\in\cC$ for some
  $\cC\in\cS'_n$ but $P_n\notin\cC'$ for every $\cC'\in\cS'_{n+1}$. This concludes the construction.

  Since each $\cC\in\cS'$ is closed under prime substructures, items~\ref{prop:primcount:cS'n},
  \ref{prop:primcount:Pnin} and~\ref{prop:primcount:Pnout} together imply that for every $n < m$,
  $P_n$ is not a substructure of $P_m$, so by Lemma~\ref{lem:almostfinite}, we get that $\cP$ is not
  almost finite.
\end{proof}

\section{Persistence in graphons}
\label{sec:pers:graph}

In this section we study the notion of (strongly) persistent families of graphs
(Definition~\ref{def:persistence} below). The main objective of this section is to characterize
persistence for theories of graphs in terms of closure under substitution and under induced
subgraphs. We also remind the reader that in this section we drop the qualifiers ``weakly'' and
``strongly'' from ``closed under substitutions'' as they are superfluous for graphs (see
Remark~\ref{rmk:substarity2}).

\begin{definition}\label{def:persistence}
  Let $W$ be a graphon. The set of \emph{positive graphs in $W$} is the set $Q(W)\df\cM[\Th(W)]$ of
  all finite graphs $G$ (up to isomorphism) such that $\phi_W(G) > 0$. The set of \emph{persistently
    positive graphs in $W$} is the set $P(W)\df\bigcap_{W'} Q(W')$, where the intersection is over
  all subgraphons $W'$ of $W$. Clearly, $P(W)$ and $Q(W)$ depend only on the limit
  $\phi_W\in\HomT{\TGraph}$. Thus, for $\phi\in\HomT{\TGraph}$, we let $P(\phi)\df P(\phi_W)$ and
  $Q(\phi)\df Q(\phi_W)$ for any graphon $W$ such that $\phi = \phi_W$.

  A graphon $W$ (or the limit $\phi_W$ it represents) is called \emph{weakly random} if
  $P(W)=Q(W)$. Equivalently, $W$ is weakly random if $Q(W)=Q(W')$ for every subgraphon $W'$ of $W$.

  A family of graphs (up to isomorphism) $\cF$ is called \emph{persistent} if there exists a graphon
  $W$ such that $P(W)=\cF$. The family $\cF$ is called \emph{strongly persistent} if there exists a
  weakly random $W$ such that $P(W) = \cF$ (which must also equal $Q(W)$); in this case, we also say
  that $W$ (or rather $\phi_W$) is a \emph{universal weakly random limit of $\cF$}. If $T$ is a
  universal theory of graphs, then we say that $W$ is a universal weakly random limit of $T$ if it
  is a universal weakly random limit of $\cM[T]$.
\end{definition}

Obviously, both $Q(W)$ and $P(W)$ are closed under induced subgraphs, $Q(W')\subseteq Q(W)$ whenever
$W'$ is a subgraphon of $W$, $P(W)\subseteq Q(W)$ and every strongly persistent family is
persistent.

\begin{example}
  The simplest weakly random graphons are of course the two trivial graphons, that is, the clique
  $W\equiv 1$ and the empty graphon $W\equiv 0$.

  The next in line are the non-trivial quasirandom graphons $W_p\equiv p$ for some $p\in(0,1)$: this
  is because just as the trivial graphons, the quasirandom graphons $W_p$ also have the property
  that the only subgraphon of $W_p$ is $W_p$, up to zero-measure change. In fact, it is an immediate
  consequence of the classic theory of graph quasirandomness~\cite{CGW89} that this property
  characterizes the quasirandom graphons (see also~\cite{SS97,SS03} for related graph
  quasirandomness properties); for general theories in finite relational languages, this property is
  called $\UInduce[1]$ in~\cite{CR20b} and is a strengthening of the more well-known
  discrepancy and clique-discrepancy properties from hypergraph quasirandomness
  (see~\cite{CG90,Chu90,KRS02,LM15b,Tow17,ACHP18}).

  Other examples of weakly random graphons are the recursive blow-up of $C_4$ (see
  Proposition~\ref{prop:recC4}) and the graphon of agreements of the quasirandom permuton (see
  Proposition~\ref{prop:agreementsofQRpermuton}).
\end{example}

The following lemma is a simple but very powerful observation about persistent families.

\begin{lemma}\label{lem:PWsubgraphon}
  Let $W$ be a graphon over a space $\Omega=(X,\cA,\mu)$. Then $P(W) = \bigcap_A Q(W\rest_A)$, where
  the intersection is over all positive measure sets $A\subseteq X$.
\end{lemma}

\begin{proof}
  Since each $W\rest_A$ is a subgraphon of $W$, it is sufficient to show that if $H\in\bigcap_A
  Q(W\rest_A)$, then $H\in P(W)$. We prove this by the contra-positive: if there exists a subgraphon
  $W'$ of $W$ such that $H\notin Q(W')$, then $\phi_{W'} = \phi_{W\rest_f}$ for some measurable
  function $f\colon X\to [0,1]$. Let $A\df\{x\in X \mid f(x) > 0\}$. It is easy to see that
  $Q(W\rest_f)=Q(W\rest_A)$, thus $H\notin\bigcap_A Q(W\rest_A)$.
\end{proof}

The objective of this section is to prove the following theorem that characterizes (strongly)
persistent families in terms of substitutions.

\begin{theorem}\label{thm:graphpersistence}
  The following are equivalent for a family $\cF$ of finite graphs (up to isomorphism) containing at
  least one graph of size at least $2$.
  \begin{enumerate}
  \item The family $\cF$ is strongly persistent.%
    \label{thm:graphpersistence:stronglypersistent}
  \item The family $\cF$ is persistent.%
    \label{thm:graphpersistence:persistent}
  \item The family $\cF$ is closed under substitutions and induced subgraphs.%
    \label{thm:graphpersistence:closed}
  \end{enumerate}
\end{theorem}

\begin{example}
  As we will see, both the theory $\TPerfect$ of perfect graphs
  (Proposition~\ref{prop:perfectgraphtheory}) and the theory $\TPermGraph$ of graphs of agreements
  of permutations (Proposition~\ref{prop:agreementsofpermutationtheory}) have their corresponding
  $\cM[T]$ closed under substitutions, hence both these theories have a universal weakly random
  graphon (in fact, Proposition~\ref{prop:agreementsofQRpermuton} gives an example for the latter
  theory that is very different from the one produced in the proof of
  Theorem~\ref{thm:graphpersistence}).

  On the negative side, the theory of triangle-free graphs does not have $\cM[T]$ closed under
  substitutions (as $K_2\in\cM[T]$ but $K_3\cong K_2^{v\to K_2}\notin\cM[T]$), so it does not have a
  universal weakly random graphon.
\end{example}

We will prove Theorem~\ref{thm:graphpersistence} through a series of lemmas. As we noted before, the
implication~\ref{thm:graphpersistence:stronglypersistent}$\implies$\ref{thm:graphpersistence:persistent}
is trivial. The
implication~\ref{thm:graphpersistence:persistent}$\implies$\ref{thm:graphpersistence:closed} is a
corollary of the following lemma.

\begin{lemma}\label{lem:PW}
  If $W$ is a graphon, then $P(W)$ is closed under substitutions and induced subgraphs.
\end{lemma}

\begin{proof}
  It is obvious that $K_0\in P(W)$, so by Lemma~\ref{lem:graph:substitutionK0}, it is sufficient to
  show that $P(W)$ is closed under substitutions.

  Let $F_1,F_2\in P(W)$ and $v_0\in V(F_1)$ and let us show that if $W'$ is a subgraphon of $W$,
  then $\tind(F_1^{v_0\to F_2},W') > 0$. Without loss of generality, we suppose $V(F_1)\cap V(F_2) =
  \varnothing$. Suppose toward a contradiction that $\tind(F_1^{v_0\to F_2},W') = 0$. By possibly
  applying the Graphon Removal Lemma~\cite[Theorem~1]{Pet13} to $W'$, we may suppose that the set
  $\Tind(F_1^{v_0\to F_2},W')$ of copies of $F_1^{v_0\to F_2}$ in $W'$ is contained in the diagonal
  set $\cD_{V(F_1^{v_0\to F_2})}$ (see~\eqref{eq:graphondiagonal}).

  Since $F_1\in P(W)$, we must have $\tind(F_1,W') > 0$, that is, the set $\Tind(F_1,W')$ has
  positive measure. For every $(x,y)\in X^{V(F_1)\setminus\{v_0\}}\times [0,1)^{\binom{V(F_1)}{2}}$,
  let
  \begin{align*}
    U_{x,y} & \df \{z\in X^{\{v_0\}} \mid ((x,z),y)\in\Tind(F_1,W')\}.
  \end{align*}

  By Fubini's Theorem, there exists $(x,y)\in X^{V(F_1)\setminus\{v_0\}}\times
  [0,1)^{\binom{V(F_1)}{2}}$ with all $x$ coordinates distinct such that $U_{x,y}$ has positive
  measure. Since $W'\rest_{U_{x,y}}$ is a subgraphon of $W'$, hence of $W$, we must have
  $\tind(F_2,W'\rest_{U_{x,y}}) > 0$, which implies that there exists $(z,w)\in
  X^{V(F_2)}\times[0,1)^{\binom{V(F_2)}{2}}$ with all $z$ coordinates in $U_{x,y}$, distinct and
  distinct from those in $x$ such that $z\in\Tind(F_2,W')$. Thus, the point
  $(\widehat{x},\widehat{y})\in X^{V(F_1^{v_0\to F_2})}\times [0,1)^{\binom{V(F_1^{v_0\to F_2})}{2}}$ defined by
  \begin{align*}
    \widehat{x}_v & \df
    \begin{dcases*}
      x_v, & if $v\in V(F_1)\setminus\{v_0\}$,\\
      z_v, & if $v\in V(F_2)$,
    \end{dcases*}
    &
    \widehat{y}_A & \df
    \begin{dcases*}
      y_A, & if $A\subseteq V(F_1)\setminus\{v_0\}$,\\
      w_A, & if $A\subseteq V(F_2)$,\\
      y_{(A\cap V(F_1))\cup\{v_0\}}, & otherwise.
    \end{dcases*}
  \end{align*}
  is a point in $\Tind(F_1^{v_0\to F_2},W')\setminus\cD_{V(F_1^{v_0\to F_2})}$, a contradiction.
\end{proof}

To show the final
implication~\ref{thm:graphpersistence:closed}$\implies$\ref{thm:graphpersistence:stronglypersistent}
of Theorem~\ref{thm:graphpersistence}, we will use the repeating recursive blow-up relative to an
infinite sequence of graphs defined below.

\begin{definition}\label{def:graphrecursiveblowup}
  Given a sequence $V = (V_\ell)_{\ell\in\NN}$ of finite sets with $\lvert V_\ell\rvert\geq 2$ for
  every $\ell\in\NN$, the \emph{Cantor probability space} corresponding to $V$ is the space
  $\Omega^V = (\prod_{\ell\in\NN} V_\ell, \cA, \nu^V)$, where $\cA$ is the Borel $\sigma$-algebra of
  the product topology on $\prod_{\ell\in\NN} V_\ell$ and $\nu^V$ is the unique Borel measure such
  that $\nu^V(K_{\sigma,V}) = \prod_{\ell=0}^{t-1} \lvert V_\ell\rvert^{-1}$ for every $t\in\NN$ and
  every $\sigma\in\prod_{\ell=0}^{t-1} V_\ell$, where $K_{\sigma,V}$ is the basic clopen set given
  by
  \begin{align}\label{eq:KsigmaV}
    K_{\sigma,V}
    & \df
    \left\{\tau\in\prod_{\ell\in\NN} V_\ell \;\middle\vert\;
    \forall\ell\in\{0,\ldots,t-1\},\tau_\ell = \sigma_\ell
    \right\}.
  \end{align}

  Let $G = (G_m)_{m\in\NN}$ be a sequence of finite graphs with $\lvert G_m\rvert\geq 2$ for every
  $m\in\NN$, we let the \emph{recursive blow-up relative to $G$} be the limit
  $\phi_G\in\HomT{\TGraph}$ defined as follows. We let $V=(V_\ell)_{\ell\in\NN}$ be defined by
  $V_\ell\df V(G_\ell)$ and define the graphon $W^G$ over the space $\Omega^V$ by
  \begin{align}\label{eq:WG}
    W^G(x,y) & \df
    \begin{dcases*}
      1, & if $x\neq y$ and $\{x_\ell,y_\ell\}\in E(G_{\ell})$,\\
      0, & otherwise,
    \end{dcases*}
  \end{align}
  where $\ell$ is the first position in which $x$ and $y$ differ. Finally, we define
  $\phi_G\df\phi_{W^G}\in\HomT{\TGraph}$ (see Example~\ref{ex:recC4} and Figures~\ref{fig:recC4}
  and~\ref{fig:recCsquares} for examples).

  We let the \emph{repeating recursive blow-up relative to $G$} be the limit
  $\phi_G^*\in\HomT{\TGraph}$ defined as follows. For each $\ell\in\NN$, we let
  \begin{align}\label{eq:mell}
    m_\ell
    & \df
    \max\{m\in\NN \mid 2^m \text{ divides } \ell+1\},
  \end{align}
  we let $G^*\df(G_{m_\ell})_{\ell\in\NN}$ and we define $\phi_G^*\df\phi_{G^*}$.
\end{definition}

\begin{remark}\label{rmk:lexicographicproduct}
  Since $W^G$ is $\{0,1\}$-valued, we can interpret it as a (measurable) graph $H$ with vertex set
  $\Omega^V$ and the reader familiarized with lexicographic products of graphs should note that
  $H$ is simply the infinite lexicographic product of $(G_m)_{m\in\NN}$.
\end{remark}

\begin{remark}\label{rmk:mell}
  The definition of the numbers $m_\ell$ in~\eqref{eq:mell} guarantees a simple but very useful
  property: for every $m\in\NN$ there exist infinitely many $\ell\in\NN$ with $m_\ell=m$. In fact,
  for every $m\in\NN$ and every $\ell\in\NN$, there exists $\ell'\in\NN$ with $\ell < \ell'\leq \ell
  + 2^m$ such that $m_{\ell'} = m$, that is, for every $m\in\NN$, we only need to wait at most $2^m$
  steps to see $m$ in the sequence $(m_\ell)_{\ell\in\NN}$ regardless of where we start.
\end{remark}

\begin{example}\label{ex:recC4}
  If the sequence $G$ consists of only one graph $G'$ repeated infinitely many times, then $\phi_G$
  is the limit of the usual notion of recursive blow-ups of a single graph $G'$ on progressively
  more and more levels.

  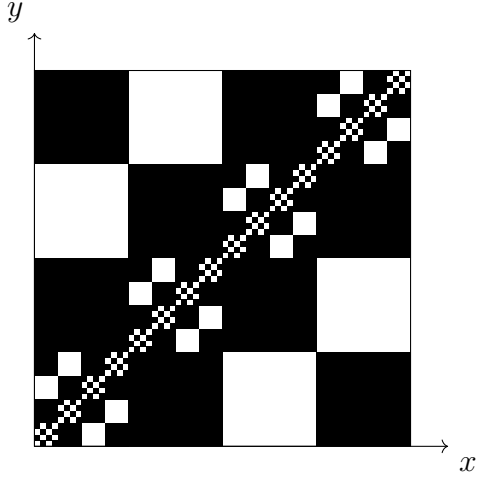
\begin{figure}[htb]
    \begingroup
\def\order{3}
\def\side{5}
\def\axis{5.5}

\begin{center}
  \begin{tikzpicture}
    \draw (0,0) -- (\side,0) -- (\side,\side) -- (0,\side) -- cycle;
    \draw[->] (0,0) -- (\axis,0);
    \draw[->] (0,0) -- (0,\axis);

    \node[below right] at (\axis,0) {$x$};
    \node[above left] at (0,\axis) {$y$};

    \foreach \i in {1,...,\order}{
      \pgfmathsetmacro{\step}{4^(-\i+1) * \side}
      \pgfmathsetmacro{\baseside}{4^(-\i) * \side}
      \pgfmathtruncatemacro{\rep}{4^(\i-1)}
      \foreach \j in {1,...,\rep}{
        \pgfmathsetmacro{\p}{(\j-1) * \step}
        \pgfmathsetmacro{\ppone}{\p + \baseside}
        \pgfmathsetmacro{\pptwo}{\ppone + \baseside}
        \pgfmathsetmacro{\ppthree}{\pptwo + \baseside}
        \foreach \x/\y in {\p/\ppone,\ppone/\pptwo,\pptwo/\ppthree,\ppthree/\p}{
          \fill (\x,\y) -- ++(\baseside,0) -- ++(0,\baseside) -- ++(-\baseside,0) -- cycle;
          \fill (\y,\x) -- ++(\baseside,0) -- ++(0,\baseside) -- ++(-\baseside,0) -- cycle;
        }
      }
    }
  \end{tikzpicture}

  \caption{Approximation of a graphon $W^{C_4}$ over $[0,1]$ representing the limit $\phi_{C_4}$ of
    recursive blow-ups of $C_4$ of Example~\ref{ex:recC4}. The graphon $W^{C_4}$ has a fractal
    structure, whose first $\order$ steps are represented in the picture.}
  \label{fig:recC4}
\end{center}
\endgroup

  \end{figure}

  For example, the limit $\phi_{C_4}$ of recursive blow-ups of $C_4$ used
  in~\cite[Definition~8.5]{CM22} (see Figure~\ref{fig:recC4}) is obtained as $\phi_G$ (or
  $\phi_G^*$) when $G$ is the sequence that is constant equal to $C_4$. Alternatively, $\phi_{C_4}$
  is also obtained as $\phi_{G'}$ for the sequence
  $G'\df(K_2,\overline{K}_2,K_2,\overline{K}_2,\ldots)$ that infinitely alternates between the edge
  graph $K_2$ and the non-edge graph $\overline{K}_2$. Also alternatively, $\phi_{C_4}$ is obtained
  as $\phi_{G''}^*$ for the sequence $G''=(K_2,\overline{K}_2,\overline{K}_2,\ldots)$ whose first
  element is $K_2$ and all other elements are $\overline{K}_2$ (as $(G'')^* = G'$).
\end{example}

Let us show a simple structural fact about the Cantor probability space $\Omega^V$.

\begin{lemma}\label{lem:KsigmaV}
  Let $V = (V_\ell)_{\ell\in\NN}$ be a sequence of finite sets with $\lvert V_\ell\rvert\geq 2$ for
  every $\ell\in\NN$ and let $A\subseteq\Omega^V$ be a set with positive measure. Then for every
  $\epsilon > 0$, there exists $t_0\in\NN$ such that for every $t\geq t_0$, there exists
  $\sigma\in\prod_{\ell=0}^{t-1} V_\ell$ such that $\nu^V(A\cap K_{\sigma,V})\geq
  (1-\epsilon)\cdot\nu^V(K_{\sigma,V})$, where
  \begin{align*}
    K_{\sigma,V}
    & \df
    \left\{\tau\in\prod_{\ell\in\NN} V_\ell \;\middle\vert\;
    \forall\ell\in\{0,\ldots,t-1\}, \tau_\ell = \sigma_\ell
    \right\}
  \end{align*}
  is the basic clopen set defined in~\eqref{eq:KsigmaV}.
\end{lemma}

\begin{proof}
  Let $\cB$ be the Boolean algebra generated by $\cC\df\{K_{\sigma,V} \mid
  t\in\NN\land\sigma\in\prod_{\ell=0}^{t-1} V_\ell\}$ and note that every set in $\cB$ is a finite
  union of elements of $\cC$. In fact, since for every $t\in\NN$ and every
  $\sigma\in\prod_{\ell=0}^{t-1} V_\ell$, the collection $\{K_{(\sigma,v),V} \mid v\in V_t\}$ forms
  a partition of $K_{\sigma,V}$, it follows that for every $B\in\cB$, there exists $t_B\in\NN$ such
  that for every $t\geq t_B$, the set $B$ can be written as the \emph{disjoint} union $B =
  \bigcup_{\sigma\in\Sigma_{B,t}} K_{\sigma,V}$ where
  \begin{align*}
    \Sigma_{B,t} & \df \left\{\sigma\in\prod_{\ell=0}^{t-1} V_\ell \;\middle\vert\;
    K_{\sigma,V}\subseteq B \right\}.
  \end{align*}
  Namely, we can take any representation of $B$ as a finite union of elements of $\cC$ and let $t_B$
  be the maximum length of a $\sigma$ used in this representation.

  Since $\cB$ generates the $\sigma$-algebra $\cA$ of $\Omega^V$, uniqueness of \Caratheodory's
  Extension Theorem implies that for every $\delta > 0$, there exists $B\in\cB$ such that
  $\nu^V(A\symdiff B)\leq\delta$.

  We claim that by taking $\delta\df\nu^V(A)\cdot\epsilon/(1+\epsilon)$ and $t_0\df t_B$, we get
  that for every $t\geq t_0$, there must exist $\sigma\in\Sigma_{B,t}$ such that
  $\nu^V(K_{\sigma,V}\setminus A)\leq\epsilon\cdot\nu^V(K_{\sigma,V})$. Suppose not. Then we have
  \begin{align*}
    \delta
    & \geq
    \nu^V(A\symdiff B)
    \geq
    \sum_{\sigma\in\Sigma_{B,t}} \nu^V(K_{\sigma,V}\setminus A)
    >
    \sum_{\sigma\in\Sigma_{B,t}} \epsilon\cdot\nu^V(K_{\sigma,V})
    =
    \epsilon\cdot\nu^V(B)
    \geq
    \epsilon(\nu^V(A) - \delta),
  \end{align*}
  from which we conclude that
  \begin{align*}
    \delta\cdot\frac{1+\epsilon}{\epsilon} > \nu^V(A),
  \end{align*}
  contradicting the definition of $\delta$. Finally, from $\nu^V(K_{\sigma,V}\setminus
  A)\leq\epsilon\cdot\nu^V(K_{\sigma,V})$, we conclude that $\nu^V(A\cap K_{\sigma,V})\geq
  (1-\epsilon)\cdot\nu^V(K_{\sigma,V})$ as desired.
\end{proof}

Our next objective is to show that $P(\phi_G^*)$ is precisely $S(\{K_0\}\cup\{G_m \mid
m\in\NN\})$. We start by showing the simpler fact $\{G_m \mid m\in\NN\}\subseteq Q(\phi_G)\subseteq
S(\{K_0\}\cup\{G_m\mid m\in\NN\})$ in Lemma~\ref{lem:QphiG} below. Clearly this implies the same
statement for $\phi_G^*$.

\begin{lemma}\label{lem:QphiG}
  Let $G = (G_m)_{m\in\NN}$ be a sequence of finite graphs with $\lvert G_m\rvert\geq 2$ for every
  $m\in\NN$. Then $\{G_m\mid m\in\NN\}\subseteq Q(\phi_G)\subseteq S(\{K_0\}\cup\{G_m\mid
  m\in\NN\})$.
\end{lemma}

\begin{proof}
  To see that every $G_m$ has positive density in $\phi_G$, simply note that if we take an arbitrary
  $\sigma\in\prod_{\ell=0}^{m-1} V(G_\ell)$ and set $x_v\df(\sigma,v)$ for every $v\in V(G_m)$, then
  $x$ is an embedding of $G_m$ in $W^G$ and thus
  \begin{align*}
    \tind(G_m,W^G) & \geq \prod_{\ell=0}^m \frac{1}{\lvert G_\ell\rvert} > 0,
  \end{align*}
  hence $\{G_m\mid m\in\NN\}\subseteq Q(\phi_G)$.

  \medskip

  To show that $Q(\phi_G)\subseteq S(\{K_0\}\cup\{G_m\mid m\in\NN\})$, we will prove a slightly
  stronger result: let us show that if $H$ is a finite graph such that
  $\Tind(H,W^G)\not\subseteq\cD_{V(H)}$, then $H\in S(\{K_0\}\cup\{G_m\mid m\in\NN\})$. The proof is
  by induction on $\lvert H\rvert$.

  The first two base cases are when $\lvert H\rvert\leq 1$ (i.e., $H\in\{K_0,K_1\}$), in which case
  trivially $H\in S(\{K_0\}\cup\{G_m\mid m\in\NN\})$.

  The next base cases are when $\lvert H\rvert\geq 2$ and $H$ is prime. In this case, we show that
  $H$ must be an induced subgraph of $G_m$ for some $m\in\NN$.

  Recall that since $W^G$ is $\{0,1\}$-valued, the set $\Tind(H,W^G)\setminus\cD_{V(H)}$ is
  alternatively described as the set of pairs $(x,y)\in(\Omega^V)^{V(H)}\times
  [0,1)^{\binom{V(H)}{2}}$ such that $x$ is an embedding of $H$ in $W^G$.

  Fix then one such point $(x,y)$ and let $t\in\NN$ be the length of the longest string
  $\sigma\in\prod_{\ell=0}^{t-1} V(G_\ell)$ that is common to all coordinates of $x$, that is, for
  every $v\in V(H)$, we have $x_v\rest_{\{0,\ldots,t-1\}} = \sigma$ and there exist $v,w\in V(H)$
  such that $(x_v)_t \neq (x_w)_t$.

  For each $i\in V(G_t)$, let $U_i\df\{v\in V(H) \mid (x_v)_t = i\}$ and let $H_i\df
  H\rest_{U_i}$. Let also $I\df\{i\in V(G_t)\mid U_i\neq\varnothing\}$. The structure of $W^G$
  implies that $H$ is obtained from $G_t\rest_I$ by substituting each $i\in I$ by $H_i$. Since
  $\lvert H_i\rvert < \lvert H\rvert$ for every $i\in I$ and $H$ is prime, it follows that $\lvert
  G_t\rest_I\rvert = \lvert H\rvert$ and $\lvert U_i\rvert = 1$ for every $i\in I$, that is, the
  unique function $\beta\colon V(H)\rightarrowtail V(G_t)$ such that $v\in U_{\beta(v)}$ is an
  embedding of $H$ in $G_t$. Thus $H$ is an induced subgraph of $G_t$.

  \medskip

  We now consider the inductive step when $H$ is not prime. Then $H$ is of the form $F_1^{v\to F_2}$
  for some graphs $F_1,F_2$ and $v\in V(F_1)$ with $\lvert F_1\rvert,\lvert F_2\rvert < \lvert
  H\rvert$. By inductive hypothesis, we have $F_1,F_2\in S(\{K_0\}\cup\{G_m\mid m\in\NN\})$ and
  since this set is closed under substitutions we get $H\in S(\{K_0\}\cup\{G_m\mid m\in\NN\})$.
\end{proof}

\begin{lemma}\label{lem:PphiGQphiG}
  Let $G = (G_m)_{m\in\NN}$ be a sequence of finite graphs with $\lvert G_m\rvert\geq 2$ for every
  $m\in\NN$. Then $P(\phi_G^*) = Q(\phi_G^*) = S(\{K_0\}\cup\{G_m\mid m\in\NN\})$.
\end{lemma}

\begin{proof}
  By Lemma~\ref{lem:QphiG}, we know that $Q(\phi_G^*)\subseteq S(\{K_0\}\cup\{G_m\mid m\in\NN\})$
  and since $P(\phi_G^*)\subseteq Q(\phi_G^*)$, it is sufficient to prove that
  $S(\{K_0\}\cup\{G_m\mid m\in\NN\})\subseteq P(\phi_G^*)$.

  Let $H\in S(\{K_0\}\cup\{G_m\mid m\in\NN\})$ and let us show that $H\in P(\phi_G^*)$ by induction
  on $\lvert H\rvert$.

  The base case is when $H$ is a prime graph. By Lemma~\ref{lem:graph:primesubstructure}, we know
  there exists $\widehat{m}\in\NN$ such that $H$ is an induced subgraph of $G_{\widehat{m}}$, so it
  is sufficient to show that $G_{\widehat{m}}\in P(\phi_G^*)$. In turn, by
  Lemma~\ref{lem:PWsubgraphon}, it is sufficient to show that for every positive measure set
  $A\subseteq\Omega^V$, the graph $G_{\widehat{m}}$ has positive density in the subgraphon
  $W^{G^*}\rest_A$.

  Let $\epsilon$ be any positive number with $\epsilon < 1/\lvert G_{\widehat{m}}\rvert$. By
  Lemma~\ref{lem:KsigmaV}, there exists $t_0\in\NN$ such that for every $t\geq t_0$, there exists
  $\sigma^t\in\prod_{\ell=0}^{t-1} V(G_{m_\ell})$ such that $\nu^V(A\cap
  K_{\sigma^t,V})\geq(1-\epsilon)\cdot\nu^V(K_{\sigma^t,V})$.

  Let $t\in\NN$ be such that $t_0 < t\leq t_0 + 2^{\widehat{m}}$ and $m_t=\widehat{m}$ as provided
  by Remark~\ref{rmk:mell}. Let also
  \begin{align*}
    T
    & \df
    \left\{\tau\in\prod_{\ell=0}^t V(G_{m_\ell}) \;\middle\vert\;
    \tau\rest_{\{0,\ldots,t-1\}} = \sigma^t
    \right\}.
  \end{align*}
  Since $\{K_{\tau,V} \mid \tau\in T\}$ partitions $K_{\sigma^t,V}$ into $\lvert T\rvert = \lvert
  G_{m_t}\rvert = \lvert G_{\widehat{m}}\rvert$ parts of equal measure, it follows that for every
  $\tau\in T$, we have
  \begin{align*}
    \nu^V(A\cap K_{\tau,V})
    & \geq
    \left(
    1 - \epsilon - \frac{\lvert G_{\widehat{m}}\rvert - 1}{\lvert G_{\widehat{m}}\rvert}
    \right)
    \nu^V(K_{\sigma^t,V})
    >
    0.
  \end{align*}
  However, the definition of $W^{G^*}$ implies that if we pick $x_v\in K_{(\sigma^t,v),V}$ for each
  $v\in V(G_{\widehat{m}})$ (and pick any $y\in [0,1)^{\binom{V(G_{\widehat{m}})}{2}}$), then we get
  a copy of $G_{\widehat{m}}$ in $W^{G^*}$ and since for every $v\in G_{\widehat{m}}$, we have
  $\nu^V(A\cap K_{(\sigma^t,v),V}) > 0$ (as $(\sigma^t,v)\in T$), it follows that
  $G_{\widehat{m}}$ has positive density in $W^{G^*}\rest_A$.

  \medskip

  For the inductive step when $H$ is not prime, we must have $H = F_1^{v\to F_2}$ for some graphs
  $F_1,F_2$ and some $v\in V(F_1)$ with $\lvert F_1\rvert,\lvert F_2\rvert < \lvert H\rvert$. Since
  $F_1$ and $F_2$ are in $P(\phi_G^*)$ by inductive hypothesis and $P(\phi_G^*)$ is closed under
  substitutions by Lemma~\ref{lem:PW}, it follows that $H\in P(\phi_G^*)$.
\end{proof}

We can finally show Theorem~\ref{thm:graphpersistence} that says that a family $\cF$ of graphs with
at least one graph of size at least $2$ is strongly persistent if and only if it is persistent if
and only if it is closed under substitutions and under induced subgraphs.

\begin{proofof}{Theorem~\ref{thm:graphpersistence}}
  The
  implication~\ref{thm:graphpersistence:stronglypersistent}$\implies$\ref{thm:graphpersistence:persistent}
  is trivial: every strongly persistent family is obviously persistent.

  \medskip

  For the implication~\ref{thm:graphpersistence:persistent}$\implies$\ref{thm:graphpersistence:closed}, if
  $\cF = P(W)$ for some graphon $W$, then Lemma~\ref{lem:PW} implies that it is closed under
  substitutions and induced subgraphs.

  \medskip

  For the final
  implication~\ref{thm:graphpersistence:closed}$\implies$\ref{thm:graphpersistence:stronglypersistent},
  by Lemma~\ref{lem:graph:strongclosuresubst}, we have $\cF = S(\cP)$, where $\cP$ is the set of
  graphs in $\cF$ that are prime. Since $\cF$ contains at least one graph of size at least $2$,
  $\cP$ must also contain one such graph (since $S(\{K_0,K_1\}) = \{K_0,K_1\}$).

  Let $G = (G_m)_{m\in\NN}$ be an enumeration of all graphs in $\cP$ of size at least $2$
  (potentially with repetitions if $\cP$ is finite). Note that since $\cF = S(\cP)$ is closed under
  induced subgraphs, it follows that $\cF = S(\{K_0\}\cup\{G_m \mid m\in\NN\})$.

  By Lemma~\ref{lem:PphiGQphiG}, the repeating recursive blow-up $\phi_G^*$ relative to $G$ satisfies
  $P(\phi_G^*)=Q(\phi_G^*)=\cF$, hence $\cF$ is strongly persistent.
\end{proofof}

\section{Weak randomness in graphons}
\label{sec:WR:graph}

Recall that Theorem~\ref{thm:graphpersistence} characterizes all universal theories of graphs that
contain a universal weakly random graphon. In this section, we study a related natural question (see
Definition~\ref{def:WR} below): when does every graphon of a universal theory of graphs contain some
weakly random subgraphon? As mentioned in the introduction, this property is a generalization of
$\AEHP$ (see Definition~\ref{def:AEHPgraphs}). We also remind the reader that in this section we
drop the qualifiers ``weakly'' and ``strongly'' from ``closed under substitutions'' as they are
superfluous for graphs (see Remark~\ref{rmk:substarity2}).

\begin{definition}\label{def:WR}
  We say that a universal theory $T$ of graphs has the \emph{weakly random \Erdos--Hajnal property}
  (abbreviated $T\in\WR$) if every limit $W$ of $T$ contains a weakly random subgraphon.
\end{definition}

\begin{remark}
  Since trivial graphons are weakly random, we obviously have $\AEHP\subseteq\WR$. Furthermore, even
  though it is also natural to ask what is the class of theories that have \emph{some} weakly random
  limit, it is clear that this is precisely the set of non-degenerate universal theories. This is
  because Ramsey's Theorem implies that any non-degenerate theory $T$ of graphs must either contain
  arbitrarily large cliques, in which case $W\equiv 1$ is a limit of $T$, or contain arbitrarily
  large anti-cliques, in which case $W\equiv 0$ is a limit of $T$.
\end{remark}

Similarly to Lemma~\ref{lem:PWsubgraphon}, the following lemma is a simple but very powerful
observation about weakly random subgraphons.

\begin{lemma}\label{lem:weaklyrandomsubgraphon}
  A graphon $W$ over a space $\Omega=(X,\cA,\mu)$ contains a weakly random subgraphon $W'$ if and
  only if there exists a positive measure set $A\subseteq X$ such that $W\rest_A$ is weakly random.
\end{lemma}

\begin{proof}
  The backward implication follows because $W\rest_A$ is a subgraphon of $W$.

  \medskip

  For the forward implication, we know that $\phi_{W'} = \phi_{W\rest_f}$ for some measurable
  function $f\colon X\to [0,1]$. Let $A\df\{x\in X \mid f(x) > 0\}$. We claim that $W\rest_A$ is
  weakly random. Indeed, this follows from Lemma~\ref{lem:PWsubgraphon} and since for every
  $B\subseteq X$ of positive $\mu_f$-measure, we have $Q(W\rest_A\rest_B) = Q(W\rest_f\rest_B)$.
\end{proof}

Our next main objective is to characterize the class $\WR$ under the assumption that the set of
graphs of the theory is closed under substitutions.

\begin{theorem}\label{thm:graphs:WR}
  Let $T$ be a universal theory of graphs such that $\cM[T]$ is closed under substitutions. Then
  $T\in\WR$ if and only if $\cM[T]$ is primally almost finite.
\end{theorem}

Before we prove Theorem~\ref{thm:graphs:WR} above, let us observe a simple corollary of it.

\begin{corollary}
  There exists a universal theory of graphs $T$ with $\cM[T]\subsetneq\cM[\TGraph]$ and $T\notin\WR$.
\end{corollary}

\begin{proof}
  The family $\{C_n \mid n\geq 5\}$ of cycles of length at least $5$ is a family of prime graphs
  that is not almost finite (see Example~\ref{ex:substitutiongeneration}). Let
  $\cF=S(\{K_0\}\cup\{C_n \mid n\geq 5\})$ and since $\cF$ is closed under substitutions and induced
  subgraphs but is not primally almost finite, the universal theory $T\df\Th(\cF)$ with
  $\cM[\Th(\cF)]=\cF$ satisfies $T\notin\WR$ by Theorem~\ref{thm:graphs:WR}. It is also easy to see
  that $\cM[T]\subsetneq\cM[\TGraph]$, as for example the prime graph $G_6$ of
  Example~\ref{ex:substitutiongeneration} is not in $\cM[T]$ by
  Lemma~\ref{lem:graph:primesubstructure}.
\end{proof}

We start by proving the easier direction of Theorem~\ref{thm:graphs:WR} in the lemma below. In fact,
for this direction, we do not even need $\cM[T]$ to be closed under substitutions.

\begin{lemma}\label{lem:graphs:primallyalmostfinite->WR}
  If $T$ is a universal theory of graphs such that $\cM[T]$ is primally almost finite, then
  $T\in\WR$.
\end{lemma}

\begin{proof}
  We prove the lemma by its contra-positive. Suppose $T\notin\WR$ and let us show that the set $\cP$
  of graphs of $T$ that are prime is not almost finite. By Lemma~\ref{lem:graph:almostfinite}, it is
  sufficient to construct a sequence $(R_n)_{n\in\NN}$ in $\cP$ such that for every $n,m\in\NN$, if
  $n < m$, then $R_n$ is not an induced subgraph of $R_m$.

  Since $T\notin\WR$, there must exist a limit $\phi\in\HomT{T}$ of $T$ that does not contain any
  weakly random sub-object.

  We now construct sequences $(\phi_n)_{n\in\NN}$ of sub-objects of $\phi$ and $(R_n)_{n\in\NN}$ of
  prime graphs in $\cM[T]$ satisfying:
  \begin{enumerate}
  \item For every $n\in\NN$, $\phi_{n+1}$ is a sub-object of $\phi_n$.
  \item For every $n\in\NN$, $R_n\in Q(\phi_n)\setminus Q(\phi_{n+1})$.
  \end{enumerate}
  We construct these sequences inductively as follows.
  \begin{enumerate}[label={\arabic*.}]
  \item Set $\phi_0\df\phi$.
  \item Given $\phi_n\in\HomT{T}$, since $\phi_n$ is a sub-object of $\phi$, we know that $\phi_n$
    is not weakly random, so there exists $G_n\in Q(\phi_n)\setminus P(\phi_n)$. Let $\cP_n$ be the
    set of induced subgraphs of $G_n$ that are prime. Since by Lemma~\ref{lem:PW}, $P(\phi_n)$ is
    closed under substitutions and $G_n\in S(\cP_n)$, there exists $R_n\in\cP_n\setminus P(\phi_n)$
    and since $Q(\phi_n)$ is closed under induced subgraphs, we get $R_n\in Q(\phi_n)\setminus
    P(\phi_n)$. From the definition of $P(\phi_n)$, it follows that there exists a sub-object
    $\phi_{n+1}$ of $\phi_n$ (hence $\phi_{n+1}$ is also a sub-object of $\phi$) such that $R_n\in
    Q(\phi_n)\setminus Q(\phi_{n+1})$.
  \end{enumerate}

  Let now $n,m\in\NN$ be such that $n < m$. By induction, we know that $\phi_m$ is a sub-object of
  $\phi_{n+1}$, so $Q(\phi_m)\subseteq Q(\phi_{n+1})$, which in turn implies that $R_n\in
  Q(\phi_n)\setminus Q(\phi_m)$. Since $R_m\in Q(\phi_m)$ and $Q(\phi_m)$ is closed under induced
  subgraphs, it follows that $R_n$ is not an induced subgraph of $R_m$, concluding the proof.
\end{proof}

For the other side of the characterization of $\WR$, the proposition below shows that under
appropriate hypotheses, the recursive blow-up $\phi_R$ of Definition~\ref{def:graphrecursiveblowup}
is a graphon without any weakly random subgraphon.

\begin{proposition}\label{prop:noweaklyrandomsubgraphon}
  Let $R = (R_n)_{n\in\NN}$ be a sequence of prime graphs of size at least $2$ such that for each
  $n\in\NN$, there exist at most finitely many $m\in\NN$ such that $R_n$ is an induced subgraph of
  $R_m$. Suppose also that $\prod_{n\in\NN} (1 - 1/\lvert R_n\rvert) = 0$. Then $\phi_R$ does not
  contain any weakly random sub-object.
\end{proposition}

Let us first give some intuition on the proof of
Proposition~\ref{prop:noweaklyrandomsubgraphon}. First, note that since all $R_n$ are prime graphs
and $\phi_R$ is obtained via a limit of recursive blow-ups, which themselves are obtained from the
$R_n$ via substitutions, it follows that copies of $R_n$ in $\phi_R$ need to correspond to copies of
$R_n$ inside some $R_m$. The condition that each $R_n$ is contained in at most finitely many $R_m$
then ensures that the restriction of $W^R$ to basic clopen sets $K_{\sigma,V}$
(see~\eqref{eq:KsigmaV}) with $\lvert\sigma\rvert$ large enough do not have any copies of $R_n$.
Thus, for every positive measure set $A$, there is some $K_{\sigma,V}$ such that $R_n\notin
Q(W^R\rest_{A\cap K_{\sigma,V}})$ and $A\cap K_{\sigma,V}$ has positive measure. However, to use
this fact show that $\phi_R$ does not contain any weakly random sub-object, we need to also ensure
that every positive measure set $A$ contains at least one $R_n$ (with $n$ depending on $A$), so that
we conclude that $Q(W^R\rest_A)\neq P(W^R\rest_A)$ since the above argument gives $R_n\in
Q(W^R\rest_A)\setminus Q(W^R\rest_{A\cap K_{\sigma,V}})$. This is where the condition
$\prod_{n\in\NN} (1 - 1/\lvert R_n\rvert) = 0$ comes in: we will show that any set $A$ avoiding all
$R_n$ has measure at most $\prod_{n\in\NN} (1 - 1/\lvert R_n\rvert)$.

\begin{proof}
  For every $t\in\NN$, let $R^t$ be the shifted sequence $(R_{n+t})_{n\in\NN}$.

  Also, for each $t\in\NN$, let $m_t$ be the maximum $m\in\NN$ such that $R_t$ is an induced
  subgraph of $R_m$. Note that for every $t\in\NN$, by
  Lemmas~\ref{lem:graph:substitutionsubstructure} and~\ref{lem:QphiG}, we have $R_t\in
  Q(\phi_{R^t})\setminus Q(\phi_{R^{m_t+1}})$ since $R_t$ is prime and is not an induced subgraph of
  any $R_{t'}$ with $t' > m_t$\footnote{In fact, since $\phi_{R^{m_t+1}}$ is a sub-object of
    $\phi_{R^t}$, we have $Q(\phi_{R^{m_t+1}})\subsetneq Q(\phi_{R^t})$.}.

  To show that $\phi_R$ does not contain a weakly random sub-object, by
  Lemma~\ref{lem:weaklyrandomsubgraphon}, it is sufficient to show that for every positive measure
  set $A\subseteq\Omega^V$, the subgraphon $W^R\rest_A$ is not weakly random.

  We claim that there exists $t\in\NN$ such that $R_t\in Q(W^R\rest_A)$. Suppose not, let $n\in\NN$
  be large enough so that $\prod_{\ell=0}^{n-1} (1-1/\lvert R_\ell\rvert) < \nu^V(A)$ (recall from
  Definition~\ref{def:graphrecursiveblowup} that $\nu^V$ is the measure in the underlying space of
  $W^R$) and consider the set
  \begin{align*}
    \Sigma
    & \df
    \left\{\sigma\in\prod_{\ell=0}^{n-1} V(R_\ell) \;\middle\vert\;
    \nu^V(A\cap K_{\sigma,V}) > 0
    \right\}.
  \end{align*}
  We claim that for every $m\in\{0,\ldots,n-2\}$ and every $\tau\in\prod_{\ell=0}^{m-1} V(R_\ell)$,
  there exists $v_\tau\in V(R_m)$ such that $(\tau,v_\tau)$ is not a prefix of any element of
  $\Sigma$. Indeed, otherwise, since mapping each $v\in V(R_m)$ to an element of $K_{(\tau,v),V}$
  gives an embedding of $R_m$ in $W^R$, we would get
  \begin{align*}
    \tind(R_m,W^R\rest_A) & \geq \prod_{v\in V(R_m)} \frac{\nu^V(A\cap K_{(\tau,v),V})}{\nu^V(A)} > 0.
  \end{align*}
  Thus, the existence of $v_\tau$ is proved.

  Let then $\Sigma^*$ be the set of $\sigma\in\prod_{\ell=0}^{n-1} V(R_\ell)$ such that for every
  $m\in\{0,\ldots,n-2\}$, we have $v_{\sigma\rest_{\{0,\ldots,m-1\}}}\neq\sigma_m$. Our last claim
  says that $\Sigma\subseteq\Sigma^*$. Now it is easy to see that
  \begin{align*}
    \nu^V(A)
    & =
    \sum_{\sigma\in\Sigma} \nu^V(A\cap K_{\sigma,V})
    \leq
    \sum_{\sigma\in\Sigma^*} \nu^V(K_{\sigma,V})
    =
    \prod_{\ell=0}^{n-1}\left(1 - \frac{1}{\lvert R_\ell\rvert}\right)
    <
    \nu^V(A),
  \end{align*}
  a contradiction. This concludes the proof that there exists $t\in\NN$ such that $R_t\in
  Q(W^R\rest_A)$.

  We will now show that $W^R\rest_A$ is not weakly random by showing that there exists a sub-object
  of $W^R\rest_A$ in which $R_t$ has density zero (so that we conclude $P(W^R\rest_A)\subsetneq
  Q(W^R\rest_A)$ as $R_t$ is in the latter set but not in the former set).

  Since $\{K_{\sigma,V} \mid \sigma\in\prod_{\ell=0}^{m_t} V(R_\ell)\}$ partitions the space
  $\Omega^V$, there must exist $\sigma\in\prod_{\ell=0}^{m_t} V(R_\ell)$ such that $\nu^V(A\cap
  K_{\sigma,V}) > 0$ but note that $\phi_{W^R\rest_{K_{\sigma,V}}} = \phi_{R^{m_t+1}}$ and since
  $Q(W\rest_{A\cap K_{\sigma,V}})\subseteq Q(W\rest_{K_{\sigma,V}})=Q(\phi_{R^{m_t+1}})$ it follows
  that $R_t\notin Q(W\rest_{A\cap K_{\sigma,V}})$ as desired.

  Therefore $\phi_R$ does not contain any weakly random sub-object.
\end{proof}

\begin{remark}\label{rmk:recCsquares}
  The product condition $\prod_{n\in\NN} (1 - 1/\lvert R_n\rvert) = 0$ in
  Proposition~\ref{prop:noweaklyrandomsubgraphon} may seem very unnatural at first. However, it is
  easy to see that it is necessary for $\phi_R$ to not contain any weakly random sub-object: for
  example, consider the limit $\phi_R$ for the sequence $R\df (C_{n^2+5})_{n\in\NN}$ (see
  Figure~\ref{fig:recCsquares}) and fixing $v_n\in V(C_{n^2+5})$ for each $n\in\NN$, let
  \begin{align*}
    A
    & \df
    \prod_{n\in\NN} (V(C_{n^2+5})\setminus\{v_n\}).
  \end{align*}

  \begin{figure}[htbp]
    \begingroup
\def\order{3}
\def\side{15}
\def\axis{15.5}

\begin{center}
  \begin{tikzpicture}
    \draw (0,0) -- (\side,0) -- (\side,\side) -- (0,\side) -- cycle;
    \draw[->] (0,0) -- (\axis,0);
    \draw[->] (0,0) -- (0,\axis);

    \node[below right] at (\axis,0) {$x$};
    \node[above left] at (0,\axis) {$y$};

    \foreach \i [%
      remember=\rep as \prevrep (initially 1),
      evaluate=\i as \rep using \prevrep * (((\i-1) * (\i-1)) + 5)
      ] in {1,...,\order}{
      \pgfmathsetmacro{\step}{1/\prevrep}
      \pgfmathsetmacro{\baseside}{\side/\rep}
      \pgfmathtruncatemacro{\clength}{(\i-1) * (\i-1) + 5}
      \pgfmathtruncatemacro{\clengthmo}{\clength-1}

      \foreach \j [%
        evaluate=\j as \p using (\j-1) * \step * \side
        ] in {1,...,\prevrep}{
        \foreach \k [%
          remember=\cp as \pp (initially \p),
          evaluate=\k as \cp using \pp + \baseside
          ] in {1,...,\clengthmo}{
          \fill (\pp,\cp) -- ++(\baseside,0) -- ++(0,\baseside) -- ++(-\baseside,0) -- cycle;
          \fill (\cp,\pp) -- ++(\baseside,0) -- ++(0,\baseside) -- ++(-\baseside,0) -- cycle;
        }
        \pgfmathsetmacro{\lastp}{\p + \clengthmo * \baseside}
        \fill (\lastp,\p) -- ++(\baseside,0) -- ++(0,\baseside) -- ++(-\baseside,0) -- cycle;
        \fill (\p,\lastp) -- ++(\baseside,0) -- ++(0,\baseside) -- ++(-\baseside,0) -- cycle;
      }
    }
  \end{tikzpicture}

  \caption{Approximation of a graphon $W$ over $[0,1]$ representing the limit $\phi_R$ of recursive
    blow-ups corresponding to the sequence $R\df(C_{n^2+5})_{n\in\NN}$ of
    Remark~\ref{rmk:recCsquares}. The graphon $W$ has a fractal structure, whose first
    $\order$ steps are represented in the picture.}
  \label{fig:recCsquares}
\end{center}
\endgroup

  \end{figure}
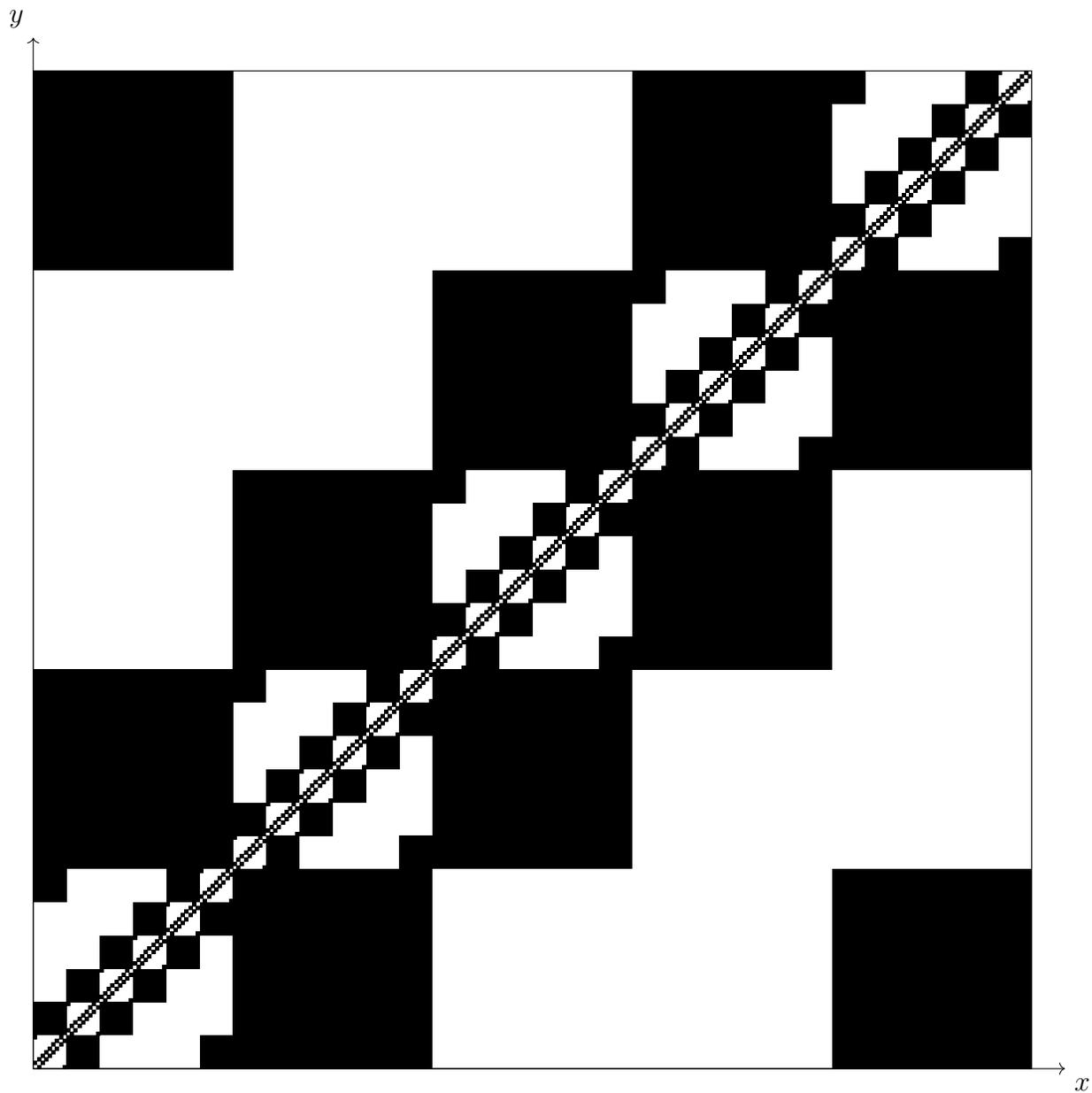

  Note that $\nu^V(A) = \prod_{n\in\NN}(1 - 1/\lvert C_{n^2+5}\rvert) > 0$. On the other hand, since
  $C_{n^2+5}-v_n\cong P_{n^2+4}$ is the path with $n^2+4$ vertices, it is obvious that $W^R\rest_A$
  represents the same limit as $W^{R'}$ for the sequence $R'=(P_{n^2+4})_{n\in\NN}$. In turn, by
  Lemma~\ref{lem:QphiG}, we have
  \begin{align*}
    Q(W^{R'})\subseteq S(\{K_0\}\cup\{P_{n^2+4} \mid n\in\NN\}) = S(\{P_n\mid n\in\NN\})
  \end{align*}
  and since the family above is primally almost finite,
  Lemma~\ref{lem:graphs:primallyalmostfinite->WR} implies that $W^{R'}$ contains a weakly random
  subgraphon, hence so does $W^R$. In fact, with a bit more effort, one can also show that $W^{R'}$
  itself is already weakly random, but we omit this proof. We will also see in
  Proposition~\ref{prop:WRproductcond} that not only does $\phi_R$ contain a weakly random
  sub-object, but we also have $\Th(\phi_R)\in\WR$.
\end{remark}

We can now prove Theorem~\ref{thm:graphs:WR} that says that a universal theory of graphs $T$ with
$\cM[T]$ closed under substitutions and with $\cM_2[T]$ non-empty is in $\WR$ if and only if
$\cM[T]$ is primally almost finite.

\begin{proofof}{Theorem~\ref{thm:graphs:WR}}
  For the backward direction, if $\cM[T]$ is primally almost finite, then by
  Lemma~\ref{lem:graphs:primallyalmostfinite->WR}, we have $T\in\WR$.

  \medskip

  We prove the forward direction by the contra-positive: suppose $\cM[T]$ is not primally almost
  finite, so there exists an infinite antichain $\{G_n \mid n\in\NN\}$ of prime graphs of $T$, and
  without loss of generality, suppose every $G_n$ has size at least $2$.

  For each $n\in\NN$, let $r_n\in\NN_+$ be large enough so that $(1 - 1/\lvert G_n\rvert)^{r_n} \leq
  1/2$ and for each $\ell\in\NN$, let $R_\ell\df G_n$ for the unique $n\in\NN$ such that
  $\sum_{m=0}^{n-1} r_m\leq\ell < \sum_{m=0}^n r_m$. Clearly, for each $\ell\in\NN$, there exist
  exactly $r_\ell$ values of $t\in\NN$ such that $R_\ell$ is an induced subgraph of $R_t$. On the
  other hand, we have
  \begin{align*}
    \prod_{\ell\in\NN} \left(1 - \frac{1}{\lvert R_\ell\rvert}\right)
    & =
    \prod_{n\in\NN} \left(1 - \frac{1}{\lvert G_n\rvert}\right)^{r_n}
    \leq
    \prod_{n\in\NN} \frac{1}{2}
    =
    0.
  \end{align*}
  By Proposition~\ref{prop:noweaklyrandomsubgraphon}, we know that $\phi_R$ does not contain any
  weakly random sub-object and by Lemma~\ref{lem:QphiG}, we know that $Q(\phi_R)\subseteq S(\{R_\ell
  \mid \ell\in\NN\})\subseteq\cM[T]$, so $\phi_R$ is a limit of $T$ without any weakly random
  sub-object.
\end{proofof}

We conclude this section with some natural examples of universal theories in $\WR$ and not in
$\WR$. We start by showing that the universal theory of induced subgraphs of recursive blow-ups of
$C_4$ studied in~\cite[\S 8]{CM22} (see Example~\ref{ex:recC4} and Figure~\ref{fig:recC4}) is the
simplest example in $\WR\setminus\AEHP$.

\begin{proposition}\label{prop:recC4}
  The limit recursive blow-up $\phi_{C_4}$ of $C_4$ is weakly random. In particular, the theory $T$
  of induced subgraphs of the recursive blow-ups of $C_4$ satisfies $T\in\WR\setminus\AEHP$.
\end{proposition}

\begin{proof}
  Recall from Example~\ref{ex:recC4} that the limit $\phi_{C_4}$ recursive blow-up of $C_4$ can be
  viewed as the repeating recursive blow-up $\phi_{G''}^*$ for the sequence
  $G''=(K_2,\overline{K}_2,\overline{K}_2,\ldots)$ whose first element is $K_2$ and all other
  elements are $\overline{K}_2$.

  There are two ways of seeing that $\phi_{C_4}$ is weakly random. The first is using
  Lemma~\ref{lem:PphiGQphiG} to conclude that $P(\phi_{C_4})=Q(\phi_{C_4}) =
  S(\{K_0,K_2,\overline{K}_2\})$. Alternatively, the result follows directly from the results
  of~\cite{CM22} and Lemma~\ref{lem:PW}: by~\cite[Lemma~8.7]{CM22}, we know that
  $Q(\phi_{C_4})\subseteq S(\{K_0,K_2,\overline{K}_2\})$, so Lemma~\ref{lem:PW} implies that
  $P(\phi_{C_4})$ can only be one of $S(\{K_0,K_2\})$, $S(\{K_0,\overline{K}_2\})$ or
  $S(\{K_0,K_2,\overline{K}_2\})$ and since by~\cite[Lemma~8.8]{CM22} does not contain trivial
  subgraphons, the first two cases are ruled out, so
  $P(\phi_{C_4})=Q(\phi_{C_4})=S(\{K_0,K_2,\overline{K}_2\})$.

  \medskip

  Since the family of induced subgraphs of recursive blow-ups of $C_4$ is precisely the family
  $S(\{K_0,K_2,\overline{K}_2\})$, which is primally finite, the fact that $T\in\WR$ follows from
  Theorem~\ref{thm:graphs:WR}. On the other hand, $\phi_{C_4}$ does not contain trivial subgraphons
  (this follows directly from~\cite[Lemma~8.8]{CM22} or alternatively from the fact that a trivial
  subgraphon $W$ must have $Q(W)$ either equal to $S(\{K_0,K_2\})$ or $S(\{K_0,\overline{K}_2\})$),
  hence $T\notin\AEHP$.
\end{proof}

\begin{proposition}\label{prop:perfectgraphtheory}
  The theory $\TPerfect$ of perfect graphs is not in $\WR$. Furthermore, the set $\cM[\TPerfect]$ is
  closed under substitutions.
\end{proposition}

\begin{proof}
  We first show that $\cM[\TPerfect]$ is closed under substitutions. By the Strong Perfect Graph
  Theorem~\cite{CRST06}, we know that a graph $G$ is perfect if and only if both $G$ and its
  complement $\overline{G}$ do not contain any induced odd-cycle of length at least $5$.

  Let us show that if $F_1,F_2$ are perfect graphs and $v\in V(F_1)$, then $F_1^{v\to F_2}$ is also
  a perfect graph. Since $\overline{F_1^{v\to F_2}}\cong(\overline{F_1})^{v\to\overline{F_2}}$, it
  is sufficient to show that $F_1^{v\to F_2}$ does not contain any induced odd-cycles of length at
  least $5$.

  Without loss of generality, let us suppose $V(F_1)\cap V(F_2)=\varnothing$. Suppose toward a
  contradiction that $v_1,\ldots,v_{2\ell+1}$ forms an induced odd-cycle of $F_1^{v\to F_2}$ with
  $\ell\geq 2$. Since both $F_1$ and $F_2$ are perfect, this odd-cycle must contain both vertices of
  $F_1$ (that are not $v$) and $F_2$. Without loss of generality, suppose $v_i\in V(F_2)$ for every
  $i\in[k]$ for some $k\in [2\ell+1]$ and $v_{k+1}\in V(F_1)$. Since $v_{k+1}\in V(F_1)$ is adjacent
  to $v_k\in V(F_2)$, it follows from the structure of $F_1^{v\to F_2}$ that $v_{k+1}$ is adjacent
  to all of $v_1,\ldots,v_k$, but since the cycle is induced, this can only happen if $k=2$ and
  $2\ell+1=3$, a contradiction. Therefore, $\cM[\TPerfect]$ is closed under substitutions.

  \medskip

  By Theorem~\ref{thm:graphs:WR}, to show that $\TPerfect\notin\WR$, it is
  sufficient to show that $\cM[\TPerfect]$ is not primally almost finite. But recall that the family
  of graphs $\{G_n \mid n\geq 6\}$ of Example~\ref{ex:substitutiongeneration} is a family of prime
  graphs that is not almost finite and since these graphs are bipartite, they are also perfect.
\end{proof}

Finally, we consider the theory $\TPermGraph\df I(\TPerm)$ of graphs of agreements of permutations,
where $I\colon\TGraph\leadsto\TPerm$ is given by
\begin{align*}
  I(E)(x,y) & \df (x\neq y\land (x\prec_1 y\tot x\prec_2 y)),
\end{align*}

The next proposition provides a natural universal weakly random limit of $\TPermGraph$ as the
graphon of agreements of the quasirandom permuton (see Figure~\ref{fig:agreementsgraphon}). However,
we defer its proof to Section~\ref{sec:pers:univ} as it will follow as an easy consequence of
naturality of weak randomness (Proposition~\ref{prop:naturality}\ref{prop:naturality:WR}) and the
fact that the quasirandom permuton is a universal weakly random limit of $\TPerm$
(Proposition~\ref{prop:QRpermuton}).

\begin{proposition}\label{prop:agreementsofQRpermuton}
  The graphon $W$ over $[0,1]^2$ of agreements of the quasirandom permuton given by
  \begin{align*}
    W(x,y) & \df \One[\pi_1(x) < \pi_1(y) \tot \pi_2(x) < \pi_2(y)],
  \end{align*}
  where $\pi_i\colon[0,1]^2\to[0,1]$ is the projection onto the $i$th coordinate, is a universal
  weakly random limit of $\TPermGraph$.
\end{proposition}

\begin{figure}[htb]
  \begingroup
\def\order{3}
\def\side{5}
\def\axis{5.5}

\begin{center}
  \begin{tikzpicture}
    \draw (0,0) -- (\side,0) -- (\side,\side) -- (0,\side) -- cycle;
    \draw[->] (0,0) -- (\axis,0);
    \draw[->] (0,0) -- (0,\axis);

    \node[below right] at (\axis,0) {$x$};
    \node[above left] at (0,\axis) {$y$};
    
    \pgfmathsetmacro{\smallside}{2^(-2*\order) * \side}
    \def\strings{v//0/0,}
    \foreach \i in {1,...,\order}{
      \let\oldstrings\strings
      \def\strings{}
      \foreach \v/\s/\done/\dtwo in \oldstrings {
        \ifx\v\empty\relax
        \else
        \pgfmathsetmacro{\newdone}{\done + 2^(-2*\i+1) * \side}
        \pgfmathsetmacro{\newdtwo}{\dtwo + 2^(-2*\i) * \side}
        \edef\newstrings{\strings v/\s0/\done/\dtwo,v/\s1/\newdone/\newdtwo,}
        \global\let\strings\newstrings
        \fi
      }
    }

    \foreach \v/\s/\done/\dtwo in \strings {
      \ifx\v\empty\relax
      \else
      \foreach \V/\S/\Done/\Dtwo in \strings {
        \ifx\V\empty\relax
        \else
        \compareStringsDo{\s}{\S}{
          \foreach \w/\t/\eone/\etwo in \strings {
            \ifx\w\empty\relax
            \else
            \foreach \W/\T/\Eone/\Etwo in \strings {
              \ifx\W\empty\relax
              \else
              \compareStringsDo{\t}{\T}{
                \fill ($(\done,\Done) + (\etwo,\Etwo)$) -- ++ (\smallside,0) -- ++ (0,\smallside)
                -- ++ (-\smallside,0) -- cycle;
              }{
              }{
              }
              \fi
            }
            \fi
          }
        }{
        }{
          \foreach \w/\t/\eone/\etwo in \strings {
            \ifx\w\empty\relax
            \else
            \foreach \W/\T/\Eone/\Etwo in \strings {
              \ifx\W\empty\relax
              \else
              \compareStringsDo{\t}{\T}{
              }{
              }{
                \fill ($(\done,\Done) + (\etwo,\Etwo)$) -- ++ (\smallside,0) -- ++ (0,\smallside)
                -- ++ (-\smallside,0) -- cycle;
              }
              \fi
            }
            \fi
          }
        }
        \fi
      }
      \fi
    }
  \end{tikzpicture}

  \caption{Approximation of a graphon $W'$ over $[0,1]$ representing the same limit as the graphon
    $W$ of agreements of the quasirandom permuton of
    Proposition~\ref{prop:agreementsofQRpermuton}. The graphon $W'$ is indirectly defined by
    $W'(F(x),F(y)) = W(x,y)$ for the standard measure-isomorphism $F$ from $[0,1]$ to $[0,1]^2$ that
    maps the point $(x,y)\in[0,1]^2$ to $\sum_{i\in\NN_+} 4^{-i}\cdot(2\cdot x_i + y_i)$, where the
    binary expansions of $x$ and $y$ are $0.x_1x_2\cdots$ and $0.y_1y_2\cdots$, respectively. The
    graphon $W'$ has a fractal structure, whose first $\order$ steps are represented in the
    picture.}
  \label{fig:agreementsgraphon}
\end{center}
\endgroup

\end{figure}
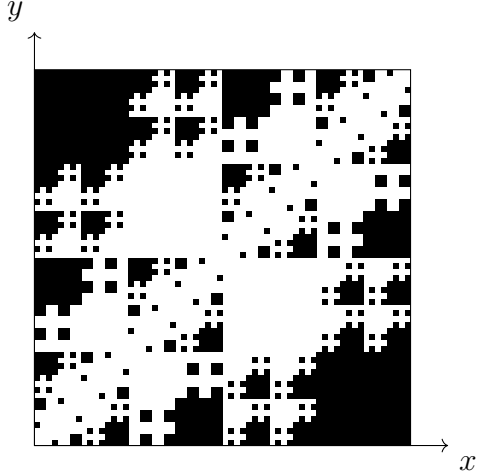

\begin{proposition}\label{prop:agreementsofpermutationtheory}
  The theory $\TPermGraph$ of graphs of agreements of permutations is not in $\WR$. Furthermore,
  $\cM[\TPermGraph]$ is closed under substitutions.
\end{proposition}

\begin{proof}
  First let us prove that $\cM[\TPermGraph]$ is closed under substitutions. Let
  $F,G\in\cM[\TPermGraph]$, let $v\in V(F)$ and without loss of generality, suppose $V(F) = [n]$ and
  $V(G)=[m]$ for some $n,m\in\NN$ and $\sigma\in S_n$ and $\tau\in S_m$ are permutations
  representing $F$ and $G$ with $\{i,j\}\in E(F)$ if and only if $i < j\tot\sigma(i)\leq\sigma(j)$
  and analogously for $G$ and $\tau$.

  It is now easy to check that $F^{v\to G}$ is the graph of agreements of the permutation $\pi\in
  S_{n+m-1}$ defined by
  \begin{align*}
    \pi(i) & \df
    \begin{dcases*}
      \sigma(i), & if $i < v$ and $\sigma(i) < \sigma(v)$,\\
      \sigma(i) + m - 1, & if $i < v$ and $\sigma(v) < \sigma(i)$,\\
      \sigma(v) + \tau(i-v+1) - 1, & if $v\leq i < v + m$,\\
      \sigma(i-m+1), & if $v+m\leq i$ and $\sigma(i-m+1) < \sigma(v)$,\\
      \sigma(i-m+1) + m - 1, & if $v+m\leq i$ and $\sigma(v) < \sigma(i-m+1)$.
    \end{dcases*}
  \end{align*}
  In fact, the above shows that $\cM[\TPerm]$ is weakly closed under substitutions, so
  $\cM[\TPermGraph]$ inherits this property.

  Now, by Theorem~\ref{thm:graphs:WR}, it is sufficient to show that $\cM[\TPermGraph]$ is not
  primally almost finite.

  Recall that the family $\{G_n \mid n\geq 6\}$ of Example~\ref{ex:substitutiongeneration} is a
  family of prime graphs that is not almost finite. We claim that for every even\footnote{It is also
    true for odd $n$, but we only need an infinite subfamily, so even $n$ suffices.} $n\geq 6$, the
  graph $G_n$ is a graph of agreements of some permutation. Indeed, $G_n$ is the graph of agreements
  of the permutation $\pi_n\in S_{n+4}$ (see Figure~\ref{fig:primegraphofagreements}) given by
  \begin{align*}
    \pi_n(i)
    & \df
    \begin{dcases*}
      n+3, & if $i=1$,\\
      n+1, & if $i=2$,\\
      n-1, & if $i=3$,\\
      n+4, & if $i=4$,\\
      n-i+3, & if $6\leq i\leq n$ and $i$ is even,\\
      n-i+7, & if $5\leq i\leq n-1$ and $i$ is odd,\\
      1, & if $i=n+1$,\\
      6, & if $i=n+2$,\\
      4, & if $i=n+3$,\\
      2, & if $i=n+4$.
    \end{dcases*}
  \end{align*}
  For example, the values of $\pi_{14}$ (in sequence) are
  \begin{align*}
    17, 15, 13, 18, 16, 11, 14, 9, 12, 7, 10, 5, 8, 3, 1, 6, 4, 2.
  \end{align*}

  Thus, $\cM[\TPermGraph]$ is not primally almost finite, hence $\TPermGraph\notin\WR$ by
  Theorem~\ref{thm:graphs:WR}.
\end{proof}

\begin{figure}[htbp]
  \begingroup
\def\n{14}
\def\basedist{0.75}
\def\ptsize{2pt}
\def\ticsize{0.2}

\begin{center}
  \begin{tikzpicture}
    \pgfmathtruncatemacro{\nmfour}{\n-4}
    \foreach \v in {4,6,...,\nmfour}{
      \pgfmathtruncatemacro{\i}{\v+3}
      \pgfmathtruncatemacro{\j}{\n-\i+7}

      \pgfmathsetmacro{\x}{\i * \basedist}
      \pgfmathsetmacro{\y}{\j * \basedist}
      \coordinate (P\v) at (\x,\y);

      \node[above right] at (P\v) {$\v$};
      \fill (P\v) circle (\ptsize);
    }

    \pgfmathtruncatemacro{\nmthree}{\n-3}
    \foreach \v in {5,7,...,\nmthree}{
      \pgfmathtruncatemacro{\i}{\v+1}
      \pgfmathtruncatemacro{\j}{\n-\i+3}

      \pgfmathsetmacro{\x}{\i * \basedist}
      \pgfmathsetmacro{\y}{\j * \basedist}
      \coordinate (P\v) at (\x,\y);

      \node[below left] at (P\v) {$\v$};
      \fill (P\v) circle (\ptsize);
    }

    \pgfmathtruncatemacro{\nmtwo}{\n-2}
    \pgfmathtruncatemacro{\nmone}{\n-1}
    \pgfmathtruncatemacro{\npone}{\n+1}
    \pgfmathtruncatemacro{\nptwo}{\n+2}
    \pgfmathtruncatemacro{\npthree}{\n+3}
    \pgfmathtruncatemacro{\npfour}{\n+4}

    \foreach \v/\i/\j/\pos in {%
      1/1/\npthree/below left,
      2/4/\npfour/above right,
      3/3/\nmone/below left,
      \nmtwo/\nptwo/6/above right,
      \nmone/\npone/1/below left,
      \n/\npfour/2/above right,
      a/2/\npone/below left,
      b/5/\nptwo/above right,
      c/\n/3/below left,
      d/\npthree/4/above right%
    }{%
      \pgfmathsetmacro{\x}{\i * \basedist}
      \pgfmathsetmacro{\y}{\j * \basedist}
      \coordinate (P\v) at (\x,\y);

      \node[\pos] at (P\v) {$\v$};
      \fill (P\v) circle (\ptsize);
    }

    \foreach \i [remember=\i as \pi (initially 1)] in {2,...,\n}{%
      \draw (P\pi) -- (P\i);
    }

    \draw (P2) -- (Pa) -- (Pb) -- (P3);
    \draw (P\nmtwo) -- (Pc) -- (Pd) -- (P\nmone);

    \pgfmathsetmacro{\x}{(\n+4.5) * \basedist}
    \draw[->] (0,0) -- (\x,0);
    \node[below right] at (\x,0) {$i$};
    \draw[->] (0,0) -- (0,\x);
    \node[above left] at (0,\x) {$\pi_{\n}(i)$};

    \foreach \i in {1,...,\npfour}{
      \pgfmathsetmacro{\x}{\i * \basedist}
      \draw (\x,0) -- (\x,\ticsize);
      \draw (0,\x) -- (\ticsize,\x);
    }
  \end{tikzpicture}
  
  \caption{Graph of permutation $\pi_{\n}$ of proof of
    Proposition~\ref{prop:agreementsofpermutationtheory} represented as points (i.e., the set
    $\{(i,\pi_{\n}(i)) \mid i\in[\n]\}$). The edges of the corresponding graph of agreements
    $G_{\n}$ are represented as lines and the labels indicate the vertices of $G_{\n}$.}
  \label{fig:primegraphofagreements}
\end{center}
\endgroup
  
\end{figure}
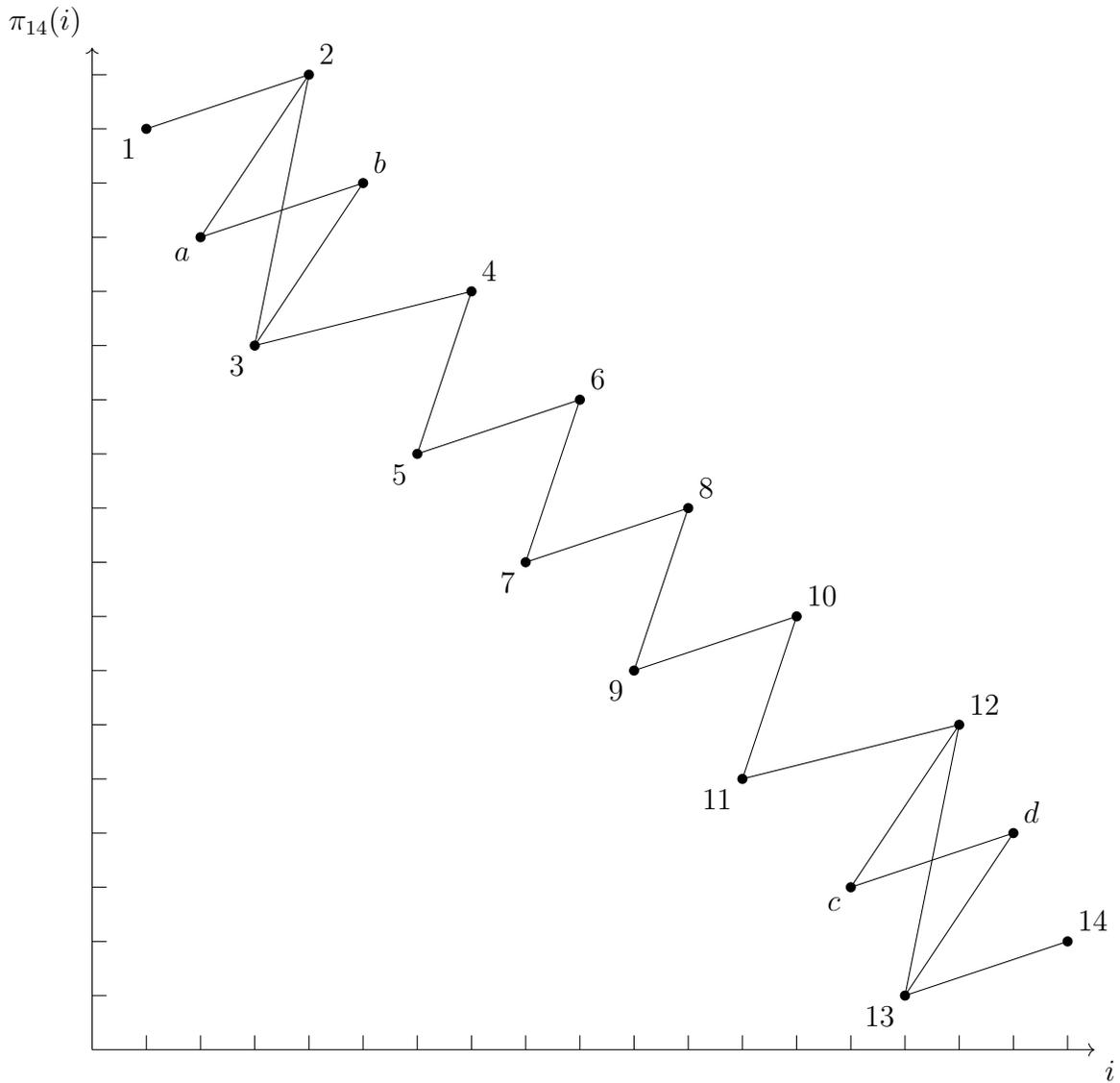

We conclude this section with an example of a universal theory $T$ of graphs that is in $\WR$
essentially because of failure of the product condition of Proposition~\ref{prop:noweaklyrandomsubgraphon}.

\begin{proposition}\label{prop:WRproductcond}
  Consider the sequence of graphs $G = (C_{n^2+5})_{n\in\NN}$ and let $T\df\Th(\phi_G)$ be the
  theory of positive models of the recursive blow-up $\phi_G$ relative to $G$ (see
  Definition~\ref{def:graphrecursiveblowup}). Then $T\in\WR$.
\end{proposition}

\begin{proof}
  Let $V=(V_n)_{n\in\NN}$ be given by $V_n\df V(C_{n^2+5})$ and let $W^G$ be the graphon
  representation of $\phi_G$ given by~\eqref{eq:WG}. Since $W^G$ is $\{0,1\}$-valued, we can view it
  as a continuum-sized graph $H$ with vertex set $\Omega^V$, consider the family $\cF$ of all finite
  graphs that are induced subgraphs of $H$.

  We claim that $\cF = Q(\phi_G)$. Indeed, we obviously have $Q(\phi_G)\subseteq\cF$ and any induced
  subgraph $M\in\cF$ of $H$ must be an induced subgraph of the (finite) recursive blow-up
  $H'_{m_0}\df R^{(C_5,C_6,C_9,\ldots,C_{m_0^2+5})}$ for some $m_0\in\NN$ (see
  Definition~\ref{def:compatibleblowup}) hence
  \begin{align*}
    \phi_G(M)
    & \geq
    p(M,H'_{m_0})\cdot\phi_G(H'_{m_0})
    \geq
    p(M,H'_{m_0})\cdot\frac{\lvert H'_{m_0}\rvert!}{\lvert\Aut(H'_{m_0})\rvert}\cdot
    \left(\prod_{n=0}^{m_0}\frac{1}{n^2+5}\right)^{\lvert H'_{m_0}\rvert}
    >
    0,
  \end{align*}
  so $M\in Q(\phi_G)$.

  Let us now show that $T\in\WR$. Let $\phi\in\HomT{T}$ be an arbitrary limit of $T$. Since
  $\cM[T]=\cF$, there exists a sequence $U = (U_n)_{n\in\NN}$ of finite subsets of $\Omega^V$ such
  that the sequence of finite graphs $(H\rest_{U_n})_{n\in\NN}$ converges to $\phi$.

  For each $k\in\NN$ and each $v\in V(C_{k^2+5})$, let
  \begin{align*}
    K_{k,v} & \df \{\sigma\in\Omega^V \mid \sigma_k = v\}.
  \end{align*}

  For each $n\in\NN$, let us construct a sequence $(U'_{n,k})_{k\in\NN}$ of subsets of $U_n$
  inductively as follows. We set $U'_{n,0}\df U_n$ and given $U'_{n,k}$, let $v_{n,k}$ be a vertex
  $v\in V(C_{k^2+5})$ that minimizes $\lvert U'_{n,k}\cap K_{k,v}\rvert$ (which can be zero) and let
  $U'_{n,k+1}\df U'_{n,k}\setminus K_{n,v_{n,k}}$; note that
  \begin{align*}
    \lvert U'_{n,k+1}\rvert & \geq \left(1 - \frac{1}{k^2+5}\right)\cdot\lvert U'_{n,k}\rvert.
  \end{align*}

  We also let $U'_n\df\bigcap_{k\in\NN} U'_{n,k}$ and note that a simple induction gives
  \begin{align}\label{eq:Uratio}
    \frac{\lvert U'_n\rvert}{\lvert U_n\rvert}
    & \geq
    \prod_{k\in\NN}\left(1 - \frac{1}{k^2+5}\right)
    >
    0.
  \end{align}
  Note also that the definition of $U'_n$ implies that for every $k\in\NN$, there exists $v\in
  V(C_{k^2+5})$ such that $U'_n\cap K_{k,v} = \varnothing$, which along with the definition of $H$
  implies that $H\rest_{U'_n}\in S(\{P_n\mid n\in\NN\})$, where $P_n$ is the path on $n$ vertices.

  Let then $(H\rest_{U'_{n_\ell}})_{\ell\in\NN}$ be a convergent subsequence of
  $(H\rest_{U'_n})_{n\in\NN}$ such that $(\lvert U'_{n_\ell}\rvert/\lvert
  U_{n_\ell}\rvert)_{\ell\in\NN}$ is also convergent and let $\psi\in\HomT{T}$ be the limit of
  $(H\rest_{U'_{n_\ell}})_{\ell\in\NN}$. Then~\eqref{eq:Uratio} implies that $\psi$ is a sub-object
  of $\phi$ of measure at least $\prod_{k\in\NN}(1-1/(k^2+5)) > 0$. But since
  $H\rest_{U'_{n_\ell}}\in S(\{P_n\mid n\in\NN\})$, it follows that $\Th(\psi)$ is primally almost
  finite, which by Lemma~\ref{lem:graphs:primallyalmostfinite->WR} gives $\Th(\psi)\in\WR$, so
  $\psi$ has a weakly random sub-object, hence so does $\phi$.
\end{proof}

\section{VC~dimension and weak randomness}
\label{sec:VC:graph}

In this section we study how weak randomness and the class $\WR$ interact with the notion of
VC~dimension. We remind the reader that in this section we drop the qualifiers ``weakly'' and
``strongly'' from ``closed under substitutions'' as they are superfluous for graphs (see
Remark~\ref{rmk:substarity2}).

Recall that for a non-trivial graph $G$, the \emph{Vapnik--Chervonenkis dimension}~\cite{VC71}
(VC~dimension) of (neighborhoods of) $G$ is the largest size $\VC(G)$ of a set $U\subseteq V(G)$
that is \emph{shattered} by neighborhoods of vertices of $G$ in the sense that for every $A\subseteq
U$, there exists $v\in V(G)$ such that $N_G(v)\cap U = A$, where $N_G(v)\df\{w\in V(G)\mid G\vDash
E(v,w)\}$ is the neighborhood of $v$ in $G$. By convention, we also let $\VC(K_0)\df 0$.

Recall also that (the edge relation in) a class of graphs $\cF$ (or the corresponding universal
theory $\Th(\cF)$) is said to have \emph{bounded VC~dimension} if $\sup\{\VC(G) \mid G\in\cF\} <
\infty$. In model theoretic language, the class has NIP (standing for \emph{not the independence
  property}).

Finally, recall that by~\cite{LS10}, a universal theory $T$ of graphs has bounded VC~dimension if
and only if all graphons of $T$ are a.e.\ $\{0,1\}$-valued\footnote{Let us warn the unfamiliarized
  reader that even if $W$ is a.e.\ $\{0,1\}$-valued, its theory of positive graphs $\Th(W)$ is not
  necessarily of bounded VC~dimension. For example, the construction in the proof of
  Theorem~\ref{thm:graphpersistence} always yields a $\{0,1\}$-valued graphon with $\Th(W)=\cF$,
  even if $\cF=\cM[\TGraph]$, which clearly has unbounded VC~dimension.}. Thus, studying
VC~dimension is directly related to studying whether the theory has fractional-valued graphons.

We start with a simple application of the theory of graph persistence developed so far.

\begin{proposition}\label{prop:01subgraphon}
  If $W$ is a graphon such that there exists a finite graph $G$ with $\phi_W(G)=0$, then $W$ has an
  a.e.\ $\{0,1\}$-valued subgraphon $W'$. Furthermore, $W'$ can be taken of the form $W' = W\rest_A$
  for some positive measure set $A$.
\end{proposition}

\begin{proof}
  Define a $\{0,1\}$-valued graphon $\widetilde{W}$ on the same space as $W$ by
  \begin{align*}
    \widetilde{W}(x,y)
    & \df
    \begin{dcases*}
      1, & if $0 < W(x,y) < 1$,\\
      0, & if $W(x,y)\in\{0,1\}$.
    \end{dcases*}
  \end{align*}

  Note that a subgraphon $W'$ of $W$ represented as $W\rest_f$ is a.e.\ $\{0,1\}$-valued if and only
  if $K_2\notin Q(\widetilde{W}\rest_f)$. Hence, $W$ has an a.e.\ $\{0,1\}$-valued subgraphon if and
  only if $K_2\notin P(\widetilde{W})$. Thus, to prove the proposition, it is sufficient to show
  that $K_2\in P(\widetilde{W})$ implies that every finite graph $G$ has positive density in $W$.

  Since $P(\widetilde{W})$ is closed under substitutions and induced subgraphs (Lemma~\ref{lem:PW}),
  it follows that $K_n\in P(\widetilde{W})$ for every $n\in\NN$ (as $K_n\cong K_2^{v\to
    K_{n-1}}$). But note that each copy of $K_n$ in $\widetilde{W}$ corresponds to points of $W$
  whose pairs all have values in $(0,1)$, hence have strictly fractional (conditional) probability
  of yielding edges, thus the fact that $K_{\lvert G\rvert}$ has positive density in $\widetilde{W}$
  implies that $G$ has positive density in $W$.

  \medskip

  For the final part, if $W'=W\rest_f$ for some function $f\colon X\to [0,1]$, then letting
  $A\df\{x\in X \mid f(x) > 0\}$ yields that $W\rest_A$ is an a.e.\ $\{0,1\}$-valued subgraphon of
  $W$.
\end{proof}

The reader might have noticed that with the notable exception of quasirandom graphons, all examples
of weakly random graphons of Section~\ref{sec:WR:graph} are $\{0,1\}$-valued. The next proposition
says this is not a coincidence: every universal weakly random limit of a \emph{proper} (strongly)
persistent class of graphs $\cF$ (i.e., $\cF\subsetneq\cM[\TGraph]$) must necessarily be
$\{0,1\}$-valued.

\begin{theorem}\label{thm:weaklyrandom01}
  If $W$ is a weakly random graphon such that there exists a finite graph $G$ with $\phi_W(G) = 0$,
  then $W$ is a.e.\ $\{0,1\}$-valued.
\end{theorem}

\begin{proof}
  We prove this by the contra-positive. Suppose $W$ is a weakly random graphon over some space
  $\Omega=(X,\cA,\mu)$ that is not a.e.\ $\{0,1\}$-valued. Let us show by induction in $\lvert
  G\rvert$ that every finite graph $G$ has $\phi_W(G) > 0$. By possibly applying the Graphon Removal
  Lemma~\cite[Theorem~1]{Pet13}, it is enough to show that $\Tind(G,W)\not\subseteq\cD_{V(G)}$.

  Obviously $\phi_W(K_0) = \phi_W(K_1) = 1$. So suppose $\lvert G\rvert\geq 2$, let $v_0\in V(G)$
  and $H\df G-v_0$. Since $W$ is not a.e.\ $\{0,1\}$-valued, there exists $x_{v_0}\in[0,1]$ such
  that the set
  \begin{align*}
    A & \df \{y\in [0,1]\setminus\{x_{v_0}\} \mid 0 < W(x_{v_0},y) < 1\}
  \end{align*}
  has positive measure. Since $W\rest_A$ is a subgraphon of $W$, $W\rest_A$ is also weakly random
  and since by induction hypothesis, $\phi_W(H) > 0$ we get that there exists a point $(z,w)\in
  A^{V(H)}\times [0,1)^{\binom{V(H)}{2}}$ that induces an off-diagonal copy of $H$ in $W\rest_A$
  (i.e., we have $(z,w)\in\Tind(H,W\rest_A)$).

  Let us extend $(z,w)$ to a point in $X^{V(G)}\times [0,1)^{\binom{V(G)}{2}}$ by defining
  $z_{v_0}\df x_{v_0}$ and
  \begin{align*}
    w_{\{v_0,w\}} & \df
    \begin{dcases*}
      0, & if $\{v_0,u\}\in E(G)$,\\
      \frac{1 + W(x_{v_0},z_u)}{2}, & if $\{v_0,u\}\notin E(G)$,
    \end{dcases*}
  \end{align*}
  for every $u\in V(H)$. It is straightforward to check that $(z,w)$ yields an off-diagonal copy of
  $G$ in $W$, concluding the proof.
\end{proof}

Our next objective is to show that for families of graphs $\cF$ that are closed under substitutions
and induced subgraphs, determining whether $\cF$ has bounded VC~dimension is reduced to determining
whether the family of prime graphs of $\cF$ has bounded VC~dimension. To do so, we need a variation
of the definition of VC~dimension.

\begin{definition}
  Given a non-trivial graph $G$, the \emph{$\VC'$~dimension} of $G$ (denoted $\VC'(G)$) is the
  largest size of a set $U\subseteq V(G)$ that is \emph{almost shattered} by the edge relation of
  $G$ in the sense that for every non-empty $A\subsetneq U$, there exists $v\in V(G)$ such that
  $N_G(v)\cap U = A\setminus\{v\}$. By convention, we also let $\VC'(K_0)\df 0$.
\end{definition}

Note that the notion of almost shattering is weaker than the notion of shattering in \emph{two}
points: we only care about sets $A$ that are \emph{non-empty proper subsets} of $U$ and $N_G(v)\cap
U$ only needs to match $A$ up to possibly removing $v$ from $A$. Note that for a non-trivial graph
$G$, we always have $\VC'(G)\geq 1$ as any singleton set is almost shattered by the edge relation of
$G$.

\begin{lemma}\label{lem:VCVC'}
  For a non-trivial graph $G$, we have
  \begin{align*}
    \max\{F(n) \mid n\leq \VC'(G)\} & \leq \VC(G) \leq \VC'(G),
  \end{align*}
  where
  \begin{align*}
    F(n)
    & \df
    \begin{dcases*}
      \max\left\{t\in\NN \;\middle\vert\; \sum_{i=0}^{t-1}\binom{n}{i} < \frac{2^n-2}{n}\right\},
      & if $n\geq 3$,
      \\
      0, & if $n\leq 2$.
    \end{dcases*}
  \end{align*}
\end{lemma}

\begin{proof}
  Since the notion of shattering implies the notion of almost shattering, it follows trivially that
  $\VC(G)\leq\VC'(G)$.

  \medskip

  Let now $n\leq\VC'(G)$ and let us show that $F(n)\leq\VC(G)$. Note that the result is trivial if
  $n\leq 2$ as $F(0)=F(1)=F(2)=0$, so let us suppose $n\geq 3$. Since $n\leq\VC'(G)$, we know that there
  exists a set $U\subseteq V(G)$ of size $n$ that is almost shattered by the edge relation of
  $G$. For each non-empty $A\subsetneq U$, let $v_A\in V(G)$ be such that $N_G(v_A)\cap U =
  A\setminus\{v_A\}$ and let $\cF\df\{N_G(v_A)\cap U \mid \varnothing\neq A\subsetneq U\}$.

  We claim that for every $B\subseteq U$, there are at most $n$ non-empty sets $A\subsetneq U$ such
  that $N_G(v_A)\cap U = B$. Indeed, since $N_G(v_A)\cap U = A\setminus\{v_A\}$, the set of non-empty
  $A\subsetneq U$ with $N_G(v_A)\cap U = B$ must be contained in
  \begin{align*}
    \{B\}\cup\{B\cup\{u\} \mid u\in U\}.
  \end{align*}
  When $B$ is non-empty, the set above has size at most $\lvert U\rvert = n$ and when $B$ is empty,
  the set above has size $n+1$ but $A$ cannot be equal to $B$.

  Since there are $2^n-2$ non-empty sets $A\subsetneq U$, we get $\lvert\cF\rvert\geq (2^n-2)/n$. On
  the other hand, by the definition of $F(n)$, we have
  \begin{align*}
    \lvert\cF\rvert & \geq \frac{2^n-2}{n} > \sum_{i=0}^{F(n)-1}\binom{n}{i},
  \end{align*}
  so by the Sauer--Shelah Lemma~\cite{Sau72,She72}, the family $\cF$ shatters some $U'\subseteq U$
  with $\lvert U'\rvert\geq F(n)$, thus $\VC(G)\geq F(n)$.
\end{proof}

\begin{remark}\label{rmk:NIPVC'}
  It is easy to see that the function $F$ of Lemma~\ref{lem:VCVC'} is unbounded. Indeed, if there
  was a bound $t_0\in\NN$ such that $F(n)\leq t_0$ for every $n\in\NN$, then we would have
  \begin{align*}
    \frac{2^n-2}{n} \leq \sum_{i=0}^{t_0}\binom{n}{i} \leq (t_0+1)\cdot n^{t_0}
  \end{align*}
  for every $n\in\NN$, which is yields a contradiction when $n$ is sufficiently large.

  As a corollary of Lemma~\ref{lem:VCVC'}, it then follows that a universal theory graphs $T$ has
  bounded VC~dimension if and only if there exists $k\in\NN$ such that $\VC'(G)\leq k$ for every
  graph $G$ of $T$.
\end{remark}

The next lemma shows that $\VC'$~dimension behaves very well with respect to the substitution
operation.

\begin{lemma}\label{lem:VC'subst}
  Let $F_1$ and $F_2$ be non-trivial finite graphs and $v_0\in V(F_1)$. Then we have
  $\VC'(F_1^{v_0\to F_2}) = \max\{\VC'(F_1),\VC'(F_2)\}$.
\end{lemma}

\begin{proof}
  Without loss of generality, suppose $V(F_1)\cap V(F_2)=\varnothing$ and let $G\df F_1^{v_0\to
    F_2}$. Since both $F_1$ and $F_2$ are induced subgraphs of $G$ (as both $F_1$ and $F_2$ are
  non-trivial), it follows that $\VC'(G)\geq\max\{\VC'(F_1),\VC'(F_2)\}$.

  To prove the other inequality, let $U\subseteq V(G)$ be a set that is almost shattered by the edge
  relation of $G$ with $\lvert U\rvert = \VC'(G)$.

  Suppose first that $U\subseteq V(F_i)$ for some $i\in[2]$. Then we claim that the edge relation of
  $F_i$ also almost shatters $U$. Indeed, if $A\subsetneq U$ is a non-empty set, then we know that
  there exists $u\in V(G)$ such that $N_G(u)\cap U = A\setminus\{u\}$. Since $A\neq\varnothing$ and
  $A\neq V(F_i)$ (as $A\subsetneq U\subseteq V(F_i)$), we must have $u\in V(F_i)$ (as every $u\in
  V(F_{3-i})$ is either adjacent to all of $V(F_i)$ or not adjacent to all of $V(F_i)$) and thus
  $N_{F_i}(u)\cap U = A\setminus\{u\}$. Therefore, in this case, we get
  $\VC'(G)\leq\VC'(F_i)\leq\max\{\VC'(F_1),\VC'(F_2)\}$.

  Suppose then that $U\not\subseteq V(F_1)$ and $U\not\subseteq V(F_2)$. Then we claim that $\lvert
  U\cap V(F_2)\rvert = 1$. Suppose not and let $v_1\in U\cap V(F_1)$ and $v_2,w_2\in U\cap V(F_2)$
  with $v_2\neq w_2$. We consider first the case when $\{v_0,v_1\}\in E(F_1)$ (recall that $v_0$ is
  the vertex of $F_1$ that is being substituted: $G=F_1^{v_0\to F_2}$), let $A\df\{v_2\}\subsetneq
  U$ and let $u\in V(G)$ be such that $N_G(u)\cap U = A\setminus\{u\}$. Since $\{v_0,v_1\}\in
  E(F_1)$ and $v_1\notin A$, we must have $u\notin V(F_2)$ (as every vertex of $V(F_2)$ is adjacent
  to $v_1$ in $G$) and since $w_2\notin A$, we must have $u\notin V(F_1)$ (as every vertex of
  $V(F_1)$ is adjacent to $w_2$ in $G$), a contradiction. Consider then the case when
  $\{v_0,v_1\}\notin E(F_1)$, let $A\df\{v_1,v_2\}\subsetneq U$ and let $u\in V(G)$ be such that
  $N_G(u)\cap U = A\setminus\{u\}$. Since $\{v_0,v_1\}\notin E(F_1)$ and $v_1\in A$, we must have
  $u\notin V(F_2)$ and since $v_2\in A$, we must have $u\notin V(F_1)$, a contradiction. This
  concludes the proof of the claim, that is, we have $\lvert U\cap V(F_2)\rvert = 1$.

  Let then $w_0$ be the unique element of $U\cap V(F_2)$. We now claim that the set $U'\df
  (U\setminus\{w_0\})\cup\{v_0\}$ is almost shattered by the edge relation of $F_1$. Let
  $A'\subsetneq U'$ be a non-empty set and let
  \begin{align*}
    A
    & \df
    \begin{dcases*}
      A', & if $v_0\notin A'$,\\
      (A'\setminus\{v_0\})\cup\{w_0\}, & if $v_0\in A'$.
    \end{dcases*}
  \end{align*}
  Then there exists $u\in V(G)$ such that $N_G(u)\cap U = A\setminus\{u\}$. Consider first the case
  when $u\in V(F_1)$. Then we have
  \begin{align*}
    N_{F_1}(u)\cap U' & =
    \begin{dcases*}
      N_G(u)\cap U, & if $w_0\notin N_G(u)\cap U$,\\
      ((N_G(u)\cap U)\setminus\{w_0\})\cup\{v_0\}, & if $w_0\in N_G(u)\cap U$,
    \end{dcases*}
  \end{align*}
  hence $N_{F_1}(u)\cap U' = A'\setminus\{u\}$ (as $N_G(u)\cap U = A\setminus\{u\}$). Consider now the
  case when $u\notin V(F_1)$ and note that
  \begin{align*}
    N_{F_1}(v_0)\cap U' = (N_G(u)\cap U)\setminus\{w_0\} = A\setminus\{u,w_0\} = A'\setminus\{v_0\}
  \end{align*}
  since $u\notin A'$ (as $A'\subseteq V(F_1)$). Thus $U'$ is almost shattered by the edge relation of
  $F_1$, hence $\VC'(G)\leq\VC'(F_1)\leq\max\{\VC'(F_1),\VC'(F_2)\}$.
\end{proof}

The following simple consequence of Lemmas~\ref{lem:VCVC'} and~\ref{lem:VC'subst} (and
Remark~\ref{rmk:NIPVC'}) says that when $\cF = S(\cF')$, then $\cF$ has bounded VC~dimension if and
only if $\cF'$ has bounded VC~dimension.

\begin{theorem}\label{thm:VC}
  Let $\cF$ and $\cF'$ be families of finite graphs up to isomorphism and suppose
  $\cF=S(\cF')$. Then $\cF$ has bounded VC~dimension if and only if $\cF'$ has bounded VC~dimension.

  In particular, if $\cF$ is a family of finite graphs that is closed under substitutions and under
  induced subgraphs and $\cP$ is the family of all prime graphs of $\cF$, then $\cF$ has bounded
  VC~dimension if and only if $\cP$ has bounded VC~dimension.
\end{theorem}

\begin{proof}
  By Lemmas~\ref{lem:VCVC'} and~\ref{lem:VC'subst} and Remark~\ref{rmk:NIPVC'} with a simple
  induction, we have
  \begin{align*}
    \sup\{\VC(F) \mid F\in\cF\} < \infty
    & \iff
    \sup\{\VC'(F) \mid F\in\cF\} < \infty
    \\
    & \iff
    \sup\{\VC'(F') \mid F'\in\cF'\} < \infty
    \\
    & \iff
    \sup\{\VC(F') \mid F'\in\cF'\} < \infty,
  \end{align*}
  so the first statement follows.

  The second statement follows from the first one along with
  Lemma~\ref{lem:graph:strongclosuresubst}.
\end{proof}

As a direct corollary of Theorem~\ref{thm:VC}, it follows that any primally finite family $\cF$ that
is closed under substitutions and under induced subgraphs has bounded VC~dimension. Our next
objective is to show that the same is true in the primally almost finite case. Before we do so, we
need yet another example of a family of prime graphs that is not almost finite.

\begin{example}\label{ex:cliquedpath}
  For each $n\geq 9$ odd, let $G_n'$ be the graph obtained from the path on $n$ vertices $P_n$ by
  adding two vertices $a$ and $b$ adjacent precisely to the fourth and fourth from last vertices of
  $P_n$, respectively and connecting all even vertices into a clique (see
  Figure~\ref{fig:cliquedpath}). Formally, we have
  \begin{align*}
    V(G_n') & \df [n]\cup\{a,b\},
    \\
    E(G_n')
    & \df
    \begin{multlined}[t]
      \{\{i,i+1\} \mid i\in[n-1]\}\cup\{a,4\}\cup\{b,n-3\}\\
      \cup\{\{2i,2j\}\mid i,j\in[(n-1)/2]\land i\neq j\}.
    \end{multlined}
  \end{align*}

  It is straightforward to check that $\{G_n' \mid n\geq 9\text{ odd}\}$ is a family of prime
  graphs that is not almost finite.
\end{example}

\begin{figure}[htbp]
  \begingroup
\def\basehorzdist{1}
\def\vertdist{1.2}
\def\controlbase{0.5}
\def\ptsize{2pt}
\def\subfigwidth{\linewidth}
\newcommand{\draw}[1]{%
  \begin{subfigure}{\subfigwidth}
    \begin{center}
      \begin{tikzpicture}
        \foreach \i in {1,...,#1}{
          \pgfmathsetmacro{\x}{\i * \basehorzdist}
          \coordinate (P\i) at (\x,0);
          \coordinate (Q\i) at (\x,-\vertdist);
        }

        \foreach \i in {2,4,...,#1}{
          \pgfmathtruncatemacro{\pi}{\i-1}
          \fill (Q\i) circle (\ptsize);
          \fill (P\pi) circle (\ptsize);

          \node[left] at (Q\i) {$\i$};
          \node[above] at (P\pi) {$\pi$};

          \pgfmathtruncatemacro{\ipone}{\i+1}
          \draw (P\pi) -- (Q\i) -- (P\ipone);
        }

        \fill (P#1) circle (\ptsize);
        \fill (P4) circle (\ptsize);
        \pgfmathtruncatemacro{\nmthree}{#1-3}
        \fill (P\nmthree) circle (\ptsize);

        \node[above] at (P#1) {$#1$};
        \node[above] at (P4) {$a$};
        \node[above] at (P\nmthree) {$b$};

        \pgfmathtruncatemacro{\nmone}{#1-1}
        \draw (P#1) -- (Q\nmone);
        \draw (P4) -- (Q4);
        \draw (P\nmthree) -- (Q\nmthree);

        \pgfmathtruncatemacro{\nmoneovertwo}{(#1-1)/2}
        \foreach \i in {2,...,\nmoneovertwo}{
          \pgfmathtruncatemacro{\twoi}{2*\i}
          \foreach \j in {1,...,\i}{
            \pgfmathtruncatemacro{\twoj}{2*\j}
            \pgfmathsetmacro{\c}{\controlbase * (\i - \j)}
            \draw (Q\twoi) .. controls ($(Q\twoi) + (0,-\c)$) and ($(Q\twoj) + (0,-\c)$) .. (Q\twoj);
          }
        }
      \end{tikzpicture}
      \caption*{$G_{#1}'$.}
    \end{center}
  \end{subfigure}
}
\begin{center}
  \draw{9}
  \vskip 0.5cm
  \draw{11}
  \vskip 0.5cm
  \draw{13}
  \caption{Prime graphs $G_n'$ of Example~\ref{ex:cliquedpath} that form a family that is
    not almost finite.}
  \label{fig:cliquedpath}
\end{center}

\endgroup

\end{figure}

Before we prove the theorem, we need a small consequence of Ramsey's Theorem.

\begin{lemma}\label{lem:Ramseydensity}
  For every $n\in\NN$ and every graphon $W$, we have
  \begin{align*}
    \phi_W(K_n) + \phi_W(\overline{K}_n) & \geq \binom{R(n,n)}{n}^{-1},
  \end{align*}
  where $R(n,n)$ is the $(n,n)$-Ramsey number.
\end{lemma}

\begin{proof}
  We have
  \begin{align*}
    \phi_W(K_n) + \phi_W(\overline{K}_n)
    & =
    \sum_{M\in\cM_{R(n,n)}[\TGraph]} (p(K_n,M) + p(\overline{K}_n,M))\cdot\phi_W(M)
    \\
    & \geq
    \binom{R(n,n)}{n}^{-1}\sum_{M\in\cM_{R(n,n)}[\TGraph]}\phi_W(M)
    =
    \binom{R(n,n)}{n}^{-1},
  \end{align*}
  where the inequality follows since at least one $n$-sized subset of each
  $M\in\cM_{R(n,n)}[\TGraph]$ must induce either $K_n$ or $\overline{K}_n$ in $M$.
\end{proof}

\begin{theorem}\label{thm:primallyalmostfiniteNIP}
  Let $T$ be a universal theory of graphs. If $\cM[T]$ is primally almost finite, then the edge
  relation in $T$ has NIP (i.e., bounded VC~dimension).

  In particular, the edge relation in every universal theory of graphs $T'\in\WR$ such that
  $\cM[T']$ is closed under substitutions has NIP.
\end{theorem}

\begin{proof}
  The second assertion follows from the first along with Theorem~\ref{thm:graphs:WR}.

  \medskip

  We prove the first assertion by the contra-positive. Since the edge relation in $T$ has unbounded
  VC~dimension, by~\cite{LS10}, there exists a graphon $W$ that is a limit of $T$ and is \emph{not}
  a.e.\ $\{0,1\}$-valued. By possibly applying the Graphon Removal Lemma~\cite[Theorem~1]{Pet13}, we
  may suppose that every graph $G$ that has an off-diagonal copy in $W$ has positive density in $W$.

  Our objective is to present a family of prime graphs that is not almost finite and such that all
  graphs in this family have an off-diagonal copy in $W$ (thus $\cM[T]$ is not primally almost
  finite).

  For this purpose, we will show that for
  each $n\in\NN$ with $n\geq 5$, one of the following graphs appears as an off-diagonal copy in $W$:
  \begin{enumerate}
  \item The graph $G_{2n-4}$ of Example~\ref{ex:substitutiongeneration}.
  \item The complement $\overline{G}_{2n-4}$ of the graph of Example~\ref{ex:substitutiongeneration}.
  \item The graph $G_{2n-1}'$ of Example~\ref{ex:cliquedpath}.
  \end{enumerate}
  Since each of these families is a family of prime graphs that is not almost finite (note that
  primality is preserved under complementation) and one of them must occur for infinitely many $n$,
  it will follow that $\cM[T]$ is not primally almost finite as desired.

  Without loss of generality, let us suppose that the underlying space of $W$ is $[0,1]$ and let
  $(x_0,y_0)\in(0,1)^2$ be a Lebesgue density point with respect to $\ell^\infty$-balls of the
  positive measure set $A\df W^{-1}((0,1))$ with $x_0\neq y_0$.

  Let $\epsilon > 0$ be such that $\epsilon < (n\cdot\binom{R(n,n)}{n})^{-2}$ and let $\delta > 0$
  be small enough so that
  \begin{align*}
    \frac{
      \lambda(A\cap B_\delta(x_0,y_0))
    }{
      \lambda(B_\delta(x_0,y_0))
    }
    & \geq
    1 - \epsilon,
  \end{align*}
  where $B_\delta(x_0,y_0)$ is the $\ell^\infty$-ball of radius $\delta$ centered at $(x_0,y_0)$. We
  may also suppose that $\delta > 0$ is small enough so that $(x_0-\delta,x_0+\delta)$ and
  $(y_0-\delta,y_0+\delta)$ are disjoint subsets of $[0,1]$.

  Consider the set
  \begin{align*}
    C & \df \{(x,y)\in B_\delta(x_0)^n\times B_\delta(y_0)^n \mid \forall i,j\in[n], (x_i,y_j)\in A\}.
  \end{align*}
  With a simple union bound, we have
  \begin{align}\label{eq:lambdaC}
    \lambda(C)
    & \geq
    (1 - n^2\epsilon)\cdot
    \lambda(B_\delta(x_0,y_0))^n
    >
    \left(1 - \binom{R(n,n)}{n}^{-2}\right)
    \lambda(B_\delta(x_0,y_0))^n.
  \end{align}

  Define also
  \begin{multline*}
    C'
    \df
    \{(x,y)\in B_\delta(x_0)^n\times B_\delta(y_0)^n \mid
    \exists z\in[0,1)^{\binom{[n]}{2}}, (x,z)\in\Tind(K_n,W)\cup\Tind(\overline{K}_n,W)
    \\
    \land \exists w\in[0,1)^{\binom{[n]}{2}}, (y,w)\in\Tind(K_n,W)\cup\Tind(\overline{K}_n,W)\}.
  \end{multline*}
  By Lemma~\ref{lem:Ramseydensity}, we have
  \begin{align*}
    \lambda(C') & \geq \binom{R(n,n)}{n}^{-2}\cdot\lambda(B_\delta(x_0,y_0))^n.
  \end{align*}
  Putting this together with~\eqref{eq:lambdaC}, we conclude that $\lambda(C\cap C') > 0$.

  Let then $(x,y)\in C\cap C'$ be a point with all coordinates distinct. We now consider four cases.

  \begin{enumerate}[label={Case~\arabic*.}, ref={\arabic*}, wide]
  \item\label{case:emptyempty} There exist points $z,w\in[0,1)^{\binom{[n]}{2}}$ such that
    $(x,z),(y,w)\in\Tind(\overline{K}_n,W)$. In this case, we construct an off-diagonal copy
    $(\widehat{x},\widehat{y})\in [0,1]^{V(G_{2n-4})}\times [0,1)^{\binom{V(G_{2n-4})}{2}}$ of the
    graph $G_{2n-4}$ of Example~\ref{ex:substitutiongeneration} as follows. Recall that
    $V(G_{2n-4}) = [2n-4]\cup\{a,b,c,d\}$ and for convenience of notation, let us make the
    identifications $a\df 2n-3$, $b\df 2n-2$, $c\df 2n-1$ and $d\df 2n$. Then for each $i,j\in
    V(G_{2n-4})$ with $i\neq j$, let
    \begin{align*}
      \widehat{x}_i & \df
      \begin{dcases*}
        x_{i/2}, & if $i\in[2n]$ is even,\\
        y_{(i+1)/2}, & if $i\in[2n]$ is odd,
      \end{dcases*}
      \\
      \widehat{y}_{\{i,j\}}
      & \df
      \begin{dcases*}
        z_{\{i/2,j/2\}}, & if $i,j\in[2n]$ are both even,\\
        w_{\{(i+1)/2,(j+1)/2\}}, & if $i,j\in[2n]$ are both odd,\\
        0,
        & if $i\in[2n]$ is even, $j\in[2n]$ is odd and
        $\{i,j\}\in E(G_{2n-4})$,
        \\
        \frac{1 + W(x_{i/2},y_{(j+1)/2})}{2},
        & if $i\in[2n]$ is even, $j\in[2n]$ is odd
        and $\{i,j\}\notin E(G_{2n-4})$.
      \end{dcases*}
    \end{align*}
    The fact that $(x,y)\in C\cap C'$ and all coordinates of $(x,y)$ are distinct guarantees that
    $(\widehat{x},\widehat{y})$ is an off-diagonal copy of $G_{2n-4}$.
  \item There exist points $z,w\in[0,1)^{\binom{[n]}{2}}$ such that $(x,z),(y,w)\in\Tind(K_n,W)$. In
    this case, a construction analogous to the one in case~\ref{case:emptyempty} yields an
    off-diagonal copy of the complement $\overline{G}_{2n-4}$ of the graph of
    Example~\ref{ex:substitutiongeneration}.
  \item\label{case:cliqueempty} There exist points $z,w\in[0,1)^{\binom{[n]}{2}}$ such that
    $(x,z)\in\Tind(K_n,W)$ and $(y,w)\in\Tind(\overline{K}_n,W)$. We construct an off-diagonal copy
    $(\widehat{x},\widehat{y})\in [0,1]^{V(G_{2n-1}')}\times [0,1)^{\binom{V(G_{2n-1}')}{2}}$ of the
    graph $G_{2n-1}'$ of Example~\ref{ex:cliquedpath} as follows. Recall that $V(G_{2n-1}') =
    [2n-1]\cup\{a,b\}$ and for convenience of notation, let us make the identifications $a\df
    2n+1$ and $b\df 2n+3$ (nothing gets identified with the points $2n$ and $2n+2$). Then for each
    $i,j\in V(G_{2n-1}')$ with $i\neq j$, let
    \begin{align*}
      \widehat{x}_i & \df
      \begin{dcases*}
        x_{i/2}, & if $i\in[2n-1]$ is even,\\
        y_{(i+1)/2}, & if $i\in[2n+3]$ is odd,
      \end{dcases*}
      \\
      \widehat{y}_{\{i,j\}}
      & \df
      \begin{dcases*}
        z_{\{i/2,j/2\}}, & if $i,j\in[2n-1]$ are both even,\\
        w_{\{(i+1)/2,(j+1)/2\}}, & if $i,j\in[2n+3]$ are both odd,\\
        0, &
        \parbox[t]{0.6\textwidth}{if $i\in[2n-1]$ is even, $j\in[2n+3]$ is odd and $\{i,j\}\in
          E(G_{2n-1}')$,}
        \\
        \frac{1 + W(x_{i/2},y_{(j+1)/2})}{2},
        & \parbox[t]{0.6\textwidth}{if $i\in[2n-1]$ is even, $j\in[2n+3]$ is odd and $\{i,j\}\notin
          E(G_{2n-1}')$.}
      \end{dcases*}
    \end{align*}
    The fact that $(x,y)\in C\cap C'$ and all coordinates of $(x,y)$ are distinct guarantees that
    $(\widehat{x},\widehat{y})$ is an off-diagonal copy of $G_{2n-1}'$.
  \item There exist points $z,w\in[0,1)^{\binom{[n]}{2}}$ such that $(x,z)\in\Tind(\overline{K}_n,W)$ and
    $(y,w)\in\Tind(K_n,W)$. This case follows from case~\ref{case:cliqueempty} by swapping the roles
    of $x$ and $y$.
  \end{enumerate}

  Therefore $\cM[T]$ is not primally almost finite.
\end{proof}

\begin{remark}
  The assumption that $\cM[T]$ is closed under substitution is crucial for the second part of
  Theorem~\ref{thm:primallyalmostfiniteNIP}, since, for example, the universal theory $\TBipartite$
  of bipartite graphs clearly is in $\AEHP\subseteq\WR$ (as every limit contains an empty subgraphon
  of measure at least $1/2$) but has unbounded VC~dimension.

  For an example of a theory with unbounded VC~dimension that is in $\WR\setminus\AEHP$, let
  $T_{C_4}$ be the universal theory of graphs that are induced subgraphs of some (finite) recursive
  blow-up of $C_4$ (which has bounded VC~dimension by Theorem~\ref{thm:primallyalmostfiniteNIP} as
  $T_{C_4}$ is primally finite), let $\cF$ be the family of graphs $G$ such that there exists a
  partition $V(G) = A\cup B$ such that both $G\rest_A$ and $G\rest_B$ are models of $T_{C_4}$ and
  let $T\df\Th(\cF)$ be the corresponding universal theory of graphs. Obviously, every bipartite
  graph is a model of $T$, so $T$ has unbounded VC~dimension.

  Since at least one of $A$ or $B$ must have at least half of the vertices, it follows that every
  limit of $T$ has a subgraphon that is a limit of $T_{C_4}$ and since $T_{C_4}\in\WR$ (by
  Proposition~\ref{prop:recC4}), we get $T\in\WR$. On the other hand, since every model of $T_{C_4}$
  is a model of $T$ and $T_{C_4}\notin\AEHP$ (by~\cite[Lemma~8.8]{CM22} or
  Proposition~\ref{prop:recC4} again), it follows that $T\notin\AEHP$.
\end{remark}

\section{Persistence for universal theories}
\label{sec:pers:univ}

In this section, we generalize the results of Section~\ref{sec:pers:graph} on (strongly) persistent
classes to arbitrary universal theories in finite relational languages. Table~\ref{tab:corresp} below
contains the correspondence between the theorems and lemmas of Sections~\ref{sec:pers:graph}
and~\ref{sec:WR:graph} and their generalizations in this and the next section.

\begin{table}[htbp]
  \begin{tabular}{*{2}{c}p{7.3cm}}
    Graph result & Universal theory result & Drawbacks
    \\
    \hline
    Lemma~\ref{lem:PWsubgraphon} & Lemma~\ref{lem:PWRset}\ref{lem:PWRset:P} & None.
    \\
    Theorem~\ref{thm:graphpersistence} & Theorem~\ref{thm:persistence} & Persistence
    (item~\ref{thm:persistence:persistent}) can only be added to the list of equivalences when all
    arities are at most $2$.
    \\
    Lemma~\ref{lem:PW} & Lemma~\ref{lem:Pphi} & Requires either that all arities are at most $2$
    (item~\ref{lem:Pphi:arity2}) or weak randomness (item~\ref{lem:Pphi:WR}).
    \\
    Lemma~\ref{lem:QphiG} &
    \begin{tabular}[t]{c}
      Lemma~\ref{lem:comprecblowup} and\\
      Proposition~\ref{prop:recblowup}\ref{prop:recblowup:conservative},\ref{prop:recblowup:nonconservative}
    \end{tabular}
    &
    When the recursive blow-ups are not conservative (see Definitions~\ref{def:compatibleblowup}
    and~\ref{def:recursiveblowup}), only partial information is known about the limit theon.
    \\
    Lemma~\ref{lem:PphiGQphiG} &
    Proposition~\ref{prop:recblowup}\ref{prop:recblowup:P} &
    None.
    \\
    Lemma~\ref{lem:weaklyrandomsubgraphon} &
    Lemma~\ref{lem:PWRset}\ref{lem:PWRset:WR} &
    None.
    \\
    Theorem~\ref{thm:graphs:WR} &
    Propositions~\ref{prop:monprimalmostfinite->WR} and~\ref{prop:WR->primalmostfinite} &
    Backward direction requires all arities to be at most $2$. Forward direction is trivial if all
    arities are at least $3$.
    \\
    Lemma~\ref{lem:graphs:primallyalmostfinite->WR} &
    Proposition~\ref{prop:monprimalmostfinite->WR} &
    Requires all arities to be at most $2$.
    \\
    Proposition~\ref{prop:noweaklyrandomsubgraphon} &
    Propositions~\ref{prop:recblowup}\ref{prop:recblowup:zeroproduct}
    and~\ref{prop:WR->primalmostfinite} &
    When all arities are at least $3$, there are only finitely many prime structures (see
    Remark~\ref{rmk:prime3}).
  \end{tabular}
  \caption{Correspondence between theorems and lemmas of Sections~\ref{sec:pers:graph}
    and~\ref{sec:WR:graph} and their generalizations in Sections~\ref{sec:pers:univ}
    and~\ref{sec:WR:univ}. Some generalizations have drawbacks (e.g., extra hypotheses or the
    result might be trivial) as pointed out in the third column.}
  \label{tab:corresp}
\end{table}

\begin{definition}
  Let $\phi\in\HomT{T_\cL}$. The set of \emph{positive $\cL$-structures in $\phi$} is the set
  $Q(\phi)\df\cM[\Th(\phi)]$ of all finite $\cL$-structures $M$ (up to isomorphism) such that
  $\phi(M) > 0$. The set of \emph{persistently positive $\cL$-structures in $\phi$} is the set
  $P(\phi)\df\bigcap_\psi Q(\psi)$, where the intersection is over all sub-objects of $\phi$. We
  extend these definitions naturally to Euclidean structures $\cN$ in $\cL$ by $Q(\cN)\df
  Q(\phi_\cN)$ and $P(\cN)\df P(\phi_\cN)$.

  We say that $\phi$ is \emph{weakly random} if $P(\phi) = Q(\phi)$.

  A family $\cF$ of finite $\cL$-structures (up to isomorphism) is called \emph{persistent} if there
  exists $\phi$ such that $P(\phi) = \cF$. The family $\cF$ is called \emph{strongly persistent} if
  there exists a weakly random $\phi$ such that $P(\phi) = \cF$ (which must also equal $Q(\phi)$);
  in this case, we also say that $\phi$ is a \emph{universal weakly random limit of $\cF$}.
\end{definition}

\begin{lemma}\label{lem:PWRset}
  Let $\cN$ be an Euclidean structure in $\cL$ over $\Omega=(X,\cA,\mu)$. Then the following hold.
  \begin{enumerate}
  \item $P(\cN) = \bigcap_A Q(\cN\rest_A^F)$, where the intersection is over all positive measure
    $A\in\cA$ and all measure-isomorphisms $F$ modulo $0$ from $\Omega_A$ to $\Omega$ (equivalently,
    we can also use a single measure-isomorphism $F_A$ modulo $0$ for each positive measure
    $A\in\cA$).%
    \label{lem:PWRset:P}
  \item $\phi_\cN$ has a weakly random sub-object if and only if there exists a positive measure
    $A\in\cA$ and a measure-isomorphism $F$ modulo $0$ from $\Omega_A$ to $\Omega$ such that
    $\phi_{\cN\rest_A^F}$ is weakly random.%
    \label{lem:PWRset:WR}
  \end{enumerate}
\end{lemma}

\begin{proof}
  Both items follow from the fact that if $f\colon X\to [0,1]$ is a measurable function with $\int_X
  f\ d\mu > 0$ and $F$ is a measure-isomorphism modulo $0$ from $\Omega_f$ to $\Omega$, then for $A
  = \{x\in X \mid f(x) > 0\}$ and any measure-isomorphism $\widetilde{F}$ modulo $0$ from $\Omega_A$
  to $\Omega$, we have $Q(\cN\rest_f^F) = Q(\cN\rest_A^{\widetilde{F}})$.
\end{proof}

\begin{lemma}\label{lem:Pphi}
  The following hold for $\phi\in\HomT{T_\cL}$.
  \begin{enumerate}
  \item If all predicate symbols of $\cL$ have arity at most $2$, then $P(\phi)$ is strongly closed
    under substitutions and closed under substructures.%
    \label{lem:Pphi:arity2}
  \item If $\phi$ is weakly random, then $P(\phi)$ is weakly closed
    under substitutions and closed under substructures.%
    \label{lem:Pphi:WR}
  \end{enumerate}
\end{lemma}

\begin{proof}
  Since obviously $K_0\in P(\phi)$, by Lemma~\ref{lem:substitutionK0}, it is sufficient to show the
  assertions of closed under substitutions in each item.

  Let $F_1,F_2\in P(\phi)$ and $v\in V(F_1)$ and let $\cF$ be the set of standard substitutions of
  $v$ in $F_1$ by $F_2$.

  Note that in item~\ref{lem:Pphi:arity2}, by Remark~\ref{rmk:substarity2}, $\cF$ has a unique
  element (and the notion of strongly and weakly closed under substitutions coincide). Thus, in both
  items, our objective is to show that $\cF\cap P(\phi)$ is non-empty.

  Let $\psi$ be a sub-object of $\phi$ and let $\cN$ be an Euclidean structure in $\cL$ over some
  space $\Omega=(X,\cA,\mu)$ with $\phi_\cN = \psi$. We claim that $\cF\cap Q(\psi)$ is
  non-empty. Suppose not, that is, suppose $\tind(F,\cN) = 0$ for every $F\in\cF$. By possibly
  applying the Induced Euclidean Removal Lemma~\cite[Theorem~3.3]{CR20a}, we may suppose that
  $\Tind(F,\cN)\subseteq\cD_V$ for every $F\in\cF$.

  Since $F_1\in P(\phi)$, we must have $F_1\in Q(\cN)$, that is, we have $\tind(F_1,\cN) > 0$. For
  every $x\in X^{r(V(F_1))\setminus\{\{v\}\}}$, let
  \begin{align*}
    U_x & \df \{y\in X \mid (x,y)\in\Tind(F_1,\cN)\}.
  \end{align*}
  By Fubini's Theorem, there exists $x\in X^{r(V(F_1))\setminus\{\{v\}\}}$ with all coordinates
  distinct such that $\mu(U_x) > 0$. Let $G$ be a measure-isomorphism modulo $0$ from $\Omega_{U_x}$
  to $\Omega$ and since $\phi_{\cN\rest_{U_x}^G}$ is a sub-object of $\psi$, hence also of $\phi$,
  we must have $F_2\in Q(\cN\rest_{U_x}^G)$, which implies that there exists $z\in\cE_{V(F_2)}$ such
  that
  \begin{enumerate}[label={\alph*.}]
  \item all coordinates of $z$ are distinct;
  \item all coordinates of $z$ are distinct from the coordinates of $x$;
  \item for every $v\in V(F_2)$, we have $z_{\{v\}}\in U_x$;
  \item we have $z\in\Tind(F_2,\cN)$.
  \end{enumerate}

  Define then the point $w\in\cE_V$ by the following procedure.
  \begin{enumerate}[label={\arabic*.}, ref={(\arabic*)}]
  \item For each $A\subseteq r(V(F_1-v))$, let $w_A\df x_A$.%
    \label{it:F1}
  \item For each $A\subseteq r(V(F_1-v))$ and each $u\in V(F_2)$, let $w_{A\cup\{u\}}\df
    x_{A\cup\{v\}}$.%
    \label{it:F1F2}
  \item For each $A\subseteq r(V(F_2))$, let $w_A\df z_A$.%
    \label{it:F2}
  \item Define all other coordinates of $w$ arbitrarily.
  \end{enumerate}

  Note that all coordinates of $w$ that are indexed by single vertices get defined in
  items~\ref{it:F1} and~\ref{it:F2} and their definitions guarantee that they are distinct from each
  other, that is, we have $w\notin\cD_V$. Let then $F$ be the unique $\cL$-structure with
  $w\in\Tind(F,\cN)$. Then items~\ref{it:F1} and~\ref{it:F1F2} ensure that all injections
  $V(F_1)\to V$ acting identically on $V(F_1-v)$ are embeddings of $F_1$ in $F$ and
  item~\ref{it:F2} ensures that the injection $V(F_2)\to V$ that acts identically on $V(F_2)$ is an
  embedding of $F_2$ in $F$. Thus, we must have $F\in\cF$.

  Therefore, we have showed that for every sub-object $\psi$ of $\phi$, we have $\cF\cap
  Q(\psi)\neq\varnothing$.

  In item~\ref{lem:Pphi:arity2}, since $\cF$ has a single element $F$, it follows that $F\in
  P(\phi)$, hence $P(\phi)$ is strongly (in this case, equivalently, weakly) closed under
  substitutions.

  In item~\ref{lem:Pphi:WR}, since $Q(\psi)=P(\phi)$ as $\phi$ is weakly random, it follows that
  $\cF\cap P(\phi)\neq\varnothing$, so $P(\phi)$ is weakly closed under substitutions.
\end{proof}

The next example shows why the hypotheses of Lemma~\ref{lem:Pphi} to get $P(\phi)$ weakly closed
under substitutions are crucial.

\begin{example}\label{ex:Pphinotclosedundersubst}
  Consider $\phi\in\HomT{\TkHypergraph[3]}$ that is the disjoint union of a clique and an
  anti-clique of the same size, that is, $\phi=\phi_\cN$ for the $\TkHypergraph[3]$-on $\cN$ over
  $[0,1]$ given by
  \begin{align*}
    \cN_E
    & \df
    \left\{x\in\cE_3 \;\middle\vert\;
    \max\{x_{\{1\}},x_{\{2\}},x_{\{3\}}\} < \frac{1}{2}\right\}.
  \end{align*}
  Since $\phi$ contains both a clique and an anti-clique of positive measure, it follows that
  $P(\phi)$ does not contain any models of size at least $3$. However, since $\TkHypergraph[3]$ is
  $2$-categorical, $P(\phi)$ must contain the unique model $K^{(3)}_2$ of size $2$. It then follows
  that $P(\phi)$ is not even weakly closed under substitutions as any substitution of any vertex of
  $K^{(3)}_2$ by $K^{(3)}_2$ must have size $3$.
\end{example}

\begin{definition}\label{def:compatibleblowup}
  Given a finite sequence $N=(N_0,\ldots,N_n)$ of finite $\cL$-structures with $\lvert N_i\rvert\geq
  2$ for every $i\in\{0,\ldots,n\}$, a \emph{recursive blow-up relative to $N$} is an
  $\cL$-structure $R$ with $V(R)\df\prod_{i=0}^n V(N_i)$ such that for every $j\in\{0,\ldots,n\}$
  and every $\sigma\in\prod_{i=0}^{j-1} V(N_i)$, every function $f\colon V(N_j)\to V(R)$ such that
  $f(v)\rest_{\{0,\ldots,j-1\}} = \sigma$ and $f(v)_j = v$ for every $v\in V(N_j)$ is an embedding
  of $N_j$ in $R$.

  The unique recursive blow-up $R$ relative to $N$ that has the smallest possible relation sets
  $P^R$ ($P\in\cL$) is called the \emph{conservative recursive blow-up relative to $N$} and is
  denoted $R^N$. Formally, it is given by $V(R^N)\df\prod_{i=0}^n V(N_i)$ and
  \begin{align*}
    P^{R^N}
    \df
    \biggl\{(\sigma,\alpha_j,\tau^j)_{j=0}^{k(P)} \in (V(R^N))_{k(P)}
    \;\bigg\vert\;
    &
    \sigma\in\prod_{\ell=0}^{i-1} V(N_\ell)\land\alpha\in P^{N_i}
    \\
    & \land
    \forall j\in [k(P)], \tau^j\in\prod_{\ell=i+1}^n V(N_\ell)
    \biggr\}
  \end{align*}
  for every $P\in\cL$.

  Given an infinite sequence $N=(N_i)_{i\in\NN}$ of finite $\cL$-structures with $\lvert
  N_i\rvert\geq 2$ for every $i\in\NN$, a \emph{compatible sequence of recursive blow-ups relative
    to $N$} is a sequence $R=(R_i)_{i\in\NN}$ such that
  \begin{enumerate}
  \item for every $i\in\NN$, $R_i$ is a recursive blow-up relative to $(N_0,\ldots,N_i)$;
  \item for every $i\in\NN$, every function $f\colon V(R_i)\to V(R_{i+1})$ such that
    $f(v)\rest_{\{0,\ldots,i\}} = v$ is an embedding of $R_i$ in $R_{i+1}$.
  \end{enumerate}
  We call $R$ \emph{conservative} if further $R_i = R^{(N_0,\ldots,N_i)}$ (it is easy to see that
  this is always compatible).
\end{definition}

\begin{lemma}\label{lem:comprecblowup}
  Let $\cF$ be a family of finite $\cL$-structures that is weakly closed under substitutions and
  closed under substructures and let $N = (N_i)_{i\in\NN}$ be a sequence in $\cF$ with $\lvert
  M_i\rvert\geq 2$ for every $i\in\NN$. Then there exists a compatible sequence $R=(R_i)_{i\in\NN}$
  of recursive blow-ups relative to $N$ with $R_i\in\cF$ for every $i\in\NN$.
\end{lemma}

\begin{proof}
  We construct the compatible sequence $R=(R_i)_{i\in\NN}$ inductively by setting $R_0\df N_0$ and
  given $R_i$, we enumerate the vertices of $R_i$ as $v^i_1,\ldots,v^i_{t_i}$, inductively define
  $F^i_0,\ldots,F^i_{t_i}$ by $F^i_0\df R_i$, let $F^i_{j+1}\in\cF$ be a standard substitution of
  $v^i_{j+1}$ in $F^i_j$ by $N_{i+1}$ and set $R_{i+1}\df F^i_{t_i}$. It is straightforward to check
  by induction that $R$ is a compatible sequence of recursive blow-ups relative to $N$ with
  $R_i\in\cF$ for every $i\in\NN$.
\end{proof}

\begin{definition}\label{def:recursiveblowup}
  Given an infinite sequence $N=(N_i)_{i\in\NN}$ of finite $\cL$-structures with $\lvert
  N_i\rvert\geq 2$ for every $i\in\NN$, we let the \emph{conservative recursive blow-up relative to
    $N$} be the $T_\cL$-on $\cN^N$ defined as follows. We let $V = (V_\ell)_{\ell\in\NN}$ be defined
  by $V_\ell\df V(N_\ell)$ and we define $\cN^N$ over the Cantor probability space
  $\Omega^V=(\prod_{\ell\in\NN} V_\ell, \cA, \nu^V)$ (see Definition~\ref{def:graphrecursiveblowup})
  by
  \begin{align*}
    \cN^N_P
    & \df
    \{x\in\cE_{k(P)}(\Omega^V) \mid
    \exists i\in\NN, R^{(N_0,\ldots,N_i)}\vDash P(t_i^P(x)) \},
  \end{align*}
  where $t_i^P(x)\in (\prod_{\ell=0}^i V_\ell)^{k(P)}$ is given by
  \begin{align}\label{eq:tiP}
    t_i^P(x)_j & \df x_{\{j\}}\rest_{\{0,\ldots,i\}} \qquad (j\in [k(P)]).
  \end{align}
\end{definition}

\begin{proposition}\label{prop:recblowup}
  Let $R=(R_i)_{i\in\NN}$ be a compatible sequence of recursive blow-ups relative to
  $N=(N_i)_{i\in\NN}$ and let $V = (V_i)_{i\in\NN}$ be given by $V_i\df V(N_i)$.

  Then $R$ is convergent and the following hold for its limit $\phi_R\in\HomT{T_\cL}$.
  \begin{enumerate}
  \item If $R$ is conservative, then $\phi_R = \phi_{\cN^N}$.%
    \label{prop:recblowup:conservative}
  \item There exists a $T_\cL$-on $\cH$ over $\Omega^V$ with $\phi_R=\phi_\cH$
    \begin{align*}
      \cN^N_P
      & \subseteq
      \cH_P
      \subseteq
      \cE_{k(P)}(\Omega^V)\setminus\cN^{\overline{N}}_P\text{ a.e.}
    \end{align*}
    for every $P\in\cL$, where $\overline{N}=(\overline{N}_i)_{i\in\NN}$ is the sequence of
    complementary canonical $\cL$-structures given by
    \begin{align*}
      V(\overline{N}_i) & \df V(N_i), &
      P^{\overline{N}_i} & \df (V(N_i))_{k(P)}\setminus P^{N_i} \qquad (P\in\cL).
    \end{align*}%
    \label{prop:recblowup:nonconservative}
  \item If $P(N)$ is the set of structures $M$ such that there exist infinitely many $i\in\NN$ with
    $M\cong N_i$, then $P(N)\subseteq P(\phi_R)$.%
    \label{prop:recblowup:P}
  \item If $\prod_{i\in\NN} (1 - 1/\lvert N_i\rvert) = 0$, then for every positive measure
    $A\subseteq\Omega^V$, there exists $i\in\NN$ such that $\tind(N_i,\cH\rest_A^F) > 0$ for every
    measure-isomorphism $F$ modulo $0$ from $\Omega^V_A$ to $\Omega^V$.%
    \label{prop:recblowup:zeroproduct}
  \end{enumerate}
\end{proposition}

\begin{proof}
  To show that $R$ is convergent, for each $i\in\NN$, define the Euclidean structure $\cN^i$ in
  $\cL$ over $\Omega^V$ by
  \begin{align*}
    \cN^i_P & \df \{x\in\cE_{k(P)}(\Omega^V) \mid R_i\vDash P(t_i^P(x))\},
  \end{align*}
  where $t_i^P(x)$ is given by~\eqref{eq:tiP}, that is, $\cN^i$ is the natural ``step'' Euclidean
  structure associated with $R_i$ over $\Omega^V$.

  First note that since $R$ is compatible, for every $i,j\in\NN$, we have
  \begin{equation}\label{eq:L1bound}
    \begin{aligned}
      & \!\!\!\!\!\!
      \sum_{P\in\cL} \nu^V(\cN^i_P\symdiff\cN^{i+j}_P)
      \\
      & \leq
      \sum_{P\in\cL}
      \nu^V(\{x\in\cE_{k(P)}(\Omega^V) \mid
      \exists a,b\in[k(P)],
      (a\neq b\land
      x_{\{a\}}\rest_{\{0,\ldots,i\}} = x_{\{b\}}\rest_{\{0,\ldots,i\}})
      \})
      \\
      & \leq
      \sum_{P\in\cL} \binom{k(P)}{2}\cdot\sum_{\sigma\in\prod_{\ell=0}^i V_\ell} \nu^V(K_{\sigma,V})^2
      \\
      & =
      \sum_{P\in\cL} \binom{k(P)}{2}\cdot\prod_{\ell=0}^i \lvert V_i\rvert^{-1}
      \xrightarrow{i\to\infty} 0.
    \end{aligned}
  \end{equation}
  Therefore, it follows that for every finite $\cL$-structure $K$, the limit
  $\lim_{i\to\infty}\tind(K,\cN^i)$ exists.

  On the other hand, it is also straightforward to check that for every finite $\cL$-structure $K$,
  we have
  \begin{align*}
    \left\lvert
    \lvert R_i\rvert^{\lvert K\rvert}\cdot\tind(K,\cN^i)
    -
    \lvert\Tind(K,R_i)\rvert
    \right\rvert
    & \leq
    \lvert R_i\rvert^{\lvert K\rvert} - (\lvert R_i\rvert)_{\lvert K\rvert}
    \leq
    O_K(\lvert R_i\rvert^{\lvert K\rvert-1})
  \end{align*}
  hence we get
  \begin{align*}
    \lim_{i\to\infty} \tind(K,R_i) = \lim_{i\to\infty} \tind(K,\cN^i),
  \end{align*}
  that is, $R = (R_i)_{i\in\NN}$ is convergent.

  \medskip

  Consider now the case when $R$ is conservative. Then the same argument used in~\eqref{eq:L1bound}
  gives
  \begin{align*}
    \sum_{P\in\cL} \nu^V(\cN^i_P\symdiff\cN^N_P)
    & \leq
    \sum_{P\in\cL} \binom{k(P)}{2}\cdot\prod_{\ell=0}^i \lvert V_i\rvert^{-1}
    \xrightarrow{i\to\infty} 0,
  \end{align*}
  so item~\ref{prop:recblowup:conservative} follows.

  \medskip

  To prove item~\ref{prop:recblowup:nonconservative}, note that~\eqref{eq:L1bound} implies that for
  each $P\in\cL$, the sequence of indicator functions $(\One_{\cN^i_P})_{i\in\NN}$ is convergent in
  $L^1(\cE_{k(P)}(\Omega^V))$, so let $f_P$ be their $L^1$-limit. Since $f_P$ is also the
  a.e.\ limit of $(\One_{\cN^i_P})_{i\in\NN}$, it must be a.e.~$\{0,1\}$-valued, so there exists
  $\cH_P$ such that $f_P = \One_{\cH_P}$ a.e. Finally, $L^1$-convergence implies that
  $\lim_{i\to\infty}\tind(K,\cN^i) = \tind(K,\cH)$.

  We claim that for every $x\in\cN^N_P$, there exists $i_0\in\NN$ such that $x\in\cN^i_P$ for every
  $i\geq i_0$. Indeed, if $x\in\cN^N_P$, then there exists $i_0\in\NN$ such that
  $R^{(N_0,\ldots,N_{i_0})}\vDash P(t_{i_0}^P(x))$. The definition of the conservative recursive
  blow-ups $R^{(N_0,\ldots,N_i)}$ implies that $R^{(N_0,\ldots,N_i)}\vDash P(t_i^P(x))$ for every
  $i\geq i_0$. From the minimality of the conservative recursive blow-ups, we get $R^i\vDash
  P(t_i^P(x))$, hence $x\in\cN^i_P$ for every $i\geq i_0$. Since $\One_{\cN^i_P}$ converges a.e.\ to
  $\One_{\cH_P}$, we conclude that $\cN^N_P \subseteq\cH_P$ a.e.

  By a symmetric argument, it follows that for every $x\in\cN^{\overline{N}}_P$, there exists
  $i_0\in\NN$ such that $x\in\cE_{k(P)}(\Omega^V)\setminus\cN^i_P$ for every $i\geq i_0$, from which
  we conclude that $\cN^{\overline{N}}_P\subseteq\cE_{k(P)}(\Omega^V)\setminus\cH_P$ a.e.\ and thus
  $\cH_P\subseteq\cE_{k(P)}(\Omega^V)\setminus\cN^{\overline{N}}_P$ a.e.

  \medskip

  Let us now show item~\ref{prop:recblowup:P}. Fix $M\in P(N)$ and let us show that $M\in
  P(\phi_R)$. By Lemma~\ref{lem:PWRset}, it is sufficient to show that for every positive measure
  $A\subseteq\Omega^V$ and every measure-isomorphism $F$ modulo $0$ from $\Omega^V_A$ to $\Omega^V$,
  we have $M\in Q(\cH\rest_A^F)$.

  Let then $\epsilon > 0$ be such that $\epsilon < 1/\lvert M\rvert$. By Lemma~\ref{lem:KsigmaV},
  there exists $t_0\in\NN$ such that for every $t\geq t_0$, there exists
  $\sigma\in\prod_{\ell=0}^{t-1} V_\ell$ such that $\nu^V(A\cap
  K_{\sigma,V})\geq(1-\epsilon)\cdot\nu^V(K_{\sigma,V})$. Since $M\in P(N)$, there exists $t\geq
  t_0$ such that $M\cong N_t$. Since $\{K_{(\sigma,u),V} \mid u\in V_t\}$ partitions $K_{\sigma,V}$
  into $\lvert V_t\rvert = \lvert M\rvert$ parts of equal measure, it follows that for every $u\in
  V_t$, we have
  \begin{align*}
    \nu^V(A\cap K_{(\sigma,u),V})
    & \geq
    \left(
    1 - \epsilon - \frac{\lvert M\rvert - 1}{\lvert M\rvert}
    \right)\cdot\nu^V(K_{\sigma,V})
    >
    0.
  \end{align*}
  Note now that if $x\in\cE_{V_t}(\Omega^V)$ is such that $x_{\{u\}}\in A\cap K_{(\sigma,u),V}$ for
  every $u\in V_t$, then $x\in\Tind(N_t,\cH\rest_A^F)$. Thus $\tind(N_t,\cH\rest_A^F) > 0$,
  hence $M=N_t\in Q(\cH\rest_A^F)$, as desired.

  \medskip

  It remains to show item~\ref{prop:recblowup:zeroproduct}. Suppose not, that is, suppose that there
  exist some positive measure $A\subseteq\Omega^V$ and some measure-isomorphism $F$ modulo $0$ from
  $\Omega^V_A$ to $\Omega^V$ such that for every $i\in\NN$, we have $\tind(N_i,\cH\rest_A^F) = 0$.

  Let $n\in\NN$ be large enough so that $\prod_{i=0}^{n-1} (1-1/\lvert N_i\rvert) < \nu^V(A)$ and
  let
  \begin{align*}
    \Sigma
    & \df
    \left\{\sigma\in\prod_{i=0}^{n-1} V_i \;\middle\vert\;
    \nu^V(A\cap K_{\sigma,V}) > 0
    \right\}.
  \end{align*}

  We claim that for every $m\in\{0,\ldots,n-1\}$ and every $\tau\in\prod_{i=1}^{m-1} V_i$, there
  exists $u_\tau\in V_m$ such that $(\tau,u_\tau)$ is not a prefix of any element of
  $\Sigma$. Suppose not, that is, suppose that there exist $m\in\{0,\ldots,n-1\}$ and
  $\tau\in\prod_{i=1}^{m-1} V_i$ such that for every $u\in V_m$, there exists some
  $\sigma^u\in\Sigma$ such that $(\tau,u)$ is a prefix of $\sigma^u$. But then the set of
  $x\in\cE_{V_m}(\Omega^V)$ such that $x_{\{u\}}\in A\cap K_{\sigma^u,V}$ for every $u\in V_m$ is a
  positive measure set that is contained in $\Tind(N_m,\cH\rest_A^F)$, contradicting the fact that
  $\tind(N_m,\cH\rest_A^F) = 0$. Thus the claim is proved.

  Let now $\Sigma^*$ be the set of $\sigma\in\prod_{i=0}^{n-1} V_i$ such that for every
  $m\in\{0,\ldots,n-2\}$, we have $u_{\sigma\rest_{\{0,\ldots,m-1\}}}\neq\sigma_m$. Our last claim
  shows that $\Sigma\subseteq\Sigma^*$. Now it is easy to see that
  \begin{align*}
    \nu^V(A)
    & =
    \sum_{\sigma\in\Sigma} \nu^V(A\cap K_{\sigma,V})
    \leq
    \sum_{\sigma\in\Sigma^*} \nu^V(K_{\sigma,V})
    =
    \prod_{i=0}^{n-1}\left(1 - \frac{1}{\lvert N_i\rvert}\right)
    <
    \nu^V(A),
  \end{align*}
  a contradiction. Thus, item~\ref{prop:recblowup:zeroproduct} is proved.
\end{proof}

\begin{theorem}\label{thm:persistence}
  The following are equivalent for a family $\cF$ of finite $\cL$-structures (up to isomorphism)
  containing at least one structure of size at least $2$.
  \begin{enumerate}
  \item The family $\cF$ is strongly persistent.%
    \label{thm:persistence:stronglypersistent}
  \item The family $\cF$ is weakly closed under substitutions and closed under substructures.%
    \label{thm:persistence:closed}
  \end{enumerate}

  Furthermore, if all predicate symbols of $\cL$ have arity at most $2$, then the above are also
  equivalent to:
  \begin{enumerate}[resume]
  \item The family $\cF$ is persistent.%
    \label{thm:persistence:persistent}
  \end{enumerate}
\end{theorem}

\begin{proof}
  The implication~\ref{thm:persistence:stronglypersistent}$\implies$\ref{thm:persistence:closed}
  follows from Lemma~\ref{lem:Pphi}\ref{lem:Pphi:WR} as $\cF = P(\phi)$ for some \emph{weakly
    random} $\phi\in\HomT{T_\cL}$.

  \medskip

  For the
  implication~\ref{thm:persistence:closed}$\implies$\ref{thm:persistence:stronglypersistent}, let
  $\cF'$ be the set of elements of $\cF$ of size at least $2$ and let $N = (N_i)_{i\in\NN}$ be an
  enumeration of all elements of $\cF'$ that repeats each element of $\cF'$ infinitely often. Since
  $\cF$ is weakly closed under substitutions and closed substructures, by
  Remark~\ref{rmk:substitutionunary}, it follows that $\cF = \cF'\cup\{K_0,F_1\}$ for some
  $\cL$-structure $F_1$ of size $1$ (and where $K_0$ is the trivial $\cL$-structure of size $0$).

  By Lemma~\ref{lem:comprecblowup}, there exists a compatible sequence $R=(R_i)_{i\in\NN}$ of
  recursive blow-ups relative to $N$ with $R_i\in\cF$ for every $i\in\NN$ and by
  Proposition~\ref{prop:recblowup}\ref{prop:recblowup:P}, we know that $R$ converges to some
  $\phi_R\in\HomT{T_\cL}$ such that $\cF'=P(N)\subseteq P(\phi_R)$ and since $P(\phi_R)$ is closed
  under substructures (see Lemma~\ref{lem:Pphi}) and $\cF = \cF'\cup\{K_0,F_1\}$, we must have
  $\cF\subseteq P(\phi_R)$. On the other hand, since $R_i\in\cF$, it follows that
  $P(\phi_R)\subseteq Q(\phi_R)\subseteq\cF$, hence $\cF = Q(\phi_R) = P(\phi_R)$ as desired.

  \medskip

  If all predicate symbols of $\cL$ have arity at most $2$, then
  implication~\ref{thm:persistence:persistent}$\implies$\ref{thm:persistence:closed} follows from
  Lemma~\ref{lem:Pphi}\ref{lem:Pphi:arity2} as $\cF = P(\phi)$ for some $\phi\in\HomT{T_\cL}$ and
  the implication~\ref{thm:persistence:stronglypersistent}$\implies$\ref{thm:persistence:persistent}
  is obvious.
\end{proof}

Again, the assumption of arity at most $2$ is crucial for the inclusion of
item~\ref{thm:persistence:persistent} in the equivalence of Theorem~\ref{thm:persistence} as
illustrated by Example~\ref{ex:Pphinotclosedundersubst}.

We conclude this section by observing operations that preserve the notions discussed so far. The
next proposition shows naturality of the operators $Q$ and $P$ and of the weak randomness property
in the sense that the operators $P$ and $Q$ commute with open interpretations and weak randomness is
preserved by open interpretations.

\begin{proposition}\label{prop:naturality}
  Let $I\colon T_1\leadsto T_2$ be an open interpretation. The following hold for
  $\phi\in\HomT{T_2}$.
  \begin{enumerate}
  \item We have $Q(\phi^I) = I(Q(\phi))$.%
    \label{prop:naturality:Q}
  \item We have $P(\phi^I) = I(P(\phi))$.%
    \label{prop:naturality:P}
  \item If $\phi$ is weakly random, then so is $\phi^I$.%
    \label{prop:naturality:WR}
  \end{enumerate}
\end{proposition}

\begin{proof}
  Item~\ref{prop:naturality:Q} follows directly from the definition of $\phi^I$, see~\eqref{eq:phiI}.

  \medskip

  Item~\ref{prop:naturality:P} follows directly from item~\ref{prop:naturality:Q} and the fact that
  if $\psi$ is a sub-object of $\phi$, then $\psi^I$ is a sub-object of $\phi^I$ and conversely,
  every sub-object of $\phi^I$ is of the form $\psi^I$ for some sub-object $\psi$ of $\phi$.

  \medskip

  Item~\ref{prop:naturality:WR} follows trivially from items~\ref{prop:naturality:Q}
  and~\ref{prop:naturality:P}.
\end{proof}

Before we proceed, we recall the notion of couplings and independent couplings of limits
from~\cite[Definitions~2.3, 2.4 and~2.5]{CR20b}, which played a key role in the study of the
natural quasirandomness properties $\UCouple[\ell]$ and $\UInduce[\ell]$ in that work.

\begin{definition}\label{def:couplings}
  Given canonical theories $T_1$ and $T_2$ in finite relational languages $\cL_1$ and $\cL_2$,
  respectively, the \emph{disjoint union} $T_1\cup T_2$ is the canonical theory in the disjoint
  union language $\cL_1\disjcup\cL_2$ whose axioms are those of $T_1$ (about predicate symbols in
  $\cL_1$) and those of $T_2$ (about predicate symbols in $\cL_2$), that is, the models of $T_1\cup
  T_2$ correspond to a model of $T_1$ and a model of $T_2$ \emph{on the same vertex set}.

  A \emph{coupling} of $\phi_1\in\HomT{T_1}$ and $\phi_2\in\HomT{T_2}$ is a limit
  $\psi\in\HomT{T_1\cup T_2}$ such that $\phi_i = \psi^{I_i}$ for every $i\in[2]$, where $I_i\colon
  T_i\leadsto T_1\cup T_2$ is the \emph{structure-erasing interpretation} that acts identically on
  predicate symbols of $T_i$.

  The \emph{independent coupling} of $\phi_1\in\HomT{T_1}$ and $\phi_2\in\HomT{T_2}$ is the limit
  $\phi_1\otimes\phi_2\in\HomT{T_1\cup T_2}$ given by
  \begin{align*}
    (\phi_1\otimes\phi_2)(M)
    & \df
    \frac{\lvert\Aut(M_1)\rvert\cdot\lvert\Aut(M_2)\rvert}{\lvert M\rvert!\cdot\lvert\Aut(M)\rvert}
    \cdot\phi_1(M_1)\cdot\phi_2(M_2),
  \end{align*}
  where $M_i\df I_i(M)$. Alternatively, if $\cN^i$ ($i\in[2]$) is a $T_i$-on over $\Omega_i$ with
  $\phi_{\cN^i}=\phi_i$, then we have $\phi_1\otimes\phi_2 = \phi_{\cN^1\otimes\cN^2}$ for the
  $(T_1\cup T_2)$-on $\cN^1\otimes\cN^2$ over the product space $\Omega_1\otimes\Omega_2$ given by
  \begin{align*}
    (\cN^1\otimes\cN^2)_P & \df \{x\in\cE_{k(P)}(\Omega_1\otimes\Omega_2) \mid \pi_{i,k(P)}(x)\in\cN^i_P\}
  \end{align*}
  whenever $P\in\cL_i$ ($i\in[2]$), where
  $\pi_{i,k(P)}\colon\cE_{k(P)}(\Omega_1\otimes\Omega_2)\to\cE_{k(P)}(\Omega_i)$ is the natural
  projection.
\end{definition}

The next proposition says that weak randomness is preserved under independent couplings.

\begin{proposition}\label{prop:indepcoup}
  If $\phi_1\in\HomT{T_1}$ and $\phi_2\in\HomT{T_2}$ are weakly random, then so is their independent
  coupling $\phi_1\otimes\phi_2$.
\end{proposition}

\begin{proof}
  Let $\cN^i$ be a $T_i$-on over $\Omega_i=(X_i,\cA_i,\mu_i)$ such that $\phi_i=\phi_{\cN^i}$ and
  let $\Omega\df\Omega_1\otimes\Omega_2$. It is clear from the definition of $\phi_1\otimes\phi_2$
  that for every $M\in\cM[T_1\cup T_2]$, we have $M\in Q(\phi_1\otimes\phi_2)$ if and only if
  $I_1(M)\in Q(\phi_1)$ and $I_2(M)\in Q(\phi_2)$, where $I_i\colon T_i\leadsto T_1\cup T_2$
  ($i\in[2]$) is the structure-erasing interpretation.

  By Lemma~\ref{lem:PWRset}\ref{lem:PWRset:P}, to show that $\phi_1\otimes\phi_2$ is weakly random,
  it is sufficient to show that for every positive measure $A\subseteq\Omega$ and every
  measure-isomorphism $F$ modulo $0$ from $\Omega_A$ to $\Omega$, we have $Q(\phi_1\otimes\phi_2) =
  Q((\cN^1\otimes\cN^2)\rest_A^F)$.

  Let $M\in Q(\phi_1\otimes\phi_2)$ and let us show that $M\in Q((\cN^1\otimes\cN^2)\rest_A^F)$. For
  each $i\in[2]$, let $M_i\df I_i(M)$ and let
  \begin{align*}
    B
    & \df
    \{(x,y)\in\cE_{V(M)}(\Omega_1)\times X_2 \mid
    x\in\Tind(M_1,\cN^1)
    \land
    \forall v\in V(M), (x_{\{v\}},y)\in A\}.
  \end{align*}

  Our objective is to show that $(\mu_1\otimes\mu_2)(B) > 0$. To do so, for each $y\in X_2$, let
  \begin{align*}
    A(y) & \df \{x\in X_1 \mid (x,y)\in A\}
  \end{align*}
  and note that Fubini's Theorem implies that the set
  \begin{align*}
    \widetilde{X}_2 & \df \{y\in X_2 \mid \mu_1(A(y)) > 0\}
  \end{align*}
  has positive $\mu_2$-measure.

  Since $\phi_1$ is weakly random, for every $y\in\widetilde{X}_2$ and every measure isomorphism
  $\widetilde{F}_y$ modulo $0$ from $(\Omega_1)_{A(y)}$ to $\Omega_1$, we have
  $\tind(M_1,\cN^1\rest_{A(y)}^{\widetilde{F}_y}) > 0$, thus Fubini's Theorem gives
  \begin{align*}
    (\mu_1\otimes\mu_2)(B)
    & \geq
    \int_{\widetilde{X}_2} \tind(M_1,\cN^1\rest_{A(y)}^{\widetilde{F}_y})\cdot\mu_1(A(y))^{\lvert M\rvert}\ d\mu_2(y)
    >
    0.
  \end{align*}

  For every $x\in\Tind(M_1,\cN^1)\subseteq\cE_{V(M)}(\Omega_1)$, define the set
  \begin{align*}
    B(x)
    & \df
    \{y\in X_2 \mid (x,y)\in B\}
    =
    \{y\in X_2 \mid \forall v\in V(M), (x_{\{v\}},y)\in A\}
  \end{align*}
  and note that Fubini's Theorem again implies that the set
  \begin{align*}
    \widetilde{\Tind}(M_1,\cN^1)
    & \df
    \{x\in\Tind(M_1,\cN^1) \mid \mu_2(B(x)) > 0\}
  \end{align*}
  has positive $\mu_1$-measure. Since $\phi_2$ is weakly random, for every
  $x\in\widetilde{\Tind}(M_1,\cN^1)$ and every measure isomorphism $\widetilde{G}_x$ modulo $0$ from
  $(\Omega_2)_{B(x)}$ to $\Omega_2$, we have $\tind(M_2,\cN^2\rest_{B(x)}^{\widetilde{G}_x}) > 0$,
  thus Fubini's Theorem gives
  \begin{align*}
    \tind(M,(\cN^1\otimes\cN^2)\rest_A^F)
    & \geq
    \int_{\widetilde{\Tind}(M_1,\cN^1)}
    \tind(M_2,\cN^2\rest_{B(x)}^{\widetilde{G}_x})\cdot\mu_2(B(x))^{\lvert M\rvert}
    \ d\mu_1(x)
    >
    0,
  \end{align*}
  concluding the proof.
\end{proof}

\begin{remark}
  As we mentioned before, weak randomness can be seen as a weakening of the natural quasirandomness
  property $\UInduce[1]$ of~\cite{CR20b}. Since $\UInduce[1]$ (and more generally, $\UInduce[\ell]$)
  is not preserved under independent couplings, one can consider the class $\UInduce_\otimes[\ell]$
  that is the closure of $\UInduce[\ell]$ under independent couplings and open interpretations and
  in~\cite[\S10]{CR20b}, it was asked if any of these classes yields a meaningful notion of
  randomness or if they are already ``too large''. It was already noted in~\cite{CR20b} that the
  quasirandom permuton (see Proposition~\ref{prop:QRpermuton}) is in $\UInduce_\otimes[\ell]$ for
  every $\ell\in\NN_+$ and that even the largest class $\UInduce_\otimes[1]$ among the
  $\UInduce_\otimes[\ell]$ does not contain all limits.

  Since $\UInduce[1]$ implies weak randomness, from
  Propositions~\ref{prop:naturality}\ref{prop:naturality:WR} and~\ref{prop:indepcoup} it follows
  that every element of $\UInduce_\otimes[1]$ is weakly random; this further justifies the adjective
  ``weak'' in weak randomness: it is a quasirandomness notion weaker than the weakening
  $\UInduce_\otimes[1]$ of $\UInduce[1]$ that is still meaningful.

  Let us point out that there are weakly random limits that are not in $\UInduce_\otimes[1]$:
  namely, one can show that if $W$ is a universal weakly random \emph{$\{0,1\}$-valued} graphon of
  $\TGraph$ (e.g., $\phi_W=\phi_G^*$ as in Lemma~\ref{lem:PphiGQphiG} for an enumeration $G =
  (G_m)_{m\in\NN}$ of all finite graphs of size at least $2$), then $\phi$ is weakly random but is
  not in $\UInduce_\otimes[1]$. However, since the length of the proof outweighs its enlightenment
  value, we omit it.
\end{remark}

Recall from Definition~\ref{def:AEHP} that a trivial limit $\phi\in\HomT{T}$ is any limit of the
form $\phi=\phi_\cN$ for some theon $\cN$ whose peons all have measure in $\{0,1\}$. For general
couplings, the next proposition says that the coupling of a trivial limit with a weakly random limit
is weakly random.

\begin{proposition}\label{prop:coup}
  If $\psi$ is a coupling of a trivial $\phi_1\in\HomT{T_1}$ and a weakly random
  $\phi_2\in\HomT{T_2}$, then $\psi$ is weakly random.
\end{proposition}

\begin{proof}
  Let $\cL_1$ and $\cL_2$ be the languages of $T_1$ and $T_2$, respectively.

  Since $\phi_1$ is trivial, it follows that $\cL_1$ can be partitioned into
  $\cL_1=\cL_1^0\cup\cL_1^1$ so that for every $M_1\in Q(\phi_1)$ every $P\in\cL_1$, we have
  \begin{align*}
    P^{M_1} & =
    \begin{dcases*}
      \varnothing, & if $P\in\cL_1^0$,\\
      (V(M_1))_{k(P)}, & if $P\in\cL_1^1$.
    \end{dcases*}
  \end{align*}
  This implies that if $\xi$ is a coupling of some $\zeta\in\HomT{T_2}$ with $\phi_1$, then
  \begin{align}\label{eq:Qxi}
    Q(\xi) & = \{\widehat{M}_2 \mid M_2\in Q(\zeta)\},
  \end{align}
  where $\widehat{M}_2\in\cM_{V(M_2)}[T_1\cup T_2]$ is given by
  \begin{align*}
    P^{\widehat{M}_2} & \df
    \begin{dcases*}
      P^{M_2}, & if $P\in\cL_2$,\\
      \varnothing, & if $P\in\cL_1^0$,\\
      (V(M_2))_{k(P)}, & if $P\in\cL_1^1$.
    \end{dcases*}
  \end{align*}

  Now since $\phi_1$ is trivial, $\phi_1$ is the only sub-object of $\phi_1$, which means that every
  sub-object $\psi'$ of $\psi$ is a coupling of $\phi_1$ with some sub-object $\phi_2'$ of
  $\phi_2$. Since $\phi_2$ is weakly random, we have $Q(\phi_2')= Q(\phi_2)$, hence $Q(\psi') =
  Q(\psi)$ follows since the right-hand side of~\eqref{eq:Qxi} is the same for
  $(\xi,\zeta)=(\psi,\phi_2)$ and $(\xi,\zeta)=(\psi',\phi_2')$. Therefore $\psi$ is weakly random.
\end{proof}

As a simple application of Propositions~\ref{prop:naturality} and~\ref{prop:indepcoup} above, let us
prove Proposition~\ref{prop:agreementsofQRpermuton} that says that the graphon of agreements of the
quasirandom permuton (see Figure~\ref{fig:agreementsgraphon}) is a universal weakly random limit of
$\TPermGraph$ by showing that the quasirandom permuton $\psi_{\QR}\in\HomT{\TPerm}$ has the same
property for $\TPerm$. We point the reader interested in the theories of limits of permutations and quasirandom
permutations to~\cite{Coo04,KP13,CKNPSV20}.

Recall that the quasirandom permuton is given by
$\psi_{\QR}\df\phi_{\cN^{\QR}}$, where $\cN^{\QR}$ is the $\TPerm$-on over $[0,1]^2$ given by
\begin{align*}
  \cN^{\QR}_{\prec_i} & \df \{x\in\cE_2([0,1]^2) \mid \pi_i(x_{\{1\}}) < \pi_i(x_{\{2\}})\}
  \qquad (i\in[2]),
\end{align*}
where $\pi_i\colon[0,1]^2\to[0,1]$ is the projection onto the $i$th coordinate.

\begin{proposition}\label{prop:QRpermuton}
  The quasirandom permuton $\psi_{\QR}$ is a universal weakly random limit of $\TPerm$.
\end{proposition}

\begin{proof}
  It is straightforward to check that $Q(\psi_{\QR}) = \cM[\TPerm]$. On the other hand, $\psi_{\QR}$
  is the independent coupling $\psi\otimes\psi$ of the unique limit $\psi\in\HomT{\TLinOrder}$ of
  the theory of (strict) linear orders with itself. Since $\psi$ is obviously weakly random (as
  $\TLinOrder$ is finitely categorical), by Proposition~\ref{prop:indepcoup}, it follows that
  $\psi_{\QR}$ is weakly random.
\end{proof}

We can now derive Proposition~\ref{prop:agreementsofQRpermuton} that says that the graphon of
agreements of the quasirandom permuton is universal weakly random limit of $\TPermGraph$ as an easy
consequence.

\begin{proofof}{Proposition~\ref{prop:agreementsofQRpermuton}}
  The graphon of agreements of the quasirandom permuton represents the limit $\phi\df\psi_{\QR}^I$
  for the open interpretation $I\colon\TGraph\leadsto\TPerm$ given by
  \begin{align*}
    I(E)(x,y) & \df x\neq y \land (x\prec_1 y \tot x\prec_2 y),
  \end{align*}
  so $\phi$ is weakly random by Propositions~\ref{prop:naturality}\ref{prop:naturality:WR}
  and~\ref{prop:QRpermuton}.

  Finally, Proposition~\ref{prop:naturality}\ref{prop:naturality:Q} implies $Q(\phi) =
  I(Q(\psi_{\QR})) = I(\TPerm) = \TPermGraph$.
\end{proofof}

\begin{remark}\label{rmk:TPermWR}
  It is easy to see that the same permutations used in the proof of
  Proposition~\ref{prop:agreementsofpermutationtheory} can be used to show that $\cM[\TPerm]$ is
  closed under substitutions but not primally almost finite, hence $\TPerm\notin\WR$. However, let
  us point out that had we proved only the result for $\TPerm$, this would not have immediately
  implied Proposition~\ref{prop:agreementsofpermutationtheory} as primality is not necessarily
  preserved under open interpretations (even though closure under substitutions is).
\end{remark}

\section{What about weak randomness in general?}
\label{sec:WR:univ}

In this brief section we provide a partial generalization of Theorem~\ref{thm:graphs:WR} of
Section~\ref{sec:WR:graph} to universal theories in finite relational languages. For the easier
direction, we will only be able to generalize Lemma~\ref{lem:graphs:primallyalmostfinite->WR} when
all arities are at most $2$ (Proposition~\ref{prop:monprimalmostfinite->WR}) and even though the
harder direction will generalize directly in Proposition~\ref{prop:WR->primalmostfinite} below, this
naive generalization is essentially empty when all arities are at least $3$, as in this case there
are only finitely many prime structures (see Remark~\ref{rmk:prime3}). It is not clear at this point
what form a characterization of $\WR$ should take in the presence of higher arity predicates.

\begin{definition}\label{def:WRgeneral}
  We say that a canonical theory $T$ in a finite relational language has the \emph{weakly random
    \Erdos--Hajnal property} (abbreviated $T\in\WR$) if every $\phi\in\HomT{T}$ has a weakly random
  sub-object.
\end{definition}

\begin{proposition}\label{prop:monprimalmostfinite->WR}
  Let $\cL$ be a finite relational language whose predicate symbols have arity at most $2$ and let
  $T$ be a canonical theory in $\cL$. If $\cM[T]$ is monochromatically primally almost finite, then
  $T\in\WR$.
\end{proposition}

\begin{proof}
  We prove this by the contra-positive. Suppose $T\notin\WR$ and let us show that the set $\cP$ of
  monochromatic prime models of $T$ is not almost finite. By Lemma~\ref{lem:almostfinite}, it is
  sufficient to present a sequence $(F_n)_{n\in\NN}$ of finite monochromatic prime models of $T$
  such that $F_n$ is not a substructure of $F_m$ whenever $n < m$.

  Since $T\notin\WR$, there must exist a limit $\phi\in\HomT{T}$ that does not contain any weakly
  random sub-object.

  We now construct a sequence $(\phi_n)_{n\in\NN}$ of sub-objects of $\phi$ and a sequence
  $(F_n)_{n\in\NN}$ of finite prime models of $T$ satisfying the following.
  \begin{enumerate}
  \item For every $n\in\NN$, $\phi_{n+1}$ is a sub-object of $\phi_n$.
  \item For every $n\in\NN$, $F_n\in Q(\phi_n)\setminus Q(\phi_{n+1})$.
  \end{enumerate}
  We construct these sequences inductively as follows.
  \begin{enumerate}[label={\arabic*.}]
  \item We claim that there exists a sub-object $\phi_0$ of $\phi$ such that there exists
    $M_1\in\cM_1[T]$ with $\phi_0(M_1) = 1$ (and thus all $M\in\cM_1[T]\setminus\{M_1\}$ have
    $\phi_0(M)=0$). Indeed, if $M_1\in\cM_1[T]$ is such that $\phi(M_1) > 0$ and $\cN$ is an
    Euclidean structure over $\Omega$ with $\phi_\cN=\phi$, then $A\df\Tind(M_1,\cN)$ is a positive
    measure set, so for any measure-isomorphism $F$ modulo $0$ from $\Omega_A$ to $\Omega$, the
    sub-object $\phi_0\df\phi_{\cN\rest_A^F}$ satisfies the desired property.
  \item Given $\phi_n\in\HomT{T}$, since $\phi_n$ is a sub-object of $\phi$, we know that $\phi_n$
    is not weakly random, so there exists $N_n\in Q(\phi_n)\setminus P(\phi_n)$. Let $\cP_n$ be the
    set of substructures of $N_n$ that are prime. By Lemma~\ref{lem:Pphi}\ref{lem:Pphi:arity2}, we
    know that $P(\phi_n)$ is \emph{strongly} closed under substitutions and since $N_n\in S(\cP_n)$,
    there must exist $F_n\in\cP_n\setminus P(\phi_n)$ and since $Q(\phi_n)$ is closed under
    substructures, we get $F_n\in Q(\phi_n)\setminus P(\phi_n)$. From the definition of $P(\phi_n)$,
    it then follows that there exists a sub-object $\phi_{n+1}$ of $\phi_n$ (hence also of $\phi$)
    such that $F_n\in Q(\phi_n)\setminus Q(\phi_{n+1})$.
  \end{enumerate}

  Let now $n,m\in\NN$ be such that $n < m$. By induction, we know that $\phi_m$ is a sub-object of
  $\phi_{n+1}$, so $Q(\phi_m)\subseteq Q(\phi_{n+1})$, which in turn implies that $F_n\in
  Q(\phi_n)\setminus Q(\phi_m)$. Since $Q(\phi_m)$ is closed under substructures and $F_m\in
  Q(\phi_m)$, it follows that $F_n$ is not a substructure of $F_m$.

  Finally, since all $\phi_n$ are also sub-objects of $\phi_0$, we must have
  $Q(\phi_n)\cap\cM_1[T]\subseteq Q(\phi_0)\cap\cM_1[T] = \{M_1\}$. This implies that for every
  \emph{unary} predicate symbol $P\in\cL$ and every $n\in\NN$, we have $M_1\vDash\forall x, P(x)$ if
  and only if $F_n\vDash\forall x, P(x)$ (otherwise, we would have
  $Q(\phi_n)\cap\cM_1[T]\neq\{M_1\}$). Thus the $F_n$ are monochromatic.
\end{proof}

\begin{lemma}\label{lem:noweaklyrandomsubobject}
  Let $N = (N_i)_{i\in\NN}$ be a sequence of prime $\cL$-structures of size at least $2$ such that
  for each $i\in\NN$, there are finitely many $j\in\NN$ such that $N_i$ is a substructure of $N_j$,
  let $R = (R_i)_{i\in\NN}$ be a compatible sequence of recursive blow-ups relative to $N$ and let
  $\phi_R\in\HomT{T_\cL}$ be the limit of $R$. If $\prod_{i\in\NN} (1 - 1/\lvert N_i\rvert) = 0$,
  then $\phi_R$ does not have any weakly random sub-object.
\end{lemma}

\begin{proof}
  Let $V$ and $\cH$ be as in Proposition~\ref{prop:recblowup}. Suppose toward a contradiction that
  $\phi_R$ has a weakly random sub-object. By Lemma~\ref{lem:PWRset}, there exists a positive
  measure set $A\subseteq\Omega^V$ and a measure-isomorphism $F$ modulo $0$ from $\Omega^V_A$ to
  $\Omega^V$ such that $\cH\rest_A^F$ is weakly random. By
  Proposition~\ref{prop:recblowup}\ref{prop:recblowup:zeroproduct}, there exists $i_0\in\NN$ such
  that $\tind(N_{i_0},\cH\rest_A^F) > 0$.

  Let $j_0\df\max\{j \mid N_{i_0}\text{ is a substructure of } N_j\} < \infty$. Since
  $\{K_{\sigma,V} \mid \sigma\in\prod_{\ell=0}^{j_0} V_\ell\}$ partitions $\Omega^V$, there must
  exist $\sigma\in\prod_{\ell=0}^{j_0} V_\ell$ such that $\nu^V(A\cap K_{\sigma,V}) > 0$.

  From the definition of $\cH$, it follows that for every measure-isomorphism $\widetilde{F}$ modulo
  $0$ from $\Omega^V_{K_{\sigma,V}}$ to $\Omega^V$, we have
  $\phi_{\cH\rest_{K_\sigma,V}^{\widetilde{F}}} = \phi_{R'}$ for the sequence $R' =
  (R'_i)_{i\in\NN}$ given by $R'_i\df R_{j_0 + 1 + i}\rest_{U_i}$, where
  \begin{align*}
    U_i
    & \df
    \left\{\tau\in\prod_{\ell=0}^{j_0 + 1 + i} V_\ell \;\middle\vert\;
    \tau\rest_{\{0,\ldots,j_0\}} = \sigma
    \right\}.
  \end{align*}
  Note also that $R'$ is a compatible sequence of recursive blow-ups relative to the shifted
  sequence $N' = (N'_i)_{i\in\NN}$ given by $N'_i\df N_{j_0+1+i}$.

  We claim now that $\tind(N_{i_0},\cH\rest_{K_{\sigma,V}}^{\widetilde{F}}) = 0$. Indeed, since
  $N_{i_0}$ is prime, for this density to be positive, $N_{i_0}$ must be a substructure of
  infinitely many $R'_i$, but since $R'_i\in S(\{N_j \mid j\geq j_0+1\})$,
  Lemma~\ref{lem:primesubstructure} says that this can only happen if $N_{i_0}$ is a substructure of
  some $N_j$ with $j\geq j_0+1$, which would contradict the definition of $j_0$.

  Finally, this is a contradiction since $\cH\rest_A^F$ was assumed to be weakly random but
  $N_{i_0}\in Q(\cH\rest_A^F)\setminus P(\cH\rest_A^F)$ as $\nu^V(A\cap K_{\sigma,V}) > 0$ and
  $\tind(N_{i_0},\cH\rest_{K_{\sigma,V}}^{\widetilde{F}}) = 0$.
\end{proof}

\begin{proposition}\label{prop:WR->primalmostfinite}
  Let $T$ be a canonical theory such that $\cM[T]$ is weakly closed under substitutions. If
  $T\in\WR$, then $T$ is primally almost finite.
\end{proposition}

Before we prove Proposition~\ref{prop:WR->primalmostfinite}, let us note that it is completely
trivial when all arities are at least $3$ as in this case there are only finitely many prime
structures by Remark~\ref{rmk:prime3}.

\begin{proof}
  We prove this by the contra-positive. Suppose $\{N'_i \mid i\in\NN\}$ is an infinite antichain of
  prime models of $T$ and without loss of generality, assume every $N'_i$ has size at least $2$ (as
  $\cM_0[T]\cup\cM_1[T]$ is finite).

  For each $n\in\NN$, let $r_n\in\NN_+$ be large enough so that $(1-1/\lvert N'_i\rvert)^{r_n}\leq
  1/2$ and for each $\ell\in\NN$, let $N_\ell\df N'_n$ for the unique $n\in\NN$ such that
  $\sum_{m=0}^{n-1} r_m\leq\ell < \sum_{m=0}^n r_m$. Clearly, for each $\ell\in\NN$, there exist
  exactly $r_\ell$ values of $t\in\NN$ such that $N_\ell$ is a substructure of $N_t$. Note also that
  \begin{align*}
    \prod_{\ell\in\NN}\left(1 - \frac{1}{\lvert N_\ell\rvert}\right)
    & =
    \prod_{m\in\NN}\left(1 - \frac{1}{\lvert N'_m\rvert}\right)^{r_m}
    \leq
    \prod_{m\in\NN} \frac{1}{2}
    =
    0.
  \end{align*}
  Since $\cM[T]$ is weakly closed under substitutions, by Lemma~\ref{lem:comprecblowup}, there
  exists a compatible sequence $R = (R_\ell)_{\ell\in\NN}$ of recursive blow-ups relative to
  $N=(N_\ell)_{\ell\in\NN}$ with $R_\ell\in\cM[T]$ for every $\ell\in\NN$, and by
  Lemma~\ref{lem:noweaklyrandomsubobject}, the limit $\phi_R\in\HomT{T}$ of $R$ does not have any
  weakly random sub-object, hence $T\notin\WR$.
\end{proof}

Let us conclude this section by observing operations that preserve $\WR$ (at the level of
theories). The next proposition shows naturality (at the level of theories) of $\WR$, that is, it is
preserved by open interpretations.

\begin{proposition}\label{prop:WRnatural}
  If $I\colon T_1\leadsto T_2$ is an open interpretation and $T_2\in\WR$, then $I(T_2)\in\WR$.
\end{proposition}

\begin{proof}
  Follows from Proposition~\ref{prop:naturality}\ref{prop:naturality:WR}, the fact that every
  $\phi\in\HomT{I(T_2)}$ is of the form $\phi=\psi^I$ for some $\psi\in\HomT{T_2}$ and the fact that
  if $\psi$ is a sub-object of $\phi$, then $\psi^I$ is a sub-object of $\phi^I$ and conversely,
  every sub-object of $\phi^I$ is of the form $\psi^I$ for some sub-object $\psi$ of $\phi$.
\end{proof}

It is easy to see that $\WR$ is not preserved under disjoint unions of theories (see
Definition~\ref{def:couplings}): the theory of linear orders $\TLinOrder$ satisfies $\WR$ (as it is
finitely categorical) but the theory of permutations $\TPerm=\TLinOrder\cup\TLinOrder$ does not
satisfy $\WR$ (see Remark~\ref{rmk:TPermWR}). However, the next proposition says that $\WR$ at least
interacts well with disjoint unions with theories with $\AEHP$ (see Definition~\ref{def:AEHP}).

\begin{proposition}\label{prop:AEHPWRindepcoup}
  Let $T_1$ and $T_2$ be universal theories and suppose $T_1\in\AEHP$. Then the following hold.
  \begin{enumerate}
  \item If $T_2\in\AEHP$, then $T_1\cup T_2\in\AEHP$.%
    \label{prop:AEHPWRindepcoup:AEHP}
  \item If $T_2\in\WR$, then $T_1\cup T_2\in\WR$.%
    \label{prop:AEHPWRindepcoup:WR}
  \end{enumerate}
\end{proposition}

To prove this proposition, we will need the following result from~\cite{CR20b} on theons
representing couplings (see Definition~\ref{def:couplings}).

\begin{proposition}[\protect{\cite[Proposition~4.3]{CR20b}}]\label{prop:theoncoupling}
  Let $\psi\in\HomT{T_1\cup T_2}$ be a coupling of $\phi_1\in\HomT{T_1}$ and $\phi_2\in\HomT{T_2}$
  and let $\cN^1$ be a $T_1$-on over $\Omega$ such that $\phi_1 = \phi_{\cN^1}$. Then there exists a
  $(T_1\cup T_2)$-on $\cH$ over $\Omega\otimes\Omega$ such that $\psi=\phi_\cH$ and $\cH_P =
  \cN_P\times\cE_{k(P)}(\Omega)$ for every predicate symbol $P$ in the language of $T_1$ (when we
  naturally identify $\cE_{k(P)}(\Omega\otimes\Omega)$ with
  $\cE_{k(P)}(\Omega)\times\cE_{k(P)}(\Omega)$).
\end{proposition}

\begin{proofof}{Proposition~\ref{prop:AEHPWRindepcoup}}
  Let $\psi\in\HomT{T_1\cup T_2}$ and let $I_i\colon T_i\leadsto T_1\cup T_2$ ($i\in[2]$) be the
  structure-erasing interpretation. Then $\psi$ is a coupling of $\phi_1\df\psi^{I_1}$ and
  $\phi_2\df\psi^{I_2}$.

  Let $\cN^1$ be a $T_1$-on over $\Omega=(X,\cA,\mu)$ such that $\phi_1=\phi_{\cN^1}$. Since
  $T_1\in\AEHP$, by~\cite[Theorem~5.11]{CM22}, there exists a positive measure set $A\subseteq X$
  and a measure-isomorphism $F$ modulo $0$ from $\Omega_A$ to $\Omega$ such that
  $\phi_{\cN^1\rest_A^F}$ is trivial.

  Let now $\cH$ be the $(T_1\cup T_2)$-on over $\Omega'\df\Omega\otimes\Omega$ given by
  Proposition~\ref{prop:theoncoupling} and let $\mu'\df\mu\otimes\mu$ be the underlying measure of
  $\Omega'$. Let also $A'\df A\times X$ and let $F'\df F\otimes\id_X$ be the measure-isomorphism
  modulo $0$ from $\Omega'_{A'}$ to $\Omega'$ that acts as $F$ on the first coordinate and acts
  identically on the second coordinate.

  Suppose $T_2\in\AEHP$. Since $I_2(\cH\rest_{A'}^{F'}) = I_2(\cH)\rest_{A'}^{F'}$ is a $T_2$-on,
  by~\cite[Theorem~5.11]{CM22}, there exists a positive $\mu'_{A'}$-measure set $B\subseteq X\times
  X$ such that $I_2(\cH)\rest_{A'}^{F'}\rest_B^{\widetilde{F}}$ is trivial for every
  measure-isomorphism $\widetilde{F}$ modulo $0$ from $(\Omega'_{A'})_B$ to $\Omega'_{A'}$.

  Set $B'\df B\cap A'$ so that $B'$ is a positive $\mu'$-measure set such that
  $I_2(\cH)\rest_{B'}^{\widetilde{F}\comp F'}$ is trivial. Note now that since $\cH_P =
  \cN_P\times\cE_{k(P)}(\Omega)$ for every predicate symbol $P$ in the language of $T_1$, we get
  $\phi_{I_1(\cH)\rest_{A'}^{F'}} = \phi_{\cN^1\rest_A^F}$, which is a trivial limit. Since
  $\phi_{I_1(\cH)\rest_{B'}^{\widetilde{F}\comp F'}}$ is a sub-object of $\phi_{\cN^1\rest_A^F}$, it
  must also be trivial. Hence, $\psi=\phi_{\cH\rest_{B'}^{\widetilde{F}\comp F'}}$ must be trivial
  as it is a coupling of two trivial limits $\phi_{\cH\rest_{B'}^{\widetilde{F}\comp
      F'}}^{I_1}=\phi_{I_1(\cH)\rest_{B'}^{\widetilde{F}\comp F'}}$ and
  $\phi_{\cH\rest_{B'}^{\widetilde{F}\comp F'}}^{I_2}=\phi_{I_2(\cH)\rest_{B'}^{\widetilde{F}\comp
      F'}}$. Thus item~\ref{prop:AEHPWRindepcoup:AEHP} is proved.

  \medskip

  For item~\ref{prop:AEHPWRindepcoup:WR}, we make the same construction but taking $B\subseteq X\times
  X$ with positive $\mu'_{A'}$-measure such that $I_2(\cH)\rest_{A'}^{F'}\rest_B^{\widetilde{F}}$ is
  weakly random as guaranteed by $T_2\in\WR$. Then $\psi=\phi_\cH$ must be weakly random by
  Proposition~\ref{prop:coup} as it is a coupling of a trivial limit
  $\phi_{\cH\rest_{B'}^{\widetilde{F}\comp F'}}^{I_1}=\phi_{I_1(\cH)\rest_{B'}^{\widetilde{F}\comp
      F'}}$ with a weakly random limit $\phi_{\cH\rest_{B'}^{\widetilde{F}\comp
      F'}}^{I_2}=\phi_{I_2(\cH)\rest_{B'}^{\widetilde{F}\comp F'}}$.
\end{proofof}

\section{Conclusion and open problems}
\label{sec:concl}

In this paper we studied the notion of weak randomness, a weakening of the quasirandomness property
$\UInduce[1]$ (see~\cite{CR20b}). In the language of graphs, a graphon is weakly random if the set
of finite graphs $G$ having non-zero density is invariant across all subgraphons. In the more
general language of structures, weak randomness requires the limit $\phi\in\HomT{T}$ to be such that
for every sub-object $\psi$ of $\phi$ and every finite structure $M$, we have $\phi(M) > 0$ if and
only if $\psi(M) > 0$. We characterized (strongly) persistent families of structures, i.e., those
that correspond to a theory $T$ that has a universal weakly random limit (that is, a weakly random
$\phi$ such that $\Th(\phi)=T$) as precisely those that are closed under substructures and weakly
closed under substitutions.

We also studied a weakening $\WR$ of $\AEHP$. In the language of graphs, $\WR$ is a property of a
hereditary class of graphs which requires that every graphon associated to the class contains a
weakly random sub-graphon. We characterized $\WR$ for hereditary classes of graphs that are closed
under substitution as precisely those classes which are ``primally almost finite'', meaning that in
the partial order on elements of the class given by induced subgraph, there is no infinite antichain
of prime graphs. In the general language of structures, $\WR$ requires every limit of $T$ to contain
a weakly random sub-object (see Definitions~\ref{def:AEHP}, \ref{def:WR}
and~\ref{def:WRgeneral}). We characterized $\WR$ for theories $T$ with maximum arity at most $2$ and
$\cM[T]$ closed under substitutions as precisely the set of theories $T$ that are monochromatically
primally almost finite.

\medskip

A very natural open problem that was not addressed in this paper is to characterize weak randomness
at the level of objects, that is, to provide an equivalent property to $\phi\in\HomT{T}$ being
weakly random. Toward this goal, a natural first step is to ask how different can two weakly random
objects $\phi$ and $\psi$ be. A first source of difference is obviously that they can have different
persistence sets $P(\phi)\neq P(\psi)$. On the other hand, if $P(\phi)=P(\psi)$, then we can attempt
to measure their difference based on the sub-object partial pre-order and it is natural to ask what
is the structure of the partially pre-ordered set $\Phi_\cF\df\{\phi \mid P(\phi) = Q(\phi) = \cF\}$
for some (strongly) persistent class $\cF$. Obviously, if $\cF=\{K_n \mid n\in\NN\}$ or
$\cF=\{\overline{K}_n \mid n\in\NN\}$, then the set $\Phi_\cF$ has only one element, but even for
the next simplest case $\cF=S(\{K_0,K_2,\overline{K}_2\})$ of induced subgraphs of recursive
blow-ups of $C_4$, the structure of the partial pre-order on $\Phi_\cF$ is not clear: does it have
incomparable elements? What about infinite antichains? By Proposition~\ref{prop:recblowup}, if
$G=(G_n)_{n\in\NN}$ in which each $G_n$ is either $K_2$ or $\overline{K}_2$ and both $K_2$ and
$\overline{K}_2$ occur infinitely often, then the recursive blow-up $\phi_G$ satisfies
$\phi_G\in\Phi_\cF$ and we believe that changing the asymptotic proportion of edges and non-edges in
$G$ should produce incomparable elements of $\Phi_\cF$.

As we mentioned in the introduction, the approximate \Erdos--Hajnal property ($\AEHP$) is a
variation of the usual \Erdos--Hajnal property ($\EHP$) that allows for negligible errors, but
requires linear-sized homogeneous sets in the presence of convergence. Since $\WR$ is a weakening of
$\AEHP$, we would like to ask the following more abstract question: what is the polynomial-sized
error-free version of $\WR$ in the finite? Furthermore, since $\AEHP$ implies $\EHP$ and $\WR$ is a
larger class than $\AEHP$, is it still true that $\WR$ implies $\EHP$ for graphs? Of course, this
implication must hold if the \Erdos--Hajnal Conjecture is true. After the submission of this paper,
Nguyen--Scott--Seymour posted a preprint with a proof of the \Erdos--Hajnal conjecture for classes
of graphs with bounded VC~dimension~\cite{NSS24}. Combining their result with
Theorem~\ref{thm:primallyalmostfiniteNIP}, one concludes that hereditary classes of graphs that are
closed under substitutions and satisfy $\WR$ must necessarily satisfy $\EHP$, so it stands to reason
to attempt to remove the closure under substitutions hypothesis.

As mentioned in Section~\ref{sec:WR:univ} (see also Table~\ref{tab:corresp}), several of the proofs
on weak randomness and the class $\WR$ do not generalize very well in the presence of predicates of
arity at least $3$. It is natural to ask if we can characterize $\WR$ in these cases in the presence
of some simplifying assumption that would replace closure under substitution used in the binary
case.

As briefly mentioned before, weak randomness is a weakening of the property $\UInduce[1]$
of~\cite{CR20b}. Since $\UInduce[1]$ is part of a hierarchy of quasirandomness properties
$\UInduce[\ell]$, one might expect that there exists a hierarchy of weak randomness as well. In
turn, it may be that our difficulty in understanding $\WR$ in arity $3$ comes from the fact that
there is a wide variety of $\UInduce[1]$ limits of $3$-hypergraphs and since $\UInduce[2]$ for
$3$-hypergraphons amounts again to only (full) quasirandom $3$-hypergraphons, one might expect that
the corresponding $\WR[2]$ property in arity $3$ defined from an appropriate notion of ``weak
$2$-randomness'' (or more generally $\WR[\ell-1]$ in arity $\ell$) could be easier to handle. Since
the definition of $\UInduce[2]$ is considerably more technical than that of $\UInduce[1]$ and our
initial attempts at a weak $2$-randomness definition did not yet yield any interesting results, we
refrain from elaborating further.

Finally, in the absence of closure under substitutions, it is obvious that $\WR$ is no longer
characterized by the primally almost finite condition: obvious counter-examples include the theories
$T_{\omega\leq k}$ ($T_{\chi\leq k}$, resp.) of graphs whose clique number (chromatic number, resp.)
is at most $k$, which clearly satisfy $\AEHP$ but are not closed under substitutions when $k\geq 2$
(as $K_{k+1}\in S(\{K_2\})$ is not a model of $T_{\omega\leq k}$ or $T_{\chi\leq k}$). It is possible to
upgrade Lemma~\ref{lem:graphs:primallyalmostfinite->WR} and
Proposition~\ref{prop:monprimalmostfinite->WR} to also cover the theories $T_{\omega\leq k}$ and
$T_{\chi\leq k}$ via an interactive proof (more precisely, a two-player game in which the first
player is attempting to show that some sub-object $\psi$ must have $Q(\psi)$ monochromatically
primally almost finite and the second player is attempting to deceive the first player), but we
leave this result to a future work.

\bibliographystyle{alpha}
\bibliography{refs}

\end{document}